\documentclass[10pt,a4paper,reqno]{amsart}
\usepackage{amssymb}

\usepackage{enumerate}

\usepackage[colorlinks=true]{hyperref}
\hypersetup{citecolor=blue}
\usepackage{color}

\numberwithin{equation}{section}
\newtheorem{thm}{Theorem}[section]
\newtheorem{lem}[thm]{Lemma}
\newtheorem{prop}[thm]{Proposition}
\newtheorem{cor}[thm]{Corollary}

\theoremstyle{definition}

\newtheorem{rem}[thm]{Remark}
\newcommand\R{{\mathbb R}}
\newcommand\C{{\mathbb C}}

\newcommand\N{\mathbb{N}}

\newcommand\Tma{T_{\mathrm{max}}}

\newcommand\Step[1]{\par\medskip\noindent {\sc Step~#1.}\quad}

\newcommand\Spa{\mathcal{X}}
\newcommand\goto{\mathop{\longrightarrow}}
\newcommand\Jdtx{H}
\newcommand\CSTu{\boldsymbol {A}} 
\newcommand\CSTd{\boldsymbol {D}} 
\newcommand\CSTt{\boldsymbol {B}} 
\newcommand\CSTq{\boldsymbol {E}} 
\newcommand\CSTc{\boldsymbol {F}} 
\newcommand\CSTs{\boldsymbol {G}} 

\newcommand\Term{\boldsymbol {T}} 
\newcommand\MDb{\boldsymbol {M}} 

\newcommand\CSTnv{\Lambda }
\newcommand\CSTnw{ \widetilde{K}}

\definecolor{bostonuniversityred}{rgb}{0.8, 0.0, 0.0}

\newcommand\MScN[1]{\href{http://www.ams.org/mathscinet-getitem?mr=#1}{\nolinkurl{(#1)}}}
\newcommand\DOI[1]{\href{http://dx.doi.org/#1}{(doi: \nolinkurl{#1})}}
\newcommand\LINK[1]{\href{#1}{(link: \nolinkurl{#1})}}

\begin{document}

\title[dissipative nonlinear Schr\"odinger equation]{Asymptotic behavior for a dissipative nonlinear Schr\"odinger equation}

\author[T. Cazenave]{Thierry Cazenave$^1$}
\email{\href{mailto:thierry.cazenave@sorbonne-universite.fr}{thierry.cazenave@sorbonne-universite.fr}}

\author[Z. Han]{Zheng Han$^{2, \dag} $}
\email{\href{mailto:hanzh_0102@hznu.edu.cn}{hanzh\_0102@hznu.edu.cn}}

\author[I. Naumkin]{Ivan Naumkin$^3$}
\email{\href{mailto:ivan.naumkin@iimas.unam.mx}{ivan.naumkin@iimas.unam.mx}}

\address{$^1$Sorbonne Universit\'e \& CNRS, Laboratoire Jacques-Louis Lions,
B.C. 187, 4 place Jussieu, 75252 Paris Cedex 05, France}

\address{$^2$Department of Mathematics, Hangzhou Normal University, Hangzhou, 311121, China}

\address{$^3$Departamento de F\'{\i}sica Matem\'{a}tica, Instituto de Investigaciones
en Matem\'{a}ticas Aplicadas y en Sistemas. Universidad Nacional Aut\'{o}noma
de M\'{e}xico, Apartado Postal 20-126, Ciudad de M\'{e}xico, 01000, M\'{e}xico.}

\thanks{$^\dag$ Corresponding author}

\thanks{Zheng Han thanks NSFC 11671353,11401153, Zhejiang Provincial Natural Science Foundation of China under Grant No. LY18A010025, and CSC for their financial support}

\thanks{Ivan Naumkin is a Fellow of Sistema Nacional de Investigadores. He was partially supported by project PAPIIT IA101820}

\subjclass[2020] {Primary 35Q55; secondary 35B40}

\keywords{Nonlinear Schr\"odinger equation; dissipative nonlinearity; Asymptotic behavior}

\begin{abstract}
We consider the Schr\"odinger equation with nonlinear dissipation
\begin{equation*} 
i \partial _t u +\Delta u=\lambda|u|^{\alpha}u
\end{equation*} 
in ${\mathbb R}^N $, $N\geq1$, where $\lambda\in {\mathbb C} $ with $\Im\lambda<0$. Assuming $\frac {2} {N+2}<\alpha<\frac2N$, we give a precise description of the long-time behavior of the solutions (including decay rates in $L^2$ and $L^\infty $, and asymptotic profile), for a class of arbitrarily large initial data.
\end{abstract}

\maketitle

\section{Introduction} \label{sIntro} 
In this paper, we consider the following nonlinear Schr\"{o}dinger equation
\begin{equation}  \tag{NLS}\label{NLS-0}
\begin{cases} 
i \partial _t u +\Delta u=\lambda|u|^{\alpha}u,\\
u(0,x)=u_{0},
\end{cases} 
\end{equation} 
where $\lambda\in \C $ with 
\begin{equation} \label{Imlamdba}
\Im \lambda <0,
\end{equation}
and $0<\alpha<\frac{2}{N}$. 
This equation~is a particular case of the
complex Ginzburg-Landau equation on $\R^N$
\begin{equation*}
 \partial _t u =e^{i\theta }\Delta u+ \zeta |u|^{\alpha }u,
\end{equation*}
where $ | \theta | \leq \frac{\pi }{2}$ and $\zeta \in \C$,
 which in turn is a generic modulation equation that decribes the nonlinear 
 evolution of patterns at near-critical conditions. See for instance~\cite{SteStu,CrHo, Mielke}.

If $\alpha <\frac {4} {N}$, then equation~\eqref{NLS-0} is mass-subcritical, hence under assumption~\eqref{Imlamdba} the associated  initial value problem is globally well-posed in $L^{2}({\mathbb{R}}^{N})$ and in $H^{1}({\mathbb{R}}^{N})$. See e.g.~\cite[Proposition~2.1]{CH}. 

For the large-time behavior of the solutions, the exponent
$\alpha =\frac{2}{N}$ is critical. More precisely, if $\alpha > 
\frac{2}{N}$, $\lambda \in {\mathbb{C}}$, then for a large set of initial
values the corresponding solutions scatter as $t\rightarrow \infty $, i.e. they behave like the solutions of the free Schr\"o\-din\-ger  equation. See~\cite{Strauss2, GV2,GV1,CW,GOV, NaOz, CCDW, CN1}. On the other hand, if $\alpha \leq \frac{2}{N}$ and~\eqref{Imlamdba} holds then, due to the dissipative nature of the nonlinear term, the solutions of equation~\eqref{NLS-0} often decay faster compared
to the solutions of the free Schr\"o\-din\-ger  equation. In particular, in
the case $\alpha =\frac{2}{N}$,  a large class of solutions of \eqref{NLS-0} have the decay rate $(t\log t)^{-\frac{N}{2}}$ as $t \to \infty $, see~\cite{Shi, KitaSh1, KitaSh, CN}. It is worth noting that, as proved in \cite{CN},  the limit
\begin{equation*}
\lim  _{ t\to \infty  }(t\log t)^{\frac{N}{2}}\Vert u(t)\Vert _{L^{\infty }} = (\alpha |\Im \lambda |)^{-\frac{N}{2}}
\end{equation*}
exists and is independent of the initial value (for a certain class of solutions). 

In the case $\alpha <$ $\frac{2}{N},$ it is possible (still assuming~\eqref{Imlamdba}) to derive
strong a priori estimates for the solutions to \eqref{NLS-0} under some further
``dissipative" condition on $\lambda $.  These estimates are then used  to
describe the large-time behavior of the
solutions to \eqref{NLS-0} for $\alpha <$ $\frac{2}{N}$ sufficiently close
to the critical power $\frac{2}{N}$. More precisely, in the one-dimensional case $N=1
$, if $\alpha <2$ is sufficiently close to $2$ and under the dissipative condition
\begin{equation} \label{res-lam}
\frac{\alpha }{2\sqrt{\alpha +1}}|\Re \lambda |\leq|\Im \lambda |,
\end{equation}
the large-time asymptotic behavior of solutions is described in \cite{KitaSh} for {\it any} initial value in $ H^1 (\R ) \cap L^{2}( \R,|x|^{2}dx)$. In particular, the solutions satisfy
\begin{equation} \label{uppbd-u}
\Vert u(t)\Vert _{L^{\infty }}\lesssim t^{-\frac{1}{\alpha }}.
\end{equation}
In addition, it is proved in \cite{HN1} that for any space dimension $N\geq 1
$, under the dissipative condition~\eqref{res-lam} and for $\alpha <\frac{2}{ N}$ sufficiently close to $\frac{2}{N}$, 
{\it all} solutions with initial value in $ H^1 (\R^N ) \cap L^{2}( \R^N ,|x|^{2}dx)$ satisfy the $L^{2}$-decay estimate 
\begin{equation*}
\Vert u(t)\Vert _{L^{2}}\lesssim t^{-(\frac{1}{\alpha }-\frac{N}{2})q},
\end{equation*}
for all $q<\frac{2}{N+2}$, $q\leq \frac{1}{2}$. If there is no ``dissipative" condition on $\lambda$, it seems problematic  to derive a priori estimates for all solutions. Therefore, in order to study the large-time behavior of the solutions, further assumptions on the initial data are required. This is achieved in \cite{KitaSh1,HN2}, where the dissipative condition~\eqref{res-lam} is removed and the time decay estimates \eqref{uppbd-u}, as well as lower estimates are established, and the large-time
asymptotic behavior  of the solutions is described, in the case of space dimensions $ N=1,2,3$, for all $\alpha <\frac{2}{N}$ sufficiently close to $\frac{2}{N}$, but only for all sufficiently small initial data in a certain space. 
In the case of any space dimension $N\geq 1$, the large-time asymptotic behavior of solutions is studied in~\cite{CH}, for $\alpha <\frac{2}{N}$ sufficiently close to $\frac{ 2}{N}$ and for a class of arbitrarily large nonvanishing initial values.  (A nonvanishing condition appears due to the lack of regularity of the nonlinear term, see~\cite{CN}.) 
Note that it is proved in~\cite{CH} that, for the solutions studied there, the limit 
\begin{equation*}
\lim  _{ t\to \infty  } t^{\frac{1}{\alpha }}\Vert u(t)\Vert _{L^{\infty }} = \Bigl( \frac  {2\alpha| \Im \lambda|} {2-N\alpha} \Bigr)^{ -\frac{1}{\alpha }} 
\end{equation*}
exists and is independent of the initial value, similarly to what happens in the case $\alpha = \frac {2} {N}$. 
Again, since \eqref{res-lam} is also {\it not} used in \cite{CH}, the results are obtained only for a class of initial data (not necessarily small).

The aim of the present work is to complete the previous results on the
time-asymptotic behavior of the solutions obtained in \cite%
{CH,HN1,HN2,KitaSh}. More precisely, we do not impose any condition on $%
\lambda $ other than~\eqref{Imlamdba}, and we improve the lower conditions
on $\alpha $, replacing certain relative smallness assumptions by the
explicit condition
\begin{equation}
\frac{2}{N+2}<\alpha <\frac{2}{N}.  \label{Conditionalpha}
\end{equation}
Before stating our results, we introduce the function spaces we will use. We fix three integers $k,m,n$ sufficiently large so that
\begin{gather} 
k > \frac {N} {2} +4 , \label{fCondonk}   \\
n > \max \Bigl\{ \frac {20} {\alpha ^2},  \frac {N(2-N\alpha ) (k+4)} {\alpha },    \frac {2 N (k+2) (2-N\alpha )} {(N+2) \alpha - 2}  \Bigr\},  \label{fCondonn:1} \\
m > \max \Bigl\{ \frac { k + n +1} {2},  \frac {5 n \alpha  |\lambda | (1 + \alpha  | \Im \lambda |) } { N(2- N \alpha ) |\Im \lambda |    } \Bigr\} ,\label{fCondonm}
\end{gather} 
and we let
\begin{equation}  \label{def-J}
J=2m+2+k+n .
\end{equation}
Moreover, we define the nonincreasing function $\MDb : \{ 0, \cdots ,J\}  \to \{ 0, \cdots , n\} $ by
\begin{equation} \label{fDfnMdb} 
\MDb ( p ) = 
\begin{cases} 
n & 0 \le p \le J - n , \\
J- p  & J- n \le p \le J .
\end{cases} 
\end{equation} 
We introduce the Banach space $ \Spa $ defined in~\cite{CN, CN1}
\begin{equation} \label{n27}
\begin{split}
 \Spa =\{u\in H^{{J}}( \R^N );  &  \,\langle x\rangle^{{n}%
}D^{\beta}u\in L^{\infty}( \R^N )\text{ for }0\leq|\beta|\leq2{m} , \\
\langle x\rangle^{ \MDb (  | \beta | ) } D^{\beta}u  &  \in L^{2}( \R^N )\text{
for } 2m+ 1 \leq |\beta|\leq J\}
\end{split}
\end{equation}
and we equip $ \Spa $ with the norm
\begin{equation*} 
\Vert u\Vert_ \Spa =  \sup_{0\leq  |\beta | \leq2m} \Vert\langle \cdot \rangle^{n}D^{\beta}u\Vert_{L^{\infty}} + 
\sup_{2m+1\leq  |\beta | \leq J } \Vert\langle \cdot \rangle^{ \MDb (  | \beta | ) }D^{\beta}
u\Vert_{L^{2}} 
\end{equation*} 
where
\begin{equation*} 
\langle x\rangle=(1+|x|^{2})^{\frac{1}{2}}.
\end{equation*} 
As observed above, under assumptions~\eqref{Imlamdba} and~\eqref{Conditionalpha}, the initial value problem
associated with equation~\eqref{NLS-0} is globally well-posed in $H^{1}({\mathbb{R}}^{N})$.
Our main result is the following.

\begin{thm} \label{T1}
Let $\lambda\in \C$ with $\Im\lambda<0$, and let $\alpha $ satisfy~\eqref{Conditionalpha}.
Assume~\eqref{fCondonk}--\eqref{fCondonm} and let $ \Spa $ be defined
by~\eqref{n27}. 
Let $v_{0}  \in \Spa $ satisfy
\begin{equation} \label{boundbelow}
\inf\limits_{x\in \R^N }\langle x\rangle^{n}|v_{0} (x)|>0 .
\end{equation}
Given $b\in \R$, let $u_0 \in H^1 (\R^N ) $ be given by $u_{0}  (x)=e^{i\frac{b|x|^{2}}{4}}v_{0}(x) $ and let $u\in C([0,\infty ), H^1 (\R^N ) )$ be the corresponding solution of~\eqref{NLS-0}. 
If $b$ is sufficiently large, then $u\in  L^{\infty}((0,\infty)\times
 \R^N )\cap L^{\infty}((0,\infty),H^{1}( \R^N ))$, and 
there exist $C, \delta>0$, and $f_{0},\omega_{0}\in L^{\infty} (\R^N ) \cap C(\R^N ) $ with $\Vert
f_{0}\Vert_{L^{\infty}}\leq\frac{1}{2}$ and $\langle \cdot \rangle^{n}\omega_{0}\in
L^{\infty}( \R^N )$ such that
\begin{equation} \label{dis81}
t^{ \frac{1}{\alpha}-\frac{N}{2} }  \Vert u(t, \cdot )-z(t, \cdot ) \Vert _{L^{2}}+ t^{\frac{1}{\alpha }} \Vert
u(t, \cdot )-z(t, \cdot ) \Vert _{L^{\infty}}\leq C t^{- \delta}
\end{equation}
for $t\ge 1$, where
\begin{equation*} 
z(t,x)=(1+bt)^{-\frac{N}{2}}e^{i\Theta(t,x)}\Psi \Bigl(  t,\frac{x} {1+bt} \Bigr)  \omega_{0} \Bigl(  \frac{x}{1+bt} \Bigr)
\end{equation*} 
with
\begin{equation*} 
\Theta(t,x)=\frac{b|x|^{2}}{4 (  1+bt )  }-\frac {\Re\lambda} {\Im\lambda} \log\Psi \Bigl(  t,\frac{x}{1+bt}\Bigr)
\end{equation*} 
and
\begin{equation*} 
\Psi(t,y)=\Bigl(  \frac{1+f_{0}(y)}{1+f_{0}(y)+\frac{2 \alpha| \Im \lambda
|}{b (  2-N\alpha )  }|v_{0}(y)|^{\alpha}[ (1+bt)^{\frac{2-N\alpha} {2}}-1]} \Bigr)  ^{\frac{1}{\alpha}}.
\end{equation*} 
Moreover,
\begin{equation*} 
 \vert \omega_{0} \vert ^{\alpha}=\frac{ \vert v_{0} \vert ^{\alpha}}{1+f_{0}},
\end{equation*} 
so that $\frac{3}{2} \vert v_{0} \vert ^{\alpha}\leq  \vert \omega_{0} \vert ^{\alpha}\leq2 \vert v_{0} \vert
^{\alpha}$. In addition,
\begin{equation} \label{dis82}
t\Vert u\Vert_{L^{\infty}}^{\alpha} \goto_{t\uparrow \frac{1}{b}} \frac {2-N\alpha}{2\alpha|\Im\lambda|} 
\end{equation}
and 
\begin{equation} \label{dis83}
a \leq (1+bt)^{ (  \frac{1}{\alpha}-\frac{N}{2} ) (  1-\frac{N} {2n} )  } \Vert u (  t )  \Vert _{L^{2}} \leq
A,
\end{equation}
as $t \to \infty $,  for some constants $0<a\leq A<\infty$.
\end{thm}

\begin{rem} \label{eRem2} 
Here are some comments on Theorem~\ref{T1}.
\begin{enumerate}[{\rm (i)}] 

\item \label{eRem2:3} 
Theorem~\ref{T1} is valid in any dimension $N\geq1$, and for any $\lambda\in\C$ with $\Im  \lambda<0$.
The main restrictions are that  the initial value $u_0$ must be sufficiently smooth, bounded from below in the sense~\eqref{boundbelow}, and oscillatory in the sense that $b$ must be sufficiently large. 
On the other hand, there is no restriction on the amplitude of $u_0$.

\item \label{eRem2:1} 
A typical initial value which is admissible in Theorem~\ref{T1} is  $v_0 = c  \langle \cdot \rangle ^{-n} + \varphi $ with $c\in \C$, $c\not = 0$, and $\varphi \in {\mathcal S} (\R^N )$, $ |\varphi | \le ( |c| -\varepsilon )  \langle \cdot \rangle ^{-n} $, $\varepsilon >0$.

\item \label{eRem2:4} 
The limit~\eqref{dis82} gives the exact decay rate of $ \| u(t)\| _{ L^\infty  }$, and this limit is independent of the initial value $u_0$. 
Compare~\cite[Remark~1.2~(iv)]{CH}.  

\item 
Estimate~\eqref{dis83} shows that $ \|u(t)\|_{L^2} $ is equivalent as $t\to \infty $ to $ t^{- (\frac {1} {\alpha } - \frac {N} {2}) ( 1- \frac {N} {2n} )} $. 
In particular, we see that the decay rate of $ \|u(t)\|_{L^2} $ depends on the initial value, through the parameter $n$ which can be chosen (provided it is sufficiently large to satisfy~\eqref{fCondonn:1}).

\item \label{eRem2:5} 
Since
\begin{equation*} 
\liminf _{ t\to \infty  }  t^{\frac 1 \alpha }\|u(t ) \|_{L^\infty} >0 , \quad   t ^{\frac {1} {\alpha } - \frac {N} {2}}\|u( t ) \|_{L^2} \goto _{ t\to \infty  } \infty  ,
\end{equation*} 
by~\eqref{dis82}-\eqref{dis83}, we see  that the asymptotic behavior of $u (t)$ as $t\to \infty $ is described by the asymptotic estimate~\eqref{dis81}.

\item \label{eRem2:6} 
We do not know if the lower condition in~\eqref{Conditionalpha} on the power $\alpha$ is
necessary to derive the asymptotic expansion~\eqref{dis81}. However,
assumption~\eqref{Conditionalpha} plays a crucial role in the proof of Proposition~\ref{BD-GWP} below to control
$ \Vert \frac{\Delta v}{ \vert v \vert } \Vert _{L^{\infty} }$ (see~\eqref{dis86}), which in turn is used in the proof of Proposition~\ref{PropAsym} to prove that  $f(t)$ is convergent to $f_{0}$ in $L^{\infty }( \R^N  )$ as $t\uparrow\frac{1}{b}$.
Condition~\eqref{Conditionalpha} is also essential to establish~\eqref{fCndeta4} in the proof of Proposition~\ref{PropAsym}. 

\end{enumerate} 
\end{rem} 

The general strategy we use to prove Theorem~\ref{T1} is
inspired by \cite{CN, CN1}. In order to obtain our results, we need strong
decay and regularity of the initial data. As the nonlinearity $|u|^{\alpha}u$
can be not smooth enough, we require the nonvanishing condition
\eqref{boundbelow} as well (see \cite{CN} for a discussion on this regularity
issue). This explains the space $ \Spa $ we work with. The other main
ingredient in the strategy of \cite{CN, CN1} is the application of the
pseudo-conformal transformation, which is given by
\begin{equation} \label{Pseudo}
v(t,x)=(1-bt)^{-\frac{N}{2}}u\Bigl(  \frac{t}{1-bt},\frac{x}{1-bt} \Bigr)
e^{-i\frac{b|x|^{2}}{4(1-bt)}}, \quad t\geq0, x\in \R^N ,
\end{equation}
for any $b>0 $.  Using this transformation, we see that the equation
\eqref{NLS-0} is equivalent to the nonautonomous equation
\begin{equation}  \tag{NLS$_b$}\label{NLS-1}
\begin{cases} 
i \partial _t v+\Delta v=\lambda(1-bt)^{-\frac{4-N\alpha}{2}}|v|^{\alpha}v,\\
v(0,x)=v_{0}.
\end{cases} 
\end{equation} 
The last equation reveals the main issue that appears in the case when
$\alpha\leq\frac{2}{N}$: the factor $(1-bt)^{-\frac{4-N\alpha}{2}}$ is not
integrable at $t=1/b $.  In order to deal with this problem, in the critical
case $\alpha=\frac{2}{N}$ considered in \cite{CN1}, the solution $v$ is
estimated by allowing a certain growth of the norms appearing in $ \Spa 
$. Then, by the Duhamel's formula for \eqref{NLS-1}
\begin{equation}
v (  t )  =e^{it\Delta}v_{0}+\lambda\int_{0}^{t}(1-bs)^{-\frac
{4-N\alpha}{2}}e^{i (  t-s )  \Delta} \vert v (  s )
 \vert ^{\alpha}v (  s )  ds \label{Duhamel}%
\end{equation}
and the elementary inequality
\begin{equation*} 
\int_{0}^{t}(1-bs)^{-1-\nu}ds=\frac{1}{b\nu} (  (1-bt)^{-\nu}-1 )
\leq \frac{1}{b\nu}(1-bt)^{-\nu},
\end{equation*} 
if $e^{i (  t-s )  \Delta} \vert v (  s )   \vert ^{\alpha} (  s )  $ is estimated in a certain norm by $(1-bs)^{-\mu
} $,  the solution $v$ is controlled in the same norm by $(1-bs)^{-\mu
-\frac{2-N\alpha}{2}} $.  In the case when $\alpha=\frac{2}{N}$, $v$ is
controlled by the same power $(1-bs)^{-\mu}$ and this is used in \cite{CN1} to
close appropriate estimates. In the case $\alpha<\frac{2}{N} $, there appears an
extra singularity $(1-bs)^{-\frac{2-N\alpha}{2}}$ in the control of
$e^{i (  t-s )  \Delta} \vert v (  s )   \vert
^{\alpha}v (  s )   $.  In \cite{CH}, this problem is overcome by using
the extra decay of the solution due to dissipation. Namely,
\begin{equation} \label{behaviour}
v (  s )  \sim b\frac{2-N\alpha}{2\alpha|\Im\lambda|}(1-bs)^{-\frac
{2-N\alpha}{2 \alpha }} 
\end{equation}
as $s \to 1/b$ (see \eqref{Decay-v0} below).\ Unfortunately, the price
to pay for this extra decay is the factor $b$ in the right-hand side of
\eqref{behaviour}, which makes it impossible to obtain smallness for large $b$
in the last term in \eqref{Duhamel}, when one applies a contraction argument.
One solution to this problem is to use the factor $2-N\alpha$ in
\eqref{behaviour}, which is small if $\alpha$ is relatively close to the
critical power $\alpha=2/N $,  to close the required estimates. This is done in
\cite{CH}. In the present paper we remove this assumption and replace it by
the condition \eqref{Conditionalpha}. This requires two new ingredients. First
of all, we observe that under certain assumptions, not only the solution
itself has an extra decay. The derivatives as well enjoy this property (see
Proposition~\ref{lem-Decay-v} below). The second ingredient is to allow a very large
growth of the derivatives of the solution as $s \to 1/b$. Roughly speaking, we let the derivatives
$ \vert D^{\beta}v \vert \sim(1-bs)^{- \vert \beta \vert
\sigma} $,  when $s \to 1/b$ (in the previous works \cite{CN1,CH}, all the
derivatives of a given order $\beta$ behave as $(1-bs)^{-\sigma_{ \vert
\beta \vert }} $,  for some $\sigma_{ \vert \beta \vert }\leq1$).
Let us try to explain how we use these ingredients to establish the necessary estimates. Differentiating equation
\eqref{NLS-1} and using the condition $\Im\lambda<0$, we deduce (see
Proposition \ref{eP1} below)
\begin{equation} \label{dis85}
\begin{split} 
 \vert D^{\beta}v \vert  \leq &  \vert D^{\beta}v_{0} \vert \\
&  + \int_{0}^{t} \vert D^{\beta}v \vert ds+C\sum_{\substack{\gamma
_{1}+\gamma_{2}=\beta,\\ \vert \gamma_{1} \vert \geq1}}\int_{0} 
^{t}(1-bt)^{-\frac{4-N\alpha}{2}} \vert D^{\gamma_{1}} \vert
v \vert ^{\alpha} \vert  \vert D^{\gamma_{2}}v \vert ds.
\end{split} 
\end{equation} 
One of the key observations is that due to the dissipation $\Im\lambda<0 $,  the
term $(1-bt)^{-\frac{4-N\alpha}{2}} \vert v \vert ^{\alpha} \vert
D^{\beta}v \vert $ is absent from \eqref{dis85}. Then, as far as we can
control the derivatives $ \vert D^{\gamma_{1}} \vert v \vert
^{\alpha} \vert $ by Proposition~\ref{lem-Decay-v} below, we control the
derivatives of the solution $D^{\beta}v$ without gaining extra singularity or
loosing the smallness because of the large factor $b $.  We must stop this argument at some
stage because, as it can be observed from \eqref{eq-v-4}, the derivatives
$ \vert D^{\gamma_{1}} \vert v \vert ^{\alpha} \vert $ are
estimated in terms of derivatives of higher-order $ \vert \gamma
_{1} \vert +2$, hence a loss of two derivatives. Letting the exact
stage $ \vert \gamma_{1} \vert =M$ at which we stop this argument unknown for
a moment, we estimate the highest two derivatives by assuming that $ \vert
D^{\beta}v \vert \sim(1-bs)^{- \vert \beta \vert \sigma}$ and
using the dissipative behavior \eqref{behaviour}. Then, the derivative
$D^{\beta}\int_{0}^{t}(1-bs)^{-\frac{4-N\alpha}{2}}e^{i (  t-s )
\Delta} \vert v (  s )   \vert ^{\alpha}v (  s )
ds$ in \eqref{Duhamel} is controlled by $\frac{C}{ \vert \beta \vert
\sigma}(1-bs)^{- \vert \beta \vert \sigma} $.  Letting $ \vert
\beta \vert $ be sufficiently large, that is, letting $m$ be sufficiently large, we
obtain a small factor $\frac{C}{ \vert \beta \vert \sigma}$, which
then is used to complete the estimates on the solution $v$ of~\eqref{NLS-1}.

\begin{rem} \label{eRem1} 
Here are some comments on blowup in equation~\eqref{NLS-0}, when the condition~\eqref{Imlamdba} is not satisfied or when $\alpha \ge \frac {4} {N}$.
\begin{enumerate}[{\rm (i)}] 

\item \label{eRem1:1} 
If $\Im \lambda \ge 0$, then blowup may occur in equation~\eqref{NLS-0}.
 Indeed, if $\Im \lambda > 0$, then finite-time blowup occurs for equation~\eqref{NLS-0}, at least for $H^1$-subcritical powers $(N-2) \alpha < 4$.  See~\cite{CMZ,CMHZ}. Moreover, if $\alpha < \frac {2} {N}$, then all nontrivial solutions blow up in finite or infinite time, see~\cite{CCDW}. Finite-time blowup also occurs if $\Im \lambda =0$, $\Re \lambda <0$, and $\alpha \ge  \frac {4} {N}$, since in this case~\eqref{NLS-0}  is the standard focusing nonlinear Schr\"o\-din\-ger equation. 

\item \label{eRem1:2} 
If $\Im \lambda <0$, $\alpha > \frac {4} {N}$ and condition~\eqref{res-lam} is not satisfied, then whether or not some solutions of~\eqref{NLS-0} blow up in finite time seems to be an open question.  
\end{enumerate}  

The rest of this paper is organized as follows. 
In Section~\ref{sPrelim} we establish preliminary estimates, for the nonhomogeneous Schr\"o\-din\-ger  equation,
and for derivatives of the form $D^\beta (  | v |^\rho )$ where $v$  is a given function. 
In Section~\ref{sAPriori}, we prove a priori estimates for certain solutions of~\eqref{NLS-1}. 
These estimates are used in Section~\ref{sGlobal} to prove global existence (i.e., on the time interval $[0, \frac {1} {b})$ for certain solutions of~\eqref{NLS-1}. 
Finally, in Section~\ref{sAsymp}, we describe the  asymptotic behavior of these solutions as $t\to \frac {1} {b}$ and complete the proof of Theorem~\ref{T1}. 

\end{rem} 

\section{Preliminary estimates} \label{sPrelim} 
We begin by proving estimates for the nonhomogeneous Schr\"o\-din\-ger  equation
\begin{equation}  \label{LS}
\begin{cases} 
i \partial _t v+\Delta v=f,\\
v(0,x)=v_{0},
\end{cases} 
\end{equation} 
which are modified versions of estimates in~\cite[Proposition 2.1]{CN}. 

\begin{prop} \label{eP1}
Assume
\begin{equation} \label{def-knm}
k > \frac {N} {2} +2 , \quad n >  \frac {N} {2} +1, \quad  2 m \ge k + n +1,
\end{equation}
\eqref{def-J}, and let $ \Spa $ be defined by~\eqref{n27}.
It follows that there exists a constant $A= A (N, n, k, m)$ such that if $T>0$, $v_{0} \in  \Spa  $ and $f\in C([0,T], \Spa   )$, then for all $0\leq t\leq T $ the solution $v$ of \eqref{LS} satisfies the following estimates: if $ \vert
\beta \vert \leq2m-2$, then
\begin{equation} \label{2m-2} 
\begin{split} 
 \vert \langle x  \rangle^{n}D^{\beta}v \vert \leq & \Vert
\langle\cdot\rangle^{n}D^{\beta}v_{0}\Vert_{L^{\infty}}+\int_{0}^{t}%
\sup_{  |\gamma | \leq \vert \beta \vert +2} \Vert\langle\cdot\rangle^{n}D^{\gamma}v(s)\Vert_{L^{\infty}}ds \\
&  +\Im\int_{0}^{t}\frac{   \langle x \rangle^{2n}  D^{\beta
}f  D^{\beta}\overline{v} }{ \langle x \rangle^{n} \vert D^{\beta}v \vert }ds,
\end{split} 
\end{equation} 
for all $x\in \R^N $.  If $2m-1\leq \vert \beta \vert \leq2m $, 
then
\begin{equation} \label{2m}
\begin{split} 
 \vert \langle x  \rangle^{n}D^{\beta}v \vert \leq & \Vert
\langle\cdot\rangle^{n}D^{\beta}v_{0}\Vert_{L^{\infty}}+ 
A \int_{0}^{t}%
\sup_{  |\beta |+2 \le  |\gamma | \leq  |\beta |+k+2} \Vert\langle
\cdot\rangle^{n}D^{\gamma }v(s)\Vert_{L^{2}}ds \\
&  +\Im\int_{0}^{t}\frac{   \langle x  \rangle^{2n} D^{\beta } f  D^{\beta}\overline{v} }{ 
\langle x \rangle^{n} \vert D^{\beta}v \vert }ds, 
\end{split} 
\end{equation} 
for all $x\in \R^N $.  In the case when $ \vert
\beta \vert ={\nu}+{\mu}+2{m}+1$ with $0\leq{\nu}\leq{k}+1$ and
$0\leq{\mu}\leq{n}$, we have
\begin{equation} \label{2m+1}
\begin{split} 
\Vert\langle\cdot\rangle^{n-\mu}D^{\beta}v\Vert_{L^{2}}\leq & \Vert
\langle\cdot\rangle^{n-\mu}D^{\beta}v_{0}\Vert_{L^{2}}+  (n-\mu ) A \int_{0}^{t} 
 \Vert\langle\cdot\rangle^{n-\mu-1}\nabla D^{\beta
}v\Vert_{L^{2}}ds \\
&  +\int_{0}^{t}\frac{\Im\int_{ \R^N }   \langle x
\rangle^{2n-2\mu} D^{\beta}f  D^{\beta}\overline
{v}  dx}{\Vert\langle\cdot\rangle^{n-\mu}D^{\beta}%
v\Vert_{L^{2}}}ds.  
\end{split} 
\end{equation} 
\end{prop}

\begin{proof}
We first prove \eqref{2m-2}. Let $ \vert \beta \vert \leq2{m}$.
Applying $ \langle x \rangle ^{n}D^{\beta}$ to equation~\eqref{LS} we obtain
\begin{equation*}
i \partial _t (\langle x \rangle^{n}D^{\beta}v)=-\langle x \rangle^{n}D^{\beta
}\Delta v+\langle x \rangle^{n}D^{\beta}f. 
\end{equation*}
Multiplying by $\langle x \rangle^{n}D^{\beta}\overline{v}$ and taking the
imaginary part we deduce that
\begin{equation} \label{dis64}
 \frac{1}{2} \partial _t  (   \vert \langle x \rangle^{n}D^{\beta}v \vert
^{2} )   =-\Im (  \langle x \rangle^{2n}D^{\beta}\Delta vD^{\beta}\overline
{v} )  +\Im (  \langle x \rangle^{2n}  D^{\beta}f   D^{\beta}\overline{v}  )  . 
\end{equation} 
Integrating this last equation on $(0,t)$ with $0<t\leq T$, we deduce that
\begin{equation} \label{dis25}
\begin{split} 
 \vert \langle x \rangle^{n}D^{\beta}v \vert \leq &  \vert
\langle x \rangle^{n}D^{\beta}v_{0} \vert \\
&  +\int_{0}^{t} \vert \langle x \rangle^{n}D^{\beta} \Delta  v \vert
ds+\Im\int_{0}^{t}\frac{   \langle x \rangle^{2n}  D^{\beta }f D^{\beta}\overline{v} }{ \vert
\langle x \rangle^{n}D^{\beta}v \vert }ds.
\end{split} 
\end{equation} 
If $ \vert \beta \vert \leq2{m-2}$, then \eqref{2m-2} immediately follows from \eqref{dis25}. 
Suppose now $2{m-2} \leq \vert \beta \vert \leq2{m}$. Since $k-2 >\frac {N} {2}$ by~\eqref{def-knm}, it follows from Sobolev's embedding theorem that $  \|\langle\cdot\rangle^{n}D^{\beta}\Delta v(s) \| _{ L^\infty  } \le 	C \| \langle\cdot\rangle^{n}D^{\beta}\Delta v(s) \| _{ H^{k-2} } $ where $C$ depends on $N$ and $k$. Using Leibniz's formula together with the estimate $ |D^\gamma \langle x\rangle ^n| \le C(n,  |\gamma |) \langle x\rangle ^n$ (see~\cite[formula~(A.3)]{CN}), we deduce that
\begin{equation} \label{fCMM1} 
 \|\langle\cdot\rangle^{n}D^{\beta}\Delta v(s) \| _{ L^\infty  } \le C \sum_{  |\gamma |\le k-2 }  \| \langle \cdot \rangle^{n} D^{ \gamma + \beta} \Delta v(s) \| _{ L^2  } 
\end{equation} 
for some constant $C$ depending on $N, k, n, m$; hence~\eqref{2m} follows from~\eqref{dis25}.

Finally, suppose that $ \vert \beta \vert ={\nu}+{\mu}+2{m}+1$ with
$0\leq{\nu}\leq{k}+1$ and $0\leq{\mu}\leq{n}$. Applying~\eqref{dis64} with
$n$ replaced with $n-\mu$ and integrating in $x$ we obtain
\begin{equation} \label{dis30}
\begin{split} 
 \frac{1}{2}\frac{d}{dt}\Vert\langle\cdot\rangle^{n-\mu}D^{\beta}
v\Vert_{L^{2}}^{2}   = &  - \Im    \int_{ \R^N }  \langle x \rangle^{2n- 2 \mu } \Delta D^{\beta} vD^{\beta}\overline
{v} dx  \\ &  +\Im\int_{ \R^N }   \langle x \rangle^{2n-2\mu} D^{\beta}f D^{\beta}\overline{v}  dx.
\end{split} 
\end{equation} 
Integrating by parts the first term in the right-hand side of~\eqref{dis30}, we see that 
\begin{equation*}
 - \Im    \int_{ \R^N }  \langle x \rangle^{2n- 2 \mu } \Delta D^{\beta} vD^{\beta}\overline
{v} dx =  \Im    \int_{ \R^N } \nabla  \langle x \rangle^{2n- 2 \mu } \nabla D^{\beta} v D^{\beta}\overline
{v} dx .
\end{equation*} 
If $\mu<n$, we use the estimate $ \vert
\nabla\langle x\rangle^{2n-2\mu} \vert \leq C \langle x\rangle
^{2n-2\mu-1}$ together with Cauchy-Schwarz to obtain
\begin{equation*} 
 - \Im    \int_{ \R^N }  \langle x \rangle^{2n- 2 \mu } \Delta D^{\beta} vD^{\beta}\overline
{v} dx  \leq C  \Vert\langle\cdot\rangle^{n-\mu-1}\nabla D^{\beta}v\Vert_{L^{2}
}\Vert\langle\cdot\rangle^{n-\mu}D^{\beta}v\Vert_{L^{2}}.
\end{equation*} 
If $\mu=n$, then $\nabla\langle x\rangle^{2n-2\mu}=0$. In both cases, dividing~\eqref{dis30} by $\Vert\langle\cdot\rangle^{n-\mu}D^{\beta} v\Vert_{L^{2}}$ and integrating on $(0,t)$, we obtain \eqref{2m+1}.
\end{proof}

We now recall the local wellposedness result for~\eqref{NLS-1} in the space $\Spa$ (see~\cite[Theorem~1]{CN1} and~\cite[Proposition~4.1]{CN}).

\begin{prop} \label{LWP}
Let $\alpha >0$, assume~\eqref{def-knm}, $n > \frac {N} {2\alpha }$, and let $ \Spa $ be defined
by~\eqref{n27}.
Let $\lambda\in  \C $ and $b\geq0$. If $v_{0}\in
 \Spa $ satisfies~\eqref{boundbelow}, then there exist $0<T<\frac{1}{b}$ and a unique solution $v\in
C([0,T], \Spa )$ of ~\eqref{NLS-1} satisfying
\begin{equation}
\inf _{0\leq t\leq T}\inf _ {x\in \R^N } (\langle
x\rangle^{n}|v(t,x)|)>0. \label{inf}%
\end{equation}
Moreover, $v$ can be extended on a maximal existence interval $[0, \Tma )$
with $0< \Tma \leq\frac{1}{b}$ to a solution $v\in C([0, \Tma 
), \Spa )$ satisfying~\eqref{inf} for all $0<T< \Tma $. Furthermore, if
$ \Tma <\frac{1}{b}$, then
\begin{equation} \label{blowup}
\Vert v(t)\Vert_ \Spa + \Bigl(  \inf _{x\in \R^N 
}\langle x\rangle^{n}|v(t,x)| \Bigr)  ^{-1} \goto _{ t \uparrow \Tma } \infty. 
\end{equation}
\end{prop}

In the following section, we will have to estimate $  |D^\beta (  |v|^\rho  )| $ for $\rho \in \R$ and $ |\beta |\ge 1$ in terms of $v$ and its derivatives. This is the purpose of the following three lemmas.

\begin{lem} \label{eCase2m2} 
Assume~\eqref{def-knm}, \eqref{def-J}, and let $ \Spa $ be defined by~\eqref{n27}.
There exists a constant $C$ such that if 
 $K\ge 1$ and $v\in  \Spa $ satisfies
\begin{equation}  \label{IV}
\Vert v \Vert_ \Spa + \Bigl(  \inf _{x\in \R^N 
}\langle x\rangle^{n}|v (x)| \Bigr)  ^{-1}\leq K, 
\end{equation}
then
\begin{equation} \label{feCase2m2:1} 
\Bigl\Vert \frac{D^{\beta} v }{|v|} \Bigr\Vert _{L^{\infty}} \le C K^2
\end{equation} 
for all $ |\beta |\le 2m +2$, where the constant $C$ depends only on $N, k, n, m$. 
\end{lem} 

\begin{proof} 
We write 
\begin{equation*} 
 \frac{ | D^{\beta} v |  }{|v|} =  \frac{ \langle x\rangle ^n | D^{\beta} v|  }{ \langle x\rangle ^n  |v|}\le K  \| \langle x\rangle ^n  D^{\beta} v \| _{ L^\infty  }.
\end{equation*} 
Estimate~\eqref{feCase2m2:1} immediately follows if $ |\beta |\le 2m$. 
If $2m+1 \le  |\beta |\le 2m+2$, we use the Sobolev embedding (cf.~\eqref{fCMM1})
\begin{equation} \label{fEstSob1} 
 \|\langle\cdot\rangle^{n}D^{\beta} v \| _{ L^\infty  } \le C \sum_{  |\gamma |\le k-2 }  \| \langle \cdot \rangle^{n} D^{ \gamma + \beta} v \| _{ L^2  } \le C  \| v \|_\Spa ,
\end{equation} 
since $2m+ 1 \le  |\beta +\gamma |\le 2m +k = J-n$. This completes the proof.
\end{proof} 

\begin{lem} \label{eRemsb1} 
Given $\rho \in \R$ and $\beta $ a multi-index with $ |\beta | \ge 1$, there exists a constant $C$ such that the following estimate holds. If $U\subset \R^N $ is an open subset, $v \in C^{ |\beta |} (U, \C) $, $ v(x) \not = 0$ for all $x\in U$, then
$  |v|^\rho \in  C^{ |\beta |} ( U, \R)$ and
\begin{equation} \label{fEstDerv1} 
\frac { | D^\beta (  |v|^\rho  ) | } { |v|^\rho } \le  | \rho | \frac {| D^{\beta  } v |} { |v|} +   C \sup  _{ D } \prod _{ \ell =1 }^{  |\beta |}    \frac {| D^{\beta _\ell } v |} { |v|}  
\end{equation} 
where $D$ is the set of $( \beta _\ell ) _{ 1\le  \ell \le  |\beta | }$ where $\beta _\ell$ are multi-indices $0 \le  |\beta _\ell | \le  |\beta | -1$ such that $\sum_{ \ell=1 }^{ |\beta |} \beta _\ell =\beta $. 
\end{lem} 

\begin{proof} 
By the Fa\`a di Bruno's formula (see Corollary~2.10 in~\cite{CoSa}), $D^\beta ( \varphi (  |v|^2) )$ is a sum of terms of the form
\begin{equation*} 
\varphi ^{(\nu)} (  |v|^2 ) \prod _{ \ell =1 }^{ \nu } D^{ \gamma _\ell} ( |v|^2 ) ,
\end{equation*} 
with appropriate coefficients, where $\nu \in \{1, \cdots,  |\beta | \}$, $ | \gamma _\ell | \ge 1$ and  $  \sum _{ \ell=1 }^{ \nu }   \gamma  _\ell = \beta  $.
Applying this to $\varphi (s)= s^{\frac {\rho } {2}}$, we see that $ D^\beta (  |v|^\rho  ) $ is a sum of terms of the form
\begin{equation*} 
\Term =  |v|^{\rho -2 \nu } \prod _{ \ell =1 }^{ \nu }  D^{ \gamma _\ell} ( |v|^2 ) ,
\end{equation*} 
with appropriate coefficients, and the same relations as above on the $\gamma _\ell$. 
For a term $\Term $ as above, we let $L_1$ and $L_2$ the (possibly empty) sets of $\ell \in \{ 1, \cdots ,\nu \}$ for which $ |\gamma _\ell | =1$ and $  | \gamma _\ell | \ge 2$, respectively. If $\nu _1 = \# L_1$ and $\nu _2 = \# L_2$, then clearly
 $\nu _1 + \nu _2 = \nu$, and $\nu _1 + 2 \nu _2 \le  |\beta |$. Next, we note that $ |v|^2 = v  \overline{v} $.
Therefore, if $\ell \in L_1$, then $ | D^{ \gamma _\ell } (  | v |^2 ) | \le 2  |v| \,  |D^{\gamma _\ell } v| $; and 
if $ \ell \in L_2$, then by Leibniz's rule, $ | D^{ \gamma _\ell } (  | v |^2 ) |$ is estimated by a sum of terms of the form $  |D^{ \gamma _\ell ^1 } v | \,  |D^{ \gamma _\ell ^2 } v | $ with $\gamma _\ell ^1 + \gamma _\ell ^2 = \gamma _\ell $. 
Thus we see that $ | \Term |$ is estimated by a sum of terms of the form 
\begin{equation*} 
  |v|^{\rho -2 \nu + \nu _1} ( \prod _{ \ell \in L_1 }  | D^{ \gamma _\ell} v | )  ( \prod _{ \ell \in L_2 }  | D^{ \gamma _\ell ^1 } v | \,  | D^{ \gamma _\ell ^2 } v | ) ,
\end{equation*} 
which we rewrite, since $2\nu = 2 \nu _1 + 2\nu _2$, in the form
\begin{equation*} 
  |v|^{\rho  } \Bigl( \prod _{ \ell \in L_1 }  \frac {| D^{ \gamma _\ell} v |} { |v|} \Bigr)  \Bigl( \prod _{ \ell \in L_2 }  \frac {| D^{ \gamma _\ell ^1 } v |} {  |v| } \,  \frac {| D^{ \gamma _\ell ^2 } v |} {  |v|} \Bigr) .
\end{equation*} 
Therefore, using $\nu _1 + 2 \nu _2 \le  |\beta |$, we see that $  |D^\beta (  |v|^\rho  )| $ is estimated by a sum of terms of the form
\begin{equation*} 
 |v|^{\rho } \prod _{ \ell =1 }^{  |\beta |}    \frac { | D^{\beta _\ell } v | } {  |v| } ,
\end{equation*} 
with appropriate coefficients, where possibly $ |\beta _\ell | = 0$ and $  \sum _{ \ell=1 }^{  |\beta | }   \beta _\ell = \beta  $. 
Finally, we notice that the only term where derivatives of order $ |\beta |$ appear in the development of $\frac { | D^\beta (  |v|^\rho  ) | } { |v|^\rho }$ corresponds to $\nu =1$ and is given by
\begin{equation*} 
\frac {\rho } {2}   |v|^{-2} ( vD^\beta  \overline{v} +  \overline{v} D^\beta v ),
\end{equation*} 
which yields the first term in~\eqref{fEstDerv1}. 
Hence~\eqref{fEstDerv1} is proved. 
\end{proof} 

\begin{lem} \label{elemu1} 
Assume~\eqref{def-knm}, \eqref{def-J}, and let $ \Spa $ be defined by~\eqref{n27}.
Given $\rho >0$, there exists a constant $C$ such that the following inequalities hold. 
If $K\ge 1$ and $v\in  \Spa $ satisfies~\eqref{IV}, then 
\begin{equation} \label{elemu1:1}
 \Bigl\| \frac {D^\beta (  |v|^\rho  )} {  |v|^\rho } \Bigr\| _{ L^\infty  } \le C K^{2  |\beta |} ,
\end{equation} 
and
\begin{equation}  \label{elemu1:2}
 \| \langle \cdot \rangle ^{n\rho  } D^\beta (  |v|^\rho   ) \| _{ L^\infty  } \le C K^{ \rho + 2  |\beta |} , \
\end{equation} 
for $  | \beta | \le 2m +2 $. 
Moreover, if 
\begin{equation} \label{elemu1:3:b1}
\eta < n\rho - \frac {N} {2}, \quad \eta \le - n + \MDb (  |\beta | ) + n\rho ,
\end{equation} 
 then for $C$ possibly larger, 
\begin{equation} \label{elemu1:3}
  \| \langle \cdot \rangle ^{ \eta } D^\beta (  |v|^\rho  ) \| _{ L^2  } \le C K^{ \rho + 2  |\beta |} 
\end{equation} 
for all $  |\beta | \le J$,  all $K\ge 1$ and all $v\in \Spa$ satisfying~\eqref{IV}, where $  \MDb (  | \beta | ) $ is defined by~\eqref{fDfnMdb}. 
\end{lem} 

\begin{proof} 
Estimates~\eqref{elemu1:1} and~\eqref{elemu1:2} are immediate consequences of~\eqref{fEstDerv1} and~\eqref{feCase2m2:1}. 

Suppose now~\eqref{elemu1:3:b1}. In particular $ \langle \cdot \rangle ^{ \eta - n\rho } \in L^2 (\R^N ) $; and so $  \| \langle \cdot \rangle ^\eta D^\beta (  |v|^\rho  ) \| _{ L^2  } \le C  \| \langle \cdot \rangle ^{n\rho  } D^\beta (  |v|^\rho   ) \| _{ L^\infty  } $. Hence~\eqref{elemu1:3} in the case $ |\beta |\le 2m$ follows from~\eqref{elemu1:2}.  
For $ |\beta |\ge 2m+1$, we argue as follows. By~\eqref{fEstDerv1}, $  | D^\beta (  |v|^\rho  ) | $ is estimated by a sum of terms of the form
\begin{equation*} 
\Term =  |v|^\rho  \prod _{ \ell =1 }^{  |\beta |}    \frac {| D^{\beta _\ell } v |} { |v|}  
\end{equation*} 
where  $\sum_{ \ell=1 }^{ |\beta |} \beta _\ell =\beta $. 
If all the $\beta _\ell$ satisfy $ |\beta _\ell | \le 2m$, then we can argue as above, and we obtain $   \| \langle \cdot \rangle ^{ \eta } \Term \| _{ L^2  } \le C K^{ \rho + 2  |\beta |} $. 
Suppose now one of the derivatives in $\Term $ has order $\ge 2m+1$, for instance $ |\beta _1|\ge 2m+1$. Then $ |\beta _\ell |\le 2m$ for all $2\le \ell \le  |\beta |$. Indeed, $\sum_{ \ell=1 }^{ |\beta |}  |\beta _\ell| = |\beta |$;  and 
for $\ell \ge 2$, 
\begin{equation*} 
 |\beta _\ell |\le  |\beta |-  |\beta _1| \le  |\beta |- 2m -1 \le J- 2m-1 \le  2m ,
\end{equation*} 
by~\eqref{def-J} and the last inequality in~\eqref{def-knm}. Therefore, we obtain
\begin{equation*} 
\Term \le K^{\alpha + 2  |\beta |-1 } \langle x\rangle ^{- n\rho } \langle x\rangle ^n | D^{\beta _1} v | ,
\end{equation*} 
Using now the second inequality in~\eqref{elemu1:3:b1}, we deduce that
\begin{equation*} 
\langle x\rangle ^\eta  \Term \le \langle x\rangle ^{- n + \MDb (  |\beta | ) + n\rho}  \Term \le K^{\alpha + 2  |\beta |-1 } \langle x\rangle ^{ \MDb (  |\beta | ) } | D^{\beta _1} v | .
\end{equation*} 
Now $ | \beta _1| \le  |\beta |$, so that $\MDb (  |\beta |) \le \MDb (  |\beta _1|) $; hence  
\begin{equation*} 
  \| \langle \cdot \rangle ^{\eta  } \Term \| _{ L^2 } \le  K^{ \rho + 2  |\beta |- 1} \| \langle \cdot \rangle ^{ \MDb (  | \beta_1 | )}D^{\beta _1} v \| _{ L^2  }  \le  K^{ \rho + 2  |\beta |} .  
\end{equation*} 
  This completes the proof.
\end{proof} 

\section{a priori estimates for~\eqref{NLS-1}} \label{sAPriori} 
In this section, we prove a priori estimates for certain solutions of~\eqref{NLS-1}. These estimates are an essential ingredient in the proof of our main theorem. 
We assume~\eqref{Conditionalpha} and~\eqref{fCondonk}--\eqref{fCondonm}. 
Since $\alpha <\frac {2} {N}$ by~\eqref{Conditionalpha}, it follows from the first inequality in~\eqref{fCondonn:1} that
\begin{equation} \label{fCondonn:2}
n > \max \Bigl\{ \frac {5N} {2} ,  \frac { 2 N} {\alpha } \Bigr\}. 
\end{equation}  
By the second inequality in~\eqref{fCondonn:2}, $\frac {N (2-N\alpha) } { n\alpha } < \frac{ 2 - N\alpha} {2}$; by the first inequality in~\eqref{fCondonn:1}, $\frac {N (2-N\alpha) } {n\alpha } <  \frac{N\alpha}{10}$; by the second inequality in~\eqref{fCondonn:1}, $ \frac {N(2-N\alpha )} { n\alpha } \le \frac {1} {k+4}$; and by the last inequality in~\eqref{fCondonn:1}, $\frac {N (2-N\alpha) } {n\alpha } <  \frac {(N+2) \alpha -2 } {2\alpha (k+2)}$.
Therefore, we may fix $\sigma $ satisfying
\begin{equation} \label{fSupplN2:b1} 
\frac {N (2-N\alpha) } {n\alpha } <  \sigma<\min \Bigl\{ \frac {N\alpha} {10} , \frac{ 2 - N\alpha} {2},  \frac {1} {k+4} ,   \frac {(N+2) \alpha -2 } {2\alpha (k+2)}\Bigr\} .
\end{equation} 
In particular, it follows from the first inequality in~\eqref{fSupplN2:b1} that 
\begin{equation} \label{fSupplN3:b3} 
\frac {n\alpha \sigma } {2- N \alpha } > N . 
\end{equation} 
Moreover, it follows from the third inequality in~\eqref{fSupplN2:b1} that 
\begin{equation} \label{fSupplN3} 
0 < 1 - \frac {2\sigma } {2- N\alpha } < 1 ,
\end{equation} 
and from the second inequality in~\eqref{fSupplN2:b1} that
\begin{equation} \label{fSupplN3:b1} 
1 - \frac   {2- N\alpha } 2 - 5 \sigma >0.
\end{equation} 
Also, it follows from the second inequality in~\eqref{fCondonm} and~\eqref{fSupplN3:b3} that
\begin{equation} \label{fSupplN1:b2} 
m >   \frac {5  |\lambda | (1 + \alpha  | \Im \lambda |) } {  |\Im \lambda | \sigma }  .
\end{equation} 
Next, we  introduce the following notation. Let
\begin{equation} \label{de-sigma}
\sigma _j =
\begin{cases} 
j\sigma & 0\le j\le 2m \\
(j+1)\sigma & j=2m+1,\\
(j+2)\sigma & 2m+2\leq j\leq J-2,\\
(j+3)\sigma & j=J-1,\\
(j+4)\sigma & j=J .
\end{cases} 
\end{equation} 
In particular, if $ j \ge J-2$, then $\sigma _j \ge j \sigma \ge 2m \sigma $. 
Using~\eqref{fSupplN1:b2},
we deduce that
\begin{equation} \label{de-sigma:b1}
 \frac {10  |\lambda | (1 + \alpha | \Im \lambda |) } {  |\Im \lambda | \sigma  _j } \le 1 , \quad J-2\le j\le J.
\end{equation} 
Given $ \ell \in \N $, we set
\begin{align}
\Vert v\Vert_{1, \ell }& =\sup\limits_{0\leq|\beta|\leq  \ell }\Vert\langle \cdot \rangle
^{n}D^{\beta}v\Vert_{L^{\infty}},\  &  \text{ if}\ 0\leq  \ell \leq2m,\label{1l}\\
\Vert v\Vert_{2, \ell }& =\sup\limits_{2m+1\leq|\beta|\leq  \ell }\Vert\langle
\cdot \rangle^{\MDb (  |\beta | ) }D^{\beta}v\Vert_{L^{2}},\  &  \text{ if}\ 2m+1\leq  \ell \leq
J   ,\label{2l}
\end{align}
where $  \MDb ( \cdot  ) $ is defined by~\eqref{fDfnMdb}.
Let $0<T\le \frac{1}{b}$ and $v\in C([0,T), \Spa )$ satisfy
\begin{equation} \label{inf:b1}
\inf _{0\leq s\leq t}\inf _ {x\in \R^N } (\langle x\rangle^{n}|v( s ,x)|)>0,\quad  \text{for all}\quad 0\le t<T. 
\end{equation}
Given $0\le t < T$, we define
\begin{align} 
\Phi_{1, t } & =\sup _ {0\leq s\leq t }\sup _{0\leq j\leq2m}  (1-b s)^{\sigma_{j}}\Vert v (s) \Vert_{1,j}, \label{fDfnPhi1} \\
\Phi_{2, t} & =\sup _{0\leq s\leq t}\sup _{2m+1\leq j\leq J }(1-b s)^{\sigma_{j}}\Vert v (s) \Vert_{2,j}, \label{fDfnPhi2} \\
\Phi_{3, t} & =\sup_{0\leq s\leq t}\frac{ (1-bs)^{\frac{2-N\alpha}{2\alpha} }}{\inf _{x\in \R^N }\langle x\rangle^{n}|v(s,x)|}\label{fDfnPhi3} \\
\Phi_{4, t} &= \sup _{0\leq s\leq t }\sup_{j\leq2m +2 } \Bigl(  (1-bs)^{\sigma
_{j}}\sup_{ \vert \beta \vert =j} \Bigl\Vert \frac{D^{\beta}\,v (s) } 
{|v (s) |} \Bigr\Vert _{L^{\infty}} \Bigr)  ,\label{fDfnPhi6} 
\end{align} 
and we set
\begin{equation} \label{fDfnPhi7} 
\Phi_{t}=\max\{\Phi_{1,t},\Phi_{2,t} \}
\end{equation} 
and
\begin{equation} \label{fDfnPhi8} 
\Psi_{t}=\max\{\Phi_{t},\Phi_{3,t},\Phi_{4,t} \}.
\end{equation} 
Note that the norms in the definition of $\Phi_{4,t}$ are finite by~\eqref{feCase2m2:1}.

\begin{lem} \label{eContPsi} 
Assume~\eqref{Conditionalpha}, \eqref{fCondonk}--\eqref{fCondonm} and~\eqref{fSupplN2:b1}, and let $\Spa $ be defined by~\eqref{n27}. 
Let $v_{0}\in \Spa $ satisfy
\begin{equation}  \label{IV:b1}
\Vert v_0 \Vert_ \Spa + \Bigl(  \inf\limits_{x\in \R^N  }\langle x\rangle^{n}|v_0 (x)| \Bigr)  ^{-1}\leq K, 
\end{equation}
 for some $K\ge 1$, let $0<T \le  \frac {1} {b}$, let $v\in C([0,T), \Spa )$ satisfy~\eqref{inf:b1} and $v(0)= v_0$. With the notation~\eqref{1l}--\eqref{fDfnPhi8}, it follows that $\Psi _t$ is a continuous function of $t\in [0,T)$, and that 
\begin{equation} \label{feContPsi} 
\Psi _0 \le K + \CSTnv  K^2, 
\end{equation} 
where the constant $\CSTnv$ depends only on $N, k, n, m$. 
\end{lem} 

\begin{proof} 
Since $  \| \cdot  \|  _{ j, \ell } \le  \| \cdot  \|_\Spa$, it follows that $\Phi  _{ j, t }$ is a continuous function of $t$ for $j=1, 2$ and $\Phi  _{ j,0 }\le  \| v_0  \|_\Spa \le K$. 
From~\eqref{inf:b1} and $v\in C([0,T), \Spa )$, it follows easily that $\Phi  _{ 3,t }$ is also a continuous function of $t$ and that $\Phi  _{ 3,0 } =  (  \inf _{x\in \R^N  }\langle x\rangle^{n}|v_0 (x)| )  ^{-1}\leq K$. 
For $\Phi  _{ 4,t }$ we write
\begin{equation*} 
\frac {D^\beta v} { |v|}= \frac {\langle x\rangle ^n D^\beta v} { \langle x\rangle ^n  |v|} .
\end{equation*} 
We observe that $ ( \langle x\rangle ^n  |v|)^{-1} $ is continuous $[0,T) \to L^\infty (\R^N ) $ (by~\eqref{inf:b1} and $v\in C([0,T], \Spa )$), and that $ \langle x\rangle ^n D^\beta v $ is continuous $[0,T) \to L^\infty (\R^N ) $ (by definition of $\Spa$ if $ |\beta |\le 2m$ and by~\eqref{fEstSob1} if $2m+1\le  |\beta |\le 2m+2$). Hence $\Phi  _{ 4,t }$ is also a continuous function of $t$.
Applying~\eqref{feCase2m2:1}, we see that $\Phi  _{ 4,0 }\le \CSTnv K^2$ where the constant $\CSTnv$ depends only on $N, k, n, m$.
\end{proof} 

The main result of this section is the following.

\begin{prop} \label{lem-Decay-v} 
Let $\lambda \in \C$ satisfy $\Im\lambda<0$. 
Assume~\eqref{Conditionalpha}, \eqref{fCondonk}--\eqref{fCondonm} and~\eqref{fSupplN2:b1}, and let $\Spa $ be defined by~\eqref{n27}. 
Let $b>0$, $K >1$, let $v_{0}\in \Spa $ satisfy~\eqref{IV:b1}, let $v\in C([0, \Tma ), \Spa )$ be the solution of~\eqref{NLS-1}
given by Proposition~$\ref{LWP}$, and let $\Psi $ be defined by~\eqref{1l}--\eqref{fDfnPhi8}. Given any $K_1 \geq K$, there exists $b_{0} > 1$ (which depends on $v$ through $K$ and $K_1$ only) such
that if $b\ge b_0$ and 
\begin{equation} \label{PSI}
\Psi_{T}\leq K_{1} 
\end{equation}
for some $0 < T < \Tma$, then
\begin{equation} \label{Decay-vmu}
 |v (t, x)| \le 2  |v_0 (x)|
\end{equation} 
and
\begin{equation} \label{Decay-v0}
|v(t,x)|^{\alpha}\leq \Bigl(  1+\frac{2-N\alpha}{2\alpha|\Im
\lambda|}\Bigr)  \min \{  2 K^\alpha  \langle x \rangle ^{-n\alpha },  bG (t) \}  
\end{equation}
on $[0,T] \times \R^N $, where
\begin{equation}  \label{defG}
G=G (  t;b,\alpha,N )  =\frac{(1-bt)^{\frac{2-N\alpha}{2}} 
}{1-(1-bt)^{\frac{2-N\alpha}{2}}}  = \frac {1} { (1-bt)^{ - \frac {2-N\alpha } {2} } -1} .
\end{equation} 
Moreover, there is a constant $C_{0}  >0$ (which depends on $v$ through $K$ and $K_1$ only and is independent of $b$ and $T$) such that if $b\geq b_{0}$, then for $0\le t\le T$,
\begin{equation} \label{est-v-alpha-L2} 
 \Vert|v (t) |^\alpha \Vert_{L^2}  \leq C_0 \min\{1,(bG (t) )^{1-\frac{2\sigma}{2-N\alpha}}\}, 
\end{equation} 
and
\begin{equation} \label{Decay-v}
\Vert \langle \cdot \rangle ^{\frac {2n\alpha \sigma } {2-N\alpha }}   D^{\beta} ( |v(t )|^{\alpha} ) \Vert_{L^{\infty}}   \leq C_{0}    \min \{  1, (  bG (t) ) ^{1-\frac{2\sigma}{2-N\alpha}} \} (1-bt)^{- (   \vert \beta \vert
-1 )  \sigma} , 
\end{equation} 
for all $1\leq \vert \beta \vert \leq 2m$; and
\begin{equation} \label{Decay-v2}
\begin{split} 
\Vert \langle \cdot  \rangle ^ { \MDb (  | \beta | )} & v(t) D^{\beta} ( |v(t )|^{\alpha} ) 
\Vert_{L^2 }   \\ & \leq C_0  (1-bt)^{-  ( |\beta | -1 ) \sigma  }  \min \{  1, (  bG (t) )
^{1-\frac{2\sigma}{2-N\alpha}} \}   , 
\end{split} 
\end{equation}
for all $2m+1 \le  |\beta | \leq J-2$.
\end{prop}

\begin{cor} \label{cor-Decay-v}
Under the assumptions of Proposition~$\ref{lem-Decay-v}$, it follows that 
\begin{equation} \label{Decay-v1}
\begin{split} 
\Vert \langle \cdot  \rangle ^ { \MDb (  | \beta | )}  D^{\beta} & ( |v(t )|^{\alpha} ) D^\gamma v (t)
\Vert_{L^2 }  \\ & \leq C_0  \min \{  1, (  bG (t) )
^{1-\frac{2\sigma}{2-N\alpha}} \}  (1-bt)^{ - (  |\beta | + |\gamma | -1 ) \sigma } , 
\end{split} 
\end{equation}
for all $0\le t\le T$, $2m+1 \le  |\beta | \leq J-2$ and $0\le  |\gamma |\le 2m +2$.
\end{cor} 

\begin{proof} 
We have
\begin{equation*} 
\Vert \langle \cdot  \rangle ^ { \MDb (  | \beta | )}  D^{\beta} ( |v(t )|^{\alpha} ) D^\gamma v (t) \Vert_{L^2 } \le 
\Vert \langle \cdot  \rangle ^ { \MDb (  | \beta | )}  D^{\beta} ( |v(t )|^{\alpha} )  v (t) \Vert_{L^2 }  \Bigl\| \frac {D^\gamma v (t) } { |v(t)|} \Bigr\| _{ L^\infty  }.
\end{equation*} 
Since
\begin{equation*} 
 \Bigl\| \frac {D^\gamma v (t) } { |v(t)|} \Bigr\| _{ L^\infty  } \le K_1  (1-bt)^{ -  |\gamma |\sigma },
\end{equation*} 
by~\eqref{PSI} and~\eqref{fDfnPhi6}, estimate~\eqref{Decay-v1} follows by applying~\eqref{Decay-v2}.  
\end{proof} 

\begin{proof} [Proof of Proposition~$\ref{lem-Decay-v}$]
In the estimates that follow, we denote by $C_0 >0$
a constant depending possibly on $\beta,\alpha,N,K,K_{1},\lambda$, etc, but not on $b$, $v$, $T$ and $\Tma$, whose exact
value is irrelevant and can change from line to line. 
We consider $b\ge 1$ and we proceed in several steps.

\Step1 Proof of~\eqref{Decay-v0}.\quad  
From equation
\eqref{NLS-1} it follows that 
\begin{equation} \label{eq-v-1}
\partial _t |v|=L+\Im\lambda(1-bt)^{-\frac{4-N\alpha}{2}}|v|^{\alpha+1},
\end{equation}
where 
\begin{equation} \label{eq-v-1:b1}
L(t,x)=- \frac{\Im(\bar{v} (t,x) \Delta v (t, x) )}{|v (t, x)|} .
\end{equation}
Multiplying both sides
of~\eqref{eq-v-1} by $|v|^{-\alpha-1}$ we see that
\begin{equation} \label{eq-v-2}
-\frac{1}{\alpha}\frac{\partial}{\partial t}|v|^{-\alpha}=|v|^{-\alpha-1}%
L+\Im\lambda(1-bt)^{-\frac{4-N\alpha}{2}}. 
\end{equation}
Let $0<t\leq T $.  Integrating~\eqref{eq-v-2} in $t$ we obtain 
\begin{equation} \label{eq-v-3}
\begin{split} 
\frac{1}{|v(t,x)|^{\alpha}}    =& \frac{1}{|v_{0}(x)|^{\alpha}}+\frac
{2\alpha| \Im\lambda|}{b (  2-N\alpha)  }[(1-bt)^{-\frac{2-N\alpha
}{2}}-1]\\
&  -\alpha\int_{0}^{t}|v(s,x)|^{-\alpha-1}L(s,x)\,ds.
\end{split} 
\end{equation} 
it follows that
\begin{equation}  \label{eq-v-4}
|v(t,x)|^{\alpha}= \frac { |v_0 (x)|^\alpha } {\Jdtx (t,x) }  ,
\end{equation}
where
\begin{equation} \label{eq-v-4:b3}
\Jdtx(t,x) = 1+f(t,x)+\frac{2\alpha
|\Im\lambda|}{b G(t) (  2-N\alpha )  }|v_{0} (x) |^{\alpha}  
\end{equation}  
with 
\begin{equation}  \label{eq-v-4:b2}
f(t,x)=-\alpha|v_{0}(x)|^{\alpha}\int_{0}^{t}|v(s,x)|^{-\alpha-1}L(s,x)\,ds.
\end{equation} 
By~\eqref{PSI} and~\eqref{fDfnPhi3} we have
\begin{equation} \label{dis68}
\frac{1}{\langle x\rangle^{\alpha n}|v(t,x)|^{\alpha}}\leq K_{1}^{\alpha
}(1-bt)^{-\frac{2-N\alpha}{2}}
\end{equation}
and by~\eqref{PSI}, \eqref{fDfnPhi6} and~\eqref{de-sigma},
\begin{equation}  \label{dis68:b1}
\frac{ \vert \Delta v(t,x) \vert }{ \vert v(t,x) \vert }\leq K_{1}(1-bt)^{-2\sigma}.
\end{equation} 
Using~\eqref{IV:b1}, \eqref{dis68}, \eqref{eq-v-1:b1} and~\eqref{dis68:b1} we obtain
\begin{equation} \label{fEstintf} 
\begin{split} 
|v_{0}(x)|^{\alpha}|v( s ,x)|^{-\alpha-1} \vert L( s ,x) \vert  &
= (  \langle x\rangle^{n}|v_{0}| )  ^{\alpha} (  \langle
x\rangle^{n}|v( s ,x)| )  ^{-\alpha}\frac{ \vert L( s ,x) \vert
}{ \vert v( s ,x) \vert }\\
&  \leq K^\alpha  K_{1}^{\alpha+1}(1-b s )^{-\frac{2-N\alpha}{2}-2\sigma} \\
&  \leq  K_{1}^{ 2 \alpha+1}(1-b s )^{-\frac{2-N\alpha}{2}-2\sigma},
\end{split} 
\end{equation} 
on $[0,T] \times \R^N $. 
Note that by the second inequality in~\eqref{fSupplN2:b1}, 
\begin{equation} \label{fdonaod} 
\frac {2-N\alpha } {2} + 2\sigma < 1, 
\end{equation} 
so that 
\begin{equation} \label{bd-f}
|f(t,x)|  \le \frac{\alpha K_{1}^{2\alpha+ 1 }}{b (  1-\frac{2-N\alpha}{2} 
-2\sigma )  }. 
\end{equation} 
We choose $b_{0} \ge 1$ sufficiently large so that
\begin{equation} \label{bd-f-1}
\frac{\alpha K_{1}^{2\alpha+ 1 }}{b_0  (  1-\frac{2-N\alpha}{2}-2\sigma )
}\leq \min \Bigl\{ \frac{1}{4}, \frac {2^\alpha -1} {2^\alpha +1} \Bigr\}, 
\end{equation}
and we deduce that if $b\ge b_0$, then
\begin{equation} \label{bd-f-1:b1}
 | f(t,x) | \le \min \Bigl\{ \frac{1}{4}, \frac {2^\alpha -1} {2^\alpha +1} \Bigr\}
\end{equation} 
on $[0, T] \times \R^N $. In particular, $1+f(t,x)\geq\frac{1}{2}$, so that 
\begin{equation} \label{bis12}
\begin{split} 
\displaystyle   \frac{ 1  }  {\Jdtx (t, x) } & \le \frac {2} {1 + \frac {4\alpha  |\Im \lambda |} {b G(t) (2-N\alpha )}  | v_0 (x) |^\alpha }
   \\  & \le  \min  \Bigl\{ 2, \frac {2-N \alpha } {2 \alpha
|\Im\lambda| \,  | v_0 (x)|^\alpha } bG (t) \Bigr\} .
\end{split} 
\end{equation} 
Applying~\eqref{eq-v-4}, \eqref{bis12}  and using~\eqref{IV:b1} we obtain
\begin{equation} \label{bis13}
|v(t,x)|^{\alpha}   \le 
 \min \Bigl\{ 2 K^{\alpha}\langle x \rangle ^{-n\alpha}
,\frac{ 2-N\alpha }{2 \alpha|\Im\lambda|} b G (t) \Bigr\}  , 
\end{equation} 
from which estimate~\eqref{Decay-v0} follows. 

\Step2 Proof of~\eqref{est-v-alpha-L2}. \quad 
Since $ \langle x\rangle ^{- n\alpha } \in L^2 (\R^N ) $ by the second inequality in~\eqref{fCondonn:2},  it follows from the first inequality in~\eqref{Decay-v0} that $ \Vert|v|^\alpha \Vert_{L^2}\leq C_0$. 
Moreover, it follows from~\eqref{Decay-v0} and $\frac {2\sigma } {2- N\alpha } \ge \frac {N} {2n\alpha }$ (by~\eqref{fSupplN3:b3}) that if $bG(t) \le 1$, then
\begin{equation*} 
\begin{split} 
 \Vert|v|^\alpha \Vert_{L^2} & \leq C_0 (\|\langle \cdot \rangle^{-n\alpha}\|_{L^2(  \langle x\rangle > (bG(t))^{- \frac {1} {n\alpha }}  )}+\|bG\|_{L^2(  \langle x\rangle < (bG(t))^{- \frac {1} {n\alpha }} ) } ) \\ &
\leq C_0 (bG (t) )^{1-\frac{N}{2n\alpha}}\leq C_0 (bG (t) )^{1-\frac{2\sigma}{2-N\alpha}},
\end{split} 
\end{equation*} 
hence~\eqref{est-v-alpha-L2} is proved.

\Step3 Further estimates of $v$ and $\Jdtx $ and proof of~\eqref{Decay-vmu}. \quad 
For any $x\in \R^N$ and $0\leq s\leq t\leq T$, 
it follows from~\eqref{eq-v-4} and $G(t)\leq G(s)$  that
\begin{equation*}
\Bigl\vert \frac {v(t,x)} {v(s,x)} \Bigr\vert ^\alpha = \frac{1+f(s,x)+\frac{2\alpha
|\Im\lambda||v_{0}|^{\alpha}}{ (  2-N\alpha )bG(s)  }}{1+f(t,x)+\frac{2\alpha
|\Im\lambda||v_{0}|^{\alpha}}{ (  2-N\alpha )bG(t)  }}
\le \frac{1+f(s,x)+\frac{2\alpha
|\Im\lambda||v_{0}|^{\alpha}}{ (  2-N\alpha )bG(s)  }}{1+f(t,x)+\frac{2\alpha
|\Im\lambda||v_{0}|^{\alpha}}{ (  2-N\alpha )bG( s )  }}.
\end{equation*}
Using~\eqref{bd-f-1:b1}, we deduce that
\begin{equation}\label{com-v-t-s}
\Bigl\vert \frac{v(t,x)}{v(s,x)} \Bigr\vert ^\alpha \leq \min \{ 2, 2^\alpha \}.
\end{equation}
This proves~\eqref{Decay-vmu}. 
Moreover, if $t\ge 0$ and $0 \le \nu \le 1$, then $\min \{ 1, t \}\le  \min \{ 1, t^\nu \}$. Thus it follows from~\eqref{bis12} and~\eqref{fSupplN3} that 
\begin{equation*}
  \frac{ 1  }  {\Jdtx (t, x) }  \le  C  \min \Bigl\{ 1 ,  \Bigl(  \frac {bG (t) } {  |v_0 |^\alpha } \Bigr)^{  1 - \frac {2\sigma } {2- N\alpha }  }  \Bigr\} .
\end{equation*} 
Since 
\begin{equation*} 
 \Bigl(  \frac {bG (t) } {  |v_0 |^\alpha } \Bigr)^{  1 - \frac {2\sigma } {2- N\alpha }  }  \le K^\alpha \langle x\rangle ^{n\alpha (  1 - \frac {2\sigma } {2- N\alpha } )} (bG(t))^{  1 - \frac {2\sigma } {2- N\alpha }}, 
\end{equation*} 
it follows that
\begin{equation} \label{fSupplN4}
  \frac{ 1  }  {\Jdtx (t, x) }  \le  C K^\alpha \langle x\rangle ^{n\alpha (  1 - \frac {2\sigma } {2- N\alpha } )} \min\{ 1, (bG(t))^{  1 - \frac {2\sigma } {2- N\alpha }} \} .
\end{equation} 

\Step4 We prove that if $ |\gamma |\le 2m +2 $ then 
\begin{equation} \label{fSupplN5} 
 \Bigl\| \langle \cdot \rangle ^{\frac {2n\alpha \sigma } {2-N\alpha }} \frac{D^\gamma  ( |v_{0}  |^{\alpha}) }{ \Jdtx (t, \cdot ) } \Bigr\| _{ L^\infty  }  \le  C_0 \min \{ 1,  (   bG (t)  )^{  1 - \frac {2\sigma } {2- N\alpha }  }  \} ,
\end{equation} 
for all $0\le t\le T$, and if $2m+1 \le  |\gamma |\le J-2$, then 
\begin{equation} \label{fSupplN8} 
 \Bigl\| \langle \cdot \rangle ^{ \MDb (  | \gamma | )}  v (t) \frac{D^\gamma  ( |v_{0}  |^{\alpha}) }{ \Jdtx (t, \cdot ) } \Bigr\| _{ L^2 } \le 
 C_0  \min\{ 1, (bG(t))^{  1 - \frac {2\sigma } {2- N\alpha }} \}  ,
\end{equation}  
for all $0\le t\le T$. 
Indeed, suppose first $ |\gamma |\le 2m + 2$. It follows from~\eqref{elemu1:2} that
\begin{equation} \label{fIntrm1:b2:2} 
 | D^\gamma (  | v_0 |^\alpha ) | \le  C \langle x \rangle ^{-\alpha n} K^{ \alpha + 2  |\gamma |}  .
\end{equation} 
Estimates~\eqref{fIntrm1:b2:2} and~\eqref{fSupplN4}  yield~\eqref{fSupplN5}.
Suppose now $  |\gamma  | \ge 2m +1$. We have 
\begin{equation} \label{fSupplN6} 
 \Bigl\| \langle \cdot \rangle ^{ \MDb (  | \gamma | )}  v (t) \frac{D^\gamma  ( |v_{0}  |^{\alpha}) }{ \Jdtx (t, \cdot ) } \Bigr\| _{ L^2 }
 \le  \|  \langle \cdot \rangle ^n v(t) \| _{ L^\infty  }  \Bigl\| \langle \cdot \rangle ^{ -n + \MDb (  | \gamma | ) }  \frac{D^\gamma  ( |v_{0}  |^{\alpha}) }{ \Jdtx (t, \cdot ) } \Bigr\| _{ L^2 } .
\end{equation} 
Using~\eqref{fSupplN4}, we see that  
\begin{equation} \label{fSupplN7} 
\begin{split} 
\Bigl\| \langle \cdot  \rangle ^{-n + \MDb (  | \gamma  | )}  & \frac{D^\gamma  ( |v_{0}  |^{\alpha}) }{ \Jdtx (t, \cdot ) } \Bigr\| _{ L^2 }  \\ & \le  C K^\alpha  \min\{ 1, (bG(t))^{  1 - \frac {2\sigma } {2- N\alpha }} \}  \| \langle \cdot \rangle ^{ \eta }  D^\gamma  (  |v_0|^\alpha  )  \| _{ L^2 } , 
\end{split} 
\end{equation} 
where
\begin{equation*} 
\eta = -n + \MDb (  | \gamma  | ) + n\alpha (  1 - \frac {2\sigma } {2- N\alpha } )
\end{equation*} 
Applying~\eqref{fSupplN3:b3}, we see that
\begin{equation*} 
\eta \le    n\alpha (  1 - \frac {2\sigma } {2- N\alpha } ) < n\alpha - \frac {N} {2}.
\end{equation*} 
Moreover,
\begin{equation*} 
\eta < -n + \MDb (  | \gamma  | ) + n\alpha .
\end{equation*} 
Therefore, we deduce from~\eqref{elemu1:3} that $ \| \langle \cdot \rangle ^{ \eta } D^\gamma  (  |v|^\alpha   ) \| _{ L^2  } \le C K^{ \alpha  + 2  |\gamma  |} $. Applying~\eqref{fSupplN6} and~\eqref{fSupplN7}, we conclude that~\eqref{fSupplN8} holds.

\Step5 Development of $D^{\beta } ( |v |^{\alpha})  $.\quad 
We  differentiate~\eqref{eq-v-4}. Given
$1\leq \vert \beta \vert \leq J-2 $, the development of $D^{\beta
} ( |v |^{\alpha})  $ contains the term
\begin{equation}
 \CSTu _1 = \frac{D^{\beta} ( |v_{0} (x) |^{\alpha}) }{ \Jdtx (t,x) }, \label{A} 
\end{equation}
and  terms of the form 
\begin{equation} \label{B} 
 \CSTu _2= \frac { D^{\rho} ( |v_{0} (x) |^{\alpha} ) } {\Jdtx (t,x ) }   \prod_{j=1}^{p} \frac{ D^{\gamma_{j} } \Jdtx (t,x )  }{ \Jdtx (t, x )  } 
\end{equation}
where
\begin{equation} \label{B:b1} 
\gamma+\rho=\beta,\quad1\leq p\leq|\gamma|,\quad|\gamma_{j}|\geq1,\quad
\sum_{j=1}^{p}\gamma_{j}=\gamma.
\end{equation} 
The term $  \CSTu _1 $ is estimated by~\eqref{fSupplN5} and~\eqref{fSupplN8}. Since $(1-bt)^{- (   \vert \beta \vert
-1 )  \sigma}\ge 1$, the contribution of $  \CSTu _1 $ satisfies estimates~\eqref{Decay-v} and~\eqref{Decay-v2}.

It remains to estimate the terms of the form $\CSTu _2$.
In view of~\eqref{eq-v-4:b3}, 
\begin{equation} \label{eq-v-4:b4}
\begin{split} 
\frac {D^{\gamma_{j}} \Jdtx(t,x) } { \Jdtx(t,x) } & = \frac{2\alpha
|\Im\lambda|}{b G (t)  (  2-N\alpha )  } \frac { D^{\gamma_{j} } ( |v_{0} (x) |^{\alpha} ) } { \Jdtx(t,x) } +  \frac {D^{\gamma_{j}} f(t,x) } { \Jdtx(t,x) }
\\ & = :  \CSTt _1^j + \CSTt _2^j .
\end{split} 
\end{equation}  

\Step6 Estimates of $\CSTt _1^j$. \quad 
Using~\eqref{bis12}, we see that
\begin{equation*} 
 | \CSTt _1^j | \le \frac {  | D^{\gamma_{j} } ( |v_{0} (x) |^{\alpha} ) | } {  |v_0 (x)|^\alpha }. 
\end{equation*} 
Applying~\eqref{elemu1:1} with $\rho =\alpha $, we obtain
\begin{equation} \label{fEstNBd1}
 \| \CSTt _1^j  \| _{ L^\infty  }  \le C_0
\end{equation} 
for $ | \gamma  _j|\le 2m +2 $. 

\Step7 Estimates of $\CSTt _2^j$. \quad 
Using~\eqref{eq-v-4:b2} and~\eqref{eq-v-1:b1}, we write
\begin{equation*} 
f(t,x) =  \alpha \int_{0}^{t}|v(s,x)|^{-\alpha-2 }  |v_{0}(x)|^{\alpha}  \Im(\bar{v} (s,x) \Delta v (s, x) ) \,ds , 
\end{equation*} 
and we see that the development of $D^{\gamma_{j}}f(t,x)$ is estimated by terms of the form
\begin{equation*} 
\int _0^t  |D^{\gamma  _{ j1 }} (  | v_0 |^\alpha  ) |  \, {  | D^{ \gamma  _{ j2 }} \Delta v | } \,  {  | D^{ \gamma  _{ j3  }}  v | }  \,  | D^{ \gamma  _{ j4 } } (  |v|^{-\alpha -2} ) | ,
\end{equation*} 
where $\gamma  _{ j1 } + \gamma  _{ j2 } +\gamma  _{ j3 } + \gamma  _{ j4 } = \gamma _j$. 
Using~\eqref{fEstDerv1} with $\beta $ replaced by $\gamma  _{ j3 }$ and $\rho $ replaced by $-\alpha -2$, we are led to estimate terms of the form
\begin{equation*} 
\int _0^t  |v|^{- \alpha } |D^{\gamma  _{ j1 }}  (  | v_0 |^\alpha  ) | \frac {  | D^{ \gamma  _{ j2 }} \Delta v | } {  | v | }   \frac {  | D^{ \gamma  _{ j3  }}  v | } {  | v | } \prod  _{ \ell =1 } ^{ |\gamma  _{ j4 }| } \frac {  | D^{ \beta _\ell } v | } {  | v | } ,
\end{equation*} 
where $\sum _{ \ell =1 } ^{ |\gamma  _{ j4 }| } \beta _\ell = \gamma  _{ j4 } $. Note that we can incorporate the term $ \frac {  | D^{ \gamma  _{ j3  }}  v | } {  | v | } $ into the product. Moreover, $ \frac {  | D^{ \beta _\ell } v | } {  | v | } =1 $ if $ | \beta _\ell |=0$, so we can incorporate as many such terms as we like into the product. Thus, we need only estimate terms of the form 
\begin{equation} \label{fNWestmm1} 
\CSTd = \frac {1} {\Jdtx(t, x) } \int _0^t  |v|^{- \alpha } |D^{\gamma  _{ j1 }}  (  | v_0 |^\alpha  ) | \frac {  | D^{ \gamma  _{ j2 }} \Delta v | } {  | v | }  \prod  _{ \ell =1 } ^{ |\gamma  _{j 3} | } \frac {  | D^{ \beta _\ell } v | } {  | v | } ,
\end{equation} 
where 
\begin{equation} \label{fNWestmm2} 
\gamma  _{ j1 } + \gamma  _{ j2 } +\gamma  _{ j3 }  = \gamma _j, \quad \sum_{ \ell =1 } ^{ | \gamma _{j 3} |} \beta _\ell = \gamma  _{ j3 } .
\end{equation} 
To estimate the terms $ \CSTd $,  we consider separately two cases. 
If a term $\CSTd$ contains only derivatives $ | \gamma  _{ jk } | \le 2m$ for $k=1,2$ and $ |\beta _\ell |\le 2m$ for $\ell \le  |\gamma  _{ j3 }|$, then we call such a term $ \CSTd ^{\mathrm{low}}$ and we estimate its  $L^\infty $ norm. 
If a term $\CSTd$ contains a derivative of order $ | \gamma  _{ jk }| \ge  2m +1$ or $ |\beta _\ell | \ge 2m+1$, then we call such a term $ \CSTd ^{\mathrm{high}}$ and we estimate its  $L^2 $ norm. 

Recall that $\frac {1} {\Jdtx} \le 2$ by~\eqref{bis12}. Writing
\begin{equation} \label{fNWestmm3} 
 |v|^{- \alpha } |D^{\gamma  _{ j1 }}  (  | v_0 |^\alpha  ) | = ( \langle x\rangle ^n   |v|) ^{- \alpha } | \langle x\rangle ^{n\alpha } D^{\gamma  _{ j1 }}  (  | v_0 |^\alpha  ) | ,
\end{equation} 
we deduce from~\eqref{elemu1:2}, \eqref{PSI}, \eqref{fDfnPhi3}, \eqref{fDfnPhi6}, and~\eqref{de-sigma} (recall that $ |\gamma  _{ j2 }|+ 2 \le 2m+2$) that
\begin{equation*} 
\begin{split} 
 \| \CSTd ^{\mathrm{low}} \| _{ L^\infty  } & \le C_0 \int _0^t (1- bs)^{ - \frac {2-N\alpha } {2} - \sigma  ( |\gamma  _{ j2 }| + 4) - \sigma \sum_{ \ell=1 }^{  |\gamma  _{ j3 }| }  |\beta _\ell | } ds \\ & = C_0 \int _0^t (1- bs)^{ - \frac {2-N\alpha } {2} - ( |\gamma  _{ j2 }| +  |\gamma  _{ j3 }| +4)\sigma  } ds \\  & \le C_0 \int _0^t (1- bs)^{ - \frac {2-N\alpha } {2} - ( |\gamma  _j | +4)\sigma  } ds
  \\ & =  C_0 \int _0^t (1- bs)^{ - \frac {2-N\alpha } {2} - 5 \sigma - ( |\gamma  _j | - 1)\sigma  } ds . 
\end{split} 
\end{equation*} 
Since $  |\gamma  _j | - 1 \ge 0$ (by~\eqref{B:b1}) and $ - \frac {2-N\alpha } {2} - 5 \sigma >-1$ (by~\eqref{fSupplN3:b1}), we obtain
\begin{equation} \label{fEstBu1} 
\begin{split} 
 \| \CSTd ^{\mathrm{low}} \| _{ L^\infty  } & \le C_0 (1-bt )^{- ( |\gamma  _j | - 1)\sigma } \int _0^t (1- bs)^{ - \frac {2-N\alpha } {2} - 5 \sigma  } ds
 \\ &\le \frac {C_0} {b} (1-bt )^{ - ( |\gamma  _j | - 1)\sigma  } . 
\end{split} 
\end{equation} 

Suppose now that $ \max \{  |\gamma  _{ j 1 }|,  |\gamma  _{ j 2 }|,  |\beta _\ell|  \} \geq2m+1 $. Then, all the other derivatives
are of order $\leq2m $, see the proof of Lemma~\ref{elemu1}.
We consider two cases. If $ | \gamma  _{ j1 }| \ge  2m +1$, then we rename $ \CSTd  ^{\mathrm{high}}$ as $ \CSTd_1 ^{\mathrm{high}}$. If $ | \gamma  _{ j2 }| \ge  2m +1$ or if one of the $\beta _\ell $ has order $\ge 2m+1$, then we rename $ \CSTd ^{\mathrm{high}}$ as $ \CSTd_2 ^{\mathrm{high}}$.
If $ | \gamma _{ j1 } |\ge 2m+1$, then it follows from~\eqref{fSupplN8} that
\begin{equation*} 
 \Bigl\| \langle \cdot \rangle ^{ \MDb (  | \gamma _{ j1 } | )}  v (t) \frac{D^{\gamma  _{ j1 }} ( |v_{0}  |^{\alpha}) }{ \Jdtx (t, \cdot ) } \Bigr\| _{ L^2 } \le 
 C_0  \min\{ 1, (bG(t))^{  1 - \frac {2\sigma } {2- N\alpha }} \}  ,
\end{equation*}  
hence using~\eqref{PSI}, \eqref{fDfnPhi3}, \eqref{fDfnPhi6}, \eqref{de-sigma}, and $ \MDb (  | \gamma _1 | ) \ge  \MDb (  | \beta  | )$,
\begin{equation}  \label{fEstBu2} 
\begin{split} 
 \| \langle \cdot \rangle & ^{ - n\alpha   + \MDb (  | \beta  | ) } v(t) \CSTd_1 ^{\mathrm{high}} \| _{ L^2  } \\ & \le 
 C_0  \min\{ 1, (bG(t))^{  1 - \frac {2\sigma } {2- N\alpha }} \}  \int _0^t  (1- bs)^{ - \frac {2-N\alpha } {2} - ( |\gamma  _{ j2 }| +  |\gamma  _{ j3 }| +4)\sigma }
  \\ &\le C_0 \min\{ 1, (bG(t))^{  1 - \frac {2\sigma } {2- N\alpha }} \} \frac { (1-bt )^{ - ( |\gamma  _j | - 1)\sigma  }} {b} ,
\end{split} 
\end{equation} 
as in~\eqref{fEstBu1}. To estimate $ \CSTd_2 ^{\mathrm{high}}$, we suppose for instance $ | \gamma  _{ j2 }| \ge 2m+1$ (the case where one of the $\beta _\ell $ has order $\ge 2m+1$ is treated similarly). We use again $\frac {1} {\Jdtx}\le 2$, and we observe that by\eqref{com-v-t-s}, 
\begin{equation*} 
\begin{split} 
 | v(t) | \CSTd_2 ^{\mathrm{high}}   \le 
 2^{1 + \frac {1} {\alpha }}    \int _0^t  
  (  \langle x\rangle^{n} | v | ) ^{-\alpha}    |  \langle x\rangle  ^{\alpha n}  D^{\gamma_{j1}} ( |v_{0} |^{\alpha 
} )  | \,
| D^{ \gamma  _{ j2 }} \Delta v | \prod  _{ \ell =1 } ^{ |\gamma  _{j 3} | } \frac {  | D^{ \beta _\ell } v | } {  | v | }.
 \end{split} 
\end{equation*} 
Note also that 
 \begin{equation*} 
  \| \langle \cdot \rangle ^{   \MDb (  | \beta | +2 ) } D^{ \gamma  _{ j 2 }} \Delta v \| _{ L^2 } \le  \| \langle \cdot \rangle ^{   \MDb (  | \gamma  _{ j2 } | +2 ) } D^{ \gamma  _{ j2 }} \Delta v \| _{ L^2 } \le K_1 (1- bs) ^{\sigma ( | \gamma  _{ j2 } | +4)} 
 \end{equation*} 
  by~\eqref{PSI}, \eqref{fDfnPhi2},   and~\eqref{de-sigma}. 
Using~\eqref{PSI}, \eqref{fDfnPhi3}, \eqref{fDfnPhi6},  and~\eqref{de-sigma}, 
and since $\MDb(  |\beta | +2 ) \ge \MDb (  |\beta | ) - 2$  we obtain
\begin{equation}  \label{fEstBu3}
\begin{split} 
 \| \langle \cdot \rangle ^{   \MDb (  | \beta | ) -2 } v(t) \CSTd_2 ^{\mathrm{high}} \| _{ L^2  } & \le 
 C_0 \int _0^t  (1- bs)^{ - \frac {2-N\alpha } {2} - ( |\gamma  _{ j2 }| +  |\gamma  _{ j3 }| + 4 )\sigma }
  \\ &\le \frac {C_0} {b} (1-bt )^{ - ( |\gamma  _j | - 1)\sigma  } ,
\end{split} 
\end{equation} 
as in~\eqref{fEstBu1}. 

\Step8 Proof of~\eqref{Decay-v}. \quad As observed above (Step~5), we need only show that the terms of the form $ \CSTu _2$ (given by~\eqref{B}) satisfy~\eqref{Decay-v}. 
Since $ |\beta |\le 2m$,  all the derivatives $\rho $ and $\gamma _j$ also have order $\le 2m$. 
We use~\eqref{fSupplN5} for the term $\frac { D^{\rho} ( |v_{0} (x) |^{\alpha} ) } {\Jdtx (t,x ) } $. 
For the terms $\frac {D^{\gamma_{j}} \Jdtx(t,x) } { \Jdtx(t,x) }$, we use the decomposition~\eqref{eq-v-4:b4}.
We have $  \| \CSTt _1^j  \| _{ L^\infty  } \le C_0$ by~\eqref{fEstNBd1}. 
Moreover, $\frac {1} {\Jdtx (t,x)} \le 2$ by~\eqref{bis12}, so that estimate~\eqref{fEstBu1} yields
\begin{equation*} 
 \|\CSTt _2^j \| _{ L^\infty  } \le 2  \| \CSTd ^{\mathrm{low}} \| _{ L^\infty  } \le  \frac {C_0} {b} (1-bt )^{ - ( |\gamma  _j | - 1)\sigma  } \le  C_0  (1-bt )^{ - ( |\gamma  _j | - 1)\sigma  } . 
\end{equation*} 
Therefore, given $p_0\in \{ 1, \cdots , p \}$, we obtain
\begin{equation*} 
 \Bigl\|   \prod_{j=p_0}^{p}  \frac {D^{\gamma_{j}} \Jdtx(t,x) } { \Jdtx(t,x) } \Bigr\| _{ L^\infty  } \le 
 C_0    \prod_{j=p_0}^{p}    (1-bt )^{ - ( |\gamma  _j | - 1)\sigma  }  \le C_0 (1-bt)^{-  \sum _{ p_0 }^p  ( |\gamma _j| -1 )  \sigma}; 
\end{equation*} 
and so, 
\begin{equation} \label{fEstAdlow2} 
 \Bigl\|   \prod_{j=p_0}^{p}  \frac {D^{\gamma_{j}} \Jdtx(t,x) } { \Jdtx(t,x) } \Bigr\| _{ L^\infty  } \le C_0 (1-bt)^{- \sigma ( \sum _{ p_0 }^p  ( |\gamma _j|  ) +  \sigma} .
\end{equation} 
Therefore, using~\eqref{fSupplN5} for the term $\frac { D^{\rho} ( |v_{0} (x) |^{\alpha} ) } {\Jdtx (t,x ) } $ and~\eqref{fEstAdlow2} with $p_0=1$,
\begin{equation} \label{fEstAdlow} 
 \| \langle \cdot \rangle ^{\frac {2n\alpha \sigma } {2-N\alpha }}  \CSTu _2 \| _{ L^\infty  } \le C_0 \min \{ 1,  (   bG (t)  )^{  1 - \frac {2\sigma } {2- N\alpha }  }  \} (1-bt)^{- \sigma ( \sum _{ 1 }^p  ( |\gamma _j|  ) +  \sigma} .
\end{equation} 
since $ \sum _{ 1 }^p  |\gamma _j| =  |\gamma |\le  |\beta |$. 
This proves estimate~\eqref{Decay-v}. 

\Step9 Proof of~\eqref{Decay-v2}. \quad As observed above (Step~5), we need only show that the terms of the form $ \CSTu _2$ (given by~\eqref{B}) satisfy~\eqref{Decay-v2}. Let $2m+1 \le  |\beta | \leq J-2$. 

If all the derivatives $\rho $ and $\gamma _j$ have order $\le 2m$, then estimate~\eqref{fEstAdlow} holds. 
Now using~\eqref{PSI} and~\eqref{fDfnPhi1}, $  \| \langle \cdot \rangle ^n v(t) \| _{ L^\infty  }  \le K_1$, so that
\begin{equation*} 
\begin{split} 
\Vert \langle \cdot  \rangle ^ { \MDb (  | \beta | )} v(t) \CSTu _2 \Vert_{L^2 } &  \le C_0 \Vert \langle \cdot  \rangle ^ {-n + \MDb (  | \beta | )}  \CSTu _2 \Vert_{L^2 }  \le C_0 \Vert   \CSTu _2 \Vert_{L^2 } \\ & \le C_0  \| \langle \cdot \rangle ^{\frac {2n\alpha \sigma } {2-N\alpha }}  \CSTu _2 \| _{ L^\infty  } ,
\end{split} 
\end{equation*} 
by~\eqref{fSupplN3:b3}.  Thus we see that $ \CSTu _2 $ satisfies~\eqref{Decay-v2}.

Suppose now $ |\rho |\ge 2m+1$. It follows that $ |\gamma _j|\le 2m$ for all $j$, see the proof of Lemma~\ref{elemu1}.
In particular, estimate~\eqref{fEstAdlow2} holds. Using~\eqref{fSupplN8} for the term $\frac { D^{\rho} ( |v_{0} (x) |^{\alpha} ) } {\Jdtx (t,x ) } $,  we see that $ \CSTu _2 $ also satisfies~\eqref{Decay-v2} in this case.

Suppose next that one of the $\gamma _j$'s have order $\ge 2m+1$, for instance $ |\gamma _1| \ge 2m+1$. 
Assume also that $p\ge 2$, the case $p=1$ being simpler. 
It follows (see the proof of Lemma~\ref{elemu1}) that $ |\rho |\le 2m$ and $ |\gamma _j |\le 2m$ for $j\ge 2$. 
Using~\eqref{eq-v-4:b4}, we deduce that
\begin{equation} \label{eq-v-4:b5}
 \Bigl| \frac {D^{\gamma_{1}} \Jdtx(t,x) } { \Jdtx(t,x) } \Bigr|  \le C_0  \frac {  | D^{\gamma_{1} } ( |v_{0} (x) |^{\alpha} ) |} {b G (t) \Jdtx(t,x) } +  \frac {  | D^{\gamma_{1}} f(t,x) | } { \Jdtx(t,x) } .
\end{equation}  
Writing
\begin{equation*} 
 \CSTu _2= \frac { D^{\rho} ( |v_{0} (x) |^{\alpha} ) } {\Jdtx (t,x ) }  \frac{ D^{\gamma_{1} } \Jdtx (t,x )  }{ \Jdtx (t, x )  }  \prod_{j=2}^{p} \frac{ D^{\gamma_{j} } \Jdtx (t,x )  }{ \Jdtx (t, x )  } , 
\end{equation*}
we are led to estimate 
\begin{equation*} 
\CSTq_1 = \frac {1} {b G(t)} \frac {  D^{\rho} ( |v_{0} (x) |^{\alpha} ) } {\Jdtx (t,x ) }  \frac{  D^{\gamma_{1} }  ( |v_{0} (x) |^{\alpha} )  }{ \Jdtx (t, x )  }  \prod_{j=2}^{p} \frac{ D^{\gamma_{j} } \Jdtx (t,x )  }{ \Jdtx (t, x )  } ,
\end{equation*} 
and
\begin{equation*} 
\CSTq_2 = \frac { D^{\rho} ( |v_{0} (x) |^{\alpha} ) } {\Jdtx (t,x ) }  \frac{ D^{\gamma_{1} } f (t,x )  }{ \Jdtx (t, x )  }  \prod_{j=2}^{p} \frac{ D^{\gamma_{j} } \Jdtx (t,x )  }{ \Jdtx (t, x )  }.
\end{equation*} 
Since 
\begin{equation*} 
\frac {1} {bG(t) \Jdtx (t, x )} \le \frac {C_0} { | v_0 (x)|^\alpha },
\end{equation*} 
by~\eqref{bis12}, we deduce that
\begin{equation*} 
 |\CSTq_1| \le C_0 \frac {  | D^{\rho} ( |v_{0} (x) |^{\alpha} )| } {  |v_{0} (x) |^{\alpha} }  \frac{  | D^{\gamma_{1} }  ( |v_{0} (x) |^{\alpha} ) | }{ \Jdtx (t, x )  }  \prod_{j=2}^{p} \frac{  |D^{\gamma_{j} } \Jdtx (t,x ) | }{ \Jdtx (t, x )  } ,
\end{equation*} 
We use~\eqref{elemu1:1} for the first term, \eqref{fSupplN8} for the second term, and~\eqref{fEstAdlow2} with $p_0=2$ for the product, and we obtain
\begin{equation*} 
  \| \langle \cdot \rangle ^{ \MDb (  | \gamma_1 | )}  v (t) \CSTq_1  \| _{ L^2 } \le 
C_0  (1-bt)^{- (   \vert \beta \vert -1 )  \sigma}  \min\{ 1, (bG(t))^{  1 - \frac {2\sigma } {2- N\alpha }} \}  ,
\end{equation*} 
so that $\CSTq_1$ satisfies~\eqref{Decay-v2}.  (Recall that $ \MDb (  | \gamma_1 | ) \ge  \MDb (  | \beta | )$.)

We now estimate $\CSTq_2$.
With the notation~\eqref{fNWestmm1} (with $j=1$) we are led to estimate 
\begin{equation*} 
\CSTc = \frac { D^{\rho} ( |v_{0} (x) |^{\alpha} ) } {\Jdtx (t,x ) }  \CSTd  \prod_{j=2}^{p} \frac{ D^{\gamma_{j} } \Jdtx (t,x )  }{ \Jdtx (t, x )  } ,
\end{equation*} 
with $ \CSTd = \CSTd ^{\mathrm{low}}$, $ \CSTd _1 ^{\mathrm{high}}$ or $ \CSTd _2 ^{\mathrm{high}}$. 
In the cases $ \CSTd = \CSTd ^{\mathrm{low}}$ and $ \CSTd = \CSTd _2 ^{\mathrm{high}}$, 
we use~\eqref{fSupplN5} for the first term and~\eqref{fEstAdlow2} for the product, and we obtain
(since $\sum _{ 2 }^p  |\gamma _j|  =  |\gamma |-  |\gamma _1|\le  |\beta |-  |\gamma _1|$)
\begin{equation} \label{fEstAu1} 
\begin{split} 
\| \langle \cdot  \rangle ^ { \MDb (  | \beta | )} v(t) \CSTc \|_{L^2 } \le &
C_0 \min \{ 1,  (   bG (t)  )^{  1 - \frac {2\sigma } {2- N\alpha }  }  \} \\ & \times 
  (1-bt)^{-  ( |\beta |-  |\gamma _1| -1 )  \sigma} 
\| \langle \cdot  \rangle ^ { \MDb (  | \beta | ) - \frac {2n\alpha \sigma } {N-2\alpha }} v(t) \CSTd \|_{L^2 } . 
\end{split} 
\end{equation} 
Since $  \| \langle \cdot \rangle ^n v(t) \| _{ L^\infty  }  \le K_1$ by~\eqref{PSI} and~\eqref{fDfnPhi1}, we deduce from~\eqref{fEstBu1} and~\eqref{fSupplN3:b3}  that
\begin{equation} \label{fEstAu2} 
\begin{split} 
\| \langle \cdot  \rangle ^ { \MDb (  | \beta | ) - \frac {2n\alpha \sigma } {N-2\alpha }} v(t) \CSTd^{\mathrm{low}} \|_{L^2 } &  \le C_0 
\| \langle \cdot  \rangle ^ { \MDb (  | \beta | ) - \frac {2n\alpha \sigma } {N-2\alpha } - n}  \CSTd^{\mathrm{low}} \|_{L^2 } \\  &  \le C_0 \| \langle \cdot  \rangle ^ {- \frac {2n\alpha \sigma } {N-2\alpha }  }  \CSTd^{\mathrm{low}} \|_{L^2 }    \le C_0 \|  \CSTd^{\mathrm{low}} \|_{L^\infty  } \\ &\le \frac {C_0} {b} (1-bt )^{ - ( |\gamma  _1 | - 1)\sigma  } \\ &
\le C_0  (1-bt )^{ - ( |\gamma  _1 | - 1)\sigma  }.
\end{split} 
\end{equation} 
Estimates~\eqref{fEstAu1} and~\eqref{fEstAu2} show that $\CSTc$  satisfies~\eqref{Decay-v2} in the case $ \CSTd = \CSTd ^{\mathrm{low}}$. 
Next, we use~\eqref{fSupplN3:b3} and~\eqref{fEstBu3} to obtain
\begin{equation} \label{fEstAu4} 
\begin{split} 
\| \langle \cdot  \rangle ^ { \MDb (  | \beta | ) - \frac {2n\alpha \sigma } {N-2\alpha }} & v(t) \CSTd _2 ^{\mathrm{high}} \|_{L^2 } 
 \le \| \langle \cdot \rangle ^{   \MDb (  | \beta | ) -2 } v(t) \CSTd_2 ^{\mathrm{high}} \| _{ L^2  } \\ & \le \frac {C_0} {b} (1-bt )^{ - ( |\gamma  _j | - 1)\sigma  }   \le  C_0  (1-bt )^{ - ( |\gamma  _j | - 1)\sigma  } .
\end{split} 
\end{equation} 
Estimates~\eqref{fEstAu1} and~\eqref{fEstAu4} show that $\CSTc$  satisfies~\eqref{Decay-v2} in the case $ \CSTd = \CSTd _2 ^{\mathrm{high}}$. 
Finally, for $ \CSTd = \CSTd _1 ^{\mathrm{high}}$, we use $\frac {1} {\Jdtx}\le 2$ and we obtain
\begin{equation*} 
 |\CSTc | \le  2 | D^{\rho} ( |v_{0} (x) |^{\alpha} ) | \, | \CSTd _1 ^{\mathrm{high}}  | \,  \Bigl| \prod_{j=2}^{p} \frac{ D^{\gamma_{j} } \Jdtx (t,x )  }{ \Jdtx (t, x )  } \Bigr| .
\end{equation*}
We use~\eqref{elemu1:2} for the first term and~\eqref{fEstAdlow2} for the product, and we deduce that
\begin{equation*} 
\| \langle \cdot  \rangle ^ { \MDb (  | \beta | )} v(t) \CSTc \|_{L^2 } \le C_0  (1-bt)^{-  ( |\beta |-  |\gamma _1| -1 )  \sigma}  \| \langle \cdot \rangle ^{ - n\alpha   + \MDb (  | \beta  | ) } v(t) \CSTd_1 ^{\mathrm{high}} \| _{ L^2  } .
\end{equation*} 
Applying~\eqref{fEstBu2}, we conclude that $\CSTc$ also satisfies~\eqref{Decay-v2} in the case $ \CSTd = \CSTd _1 ^{\mathrm{high}}$. This completes the proof.
\end{proof}

\section{global existence for~\eqref{NLS-1}} \label{sGlobal} 
In this section, we apply Proposition~\ref{lem-Decay-v} and Corollary~\ref{cor-Decay-v} of Section~\ref{sAPriori} to prove the global existence of solutions to~\eqref{NLS-1} for sufficiently large $b$. 

The main result of this section is the following.

\begin{prop} \label{BD-GWP} 
Let $\lambda \in \C$ satisfy $\Im\lambda<0$. 
Assume~\eqref{Conditionalpha}, \eqref{fCondonk}--\eqref{fCondonm}, \eqref{fSupplN2:b1}  and let $\Spa $ be defined by~\eqref{n27}.
Let $b>0$, $K >1$, let $v_{0}\in \Spa $ satisfy~\eqref{IV:b1}, let $v\in C([0, \Tma ), \Spa )$ be the solution of~\eqref{NLS-1}
given by Proposition~$\ref{LWP}$, and let $\Psi $ be defined by~\eqref{1l}--\eqref{fDfnPhi8}.
Set
\begin{equation} \label{feContPsi:2} 
 \CSTnw = K + \CSTnv K^2
\end{equation} 
where $\CSTnv $ is the constant in~\eqref{feContPsi}.
It follows that there exists $b_{1} \ge b_0> 1$, where $b_0$ is given by Proposition~$\ref{lem-Decay-v}$, such that if 
$b\geq b_{1}$, then  $\Tma =\frac{1}%
{b}$,
\begin{equation} \label{UB}
\sup\limits_{0\leq T<\frac{1}{b}} \Psi_{T} \leq5 \CSTnw , 
\end{equation}
where $\Psi_{T}$ is defined by~\eqref{fDfnPhi8}, and
\begin{equation} \label{fEstFu} 
\sup _{ 0\le t< \Tma }\sup _{ x\in \R^N }  |f(t,x)| \le \frac {1} {4},
\end{equation} 
where $f(t,x)$ is defined by~\eqref{eq-v-4:b2}.
\end{prop}

For the proof of Proposition~\ref{BD-GWP}, we will use the following lemma.

\begin{lem} \label{eLemInt1} 
Given $b\ge 1$, $\nu > 0 $, $M\ge 1$ and $0< T^\ast \le \frac {1} {b}$, let 
\begin{equation}  \label{feLemInt1:1:b1} 
I_1 (t) = \int _0^t  \min \{ 1 , (bG(s))^{1 - \frac {2\sigma } {2-N\alpha }} \} (1-bs)^{- \frac {4-N\alpha } {2} - \nu + \sigma }   ds,
\end{equation} 
and 
\begin{equation}  \label{feLemInt1:1:b2} 
I_2 (t) = \int _0^t  \min \{ M , bG(s) \} (1-bs)^{- \frac {4-N\alpha } {2} - \nu }   ds,
\end{equation}
$0\le t< T^\ast$, where $G$ is defined by~\eqref{defG}. 
It follows that there exists a constant $C= C(  N, \alpha , \sigma )$ such that 
\begin{equation} \label{feLemInt1:1} 
I _1 (t) \le  \frac {C} {b ^{\frac {2\sigma } {2- N\alpha }}} (1-bt)^{ -\nu  },
\end{equation} 
and that
\begin{equation} \label{feLemInt1:1:bz1} 
I_2 (t) \le \Bigl(  \frac {C M } { b^{  \frac{2\sigma}{2-N\alpha}}  }   + \frac {2} {\nu }  \Bigr) (1-bt)^{- \nu },
\end{equation} 
for all $b\ge 1$, $M\ge 1$ and $0\le t< T^\ast $.
\end{lem} 

\begin{proof}
We first prove~\eqref{feLemInt1:1}.
For a given $b\ge 1,$ we consider
$T_b  \in (0, \frac {1} {b} ) $ defined by
\begin{equation} \label{feLemInt1:2}
bG(T_b)=1, \text{ i.e. }(1-bT_b)^{-\frac{2-N\alpha}{2}}=b+1,
\end{equation}
and study two cases separately: the case $0\leq t \le T_{b}$, $t< T^\ast$; and, if
$T_{b}<T^{\ast},$ the case $T_{b}\leq t < T^{\ast}$. 

Suppose first that
$0\leq t \le T_{b}$, $t< T^\ast$. Since  $\min \{  1, (  bG (s) )  ^{1-\frac{2\sigma}{2-N\alpha}} \} =1$ for $0\le s\le t$, we obtain
\begin{equation*} 
I_1 (t)  \le  \int _0^t (1-bs)^{-\frac {4-N\alpha}{2} - \nu + \sigma } \le  (1-bt)^{-  \nu} \int _0^t (1-bs)^{-\frac {4-N\alpha}{2} + \sigma }  .
\end{equation*} 
Note that $ - \frac {2-N\alpha } {2} + \sigma <0$ by the third inequality in~\eqref{fSupplN2:b1}, so that using~\eqref{feLemInt1:2} 
\begin{equation*} 
\begin{split} 
b  \Bigl( \frac {2-N\alpha } {2} - \sigma  \Bigr) \int _0^t (1-bs)^{-\frac {4-N\alpha}{2} + \sigma } & =  (1-bt) ^{- \frac {2-N\alpha } {2} + \sigma } -1  \le    (1-bt) ^{- \frac {2-N\alpha } {2} + \sigma }  \\
& \le   (1-bT_b ) ^{- \frac {2-N\alpha } {2} + \sigma }  =    (b +1 ) ^{1 - \frac {2\sigma  } {2-N\alpha } }  . 
\end{split} 
\end{equation*} 
Thus we see that~\eqref{feLemInt1:1} holds for $0\leq t\leq  T_{b}$, $t < T^\ast$.

Suppose now that $T_{b}<T^{\ast}$ and $T_{b}\leq t < T^{\ast}$.  
Given $s\in (t, T^{\ast})$, we have
\begin{equation*} 
\min\{1,(bG (s) )^{1-\frac{2\sigma}{2-N\alpha}}\} = (bG (s) )^{1-\frac{2\sigma}{2-N\alpha}} 
=  b ^{1-\frac{2\sigma}{2-N\alpha}} G (s) ^{1-\frac{2\sigma}{2-N\alpha}} .
\end{equation*} 
Moreover, by~\eqref{defG} and~\eqref{feLemInt1:2},
\begin{equation} \label{feLemInt1:4} 
G(s) \le \frac{(1-bs)^{\frac{2-N\alpha}{2}}  }{1-(1-bT_b)^{\frac{2-N\alpha}{2}}}  =  \frac{(1-bs)^{\frac{2-N\alpha}{2}}  } { 1
  - \frac {1} {b +1} }  \le 2 (1-bs)^{\frac{2-N\alpha}{2}};
\end{equation} 
and so
\begin{equation*} 
\min\{1,(bG (s) )^{1-\frac{2\sigma}{2-N\alpha}}\} \le 2 b ^{1-\frac{2\sigma}{2-N\alpha}}  (1-bs)^{ \frac {2-N\alpha } {2} -\sigma  }.
\end{equation*} 
Applying~\eqref{feLemInt1:1} with $T= T_b$, we deduce that
\begin{equation*} 
\begin{split} 
 I _1(t) & \le  \frac {C} {b ^{ \frac {2\sigma  } {2-N\alpha } }} (1-b T_b )^{-  \nu }+  2 b ^{1-\frac{2\sigma}{2-N\alpha}}    \int_{T_b}^{t} (1-bs)^{-1- \nu } \\ & \le  \frac {C} {b ^{ \frac {2\sigma  } {2-N\alpha } }} (1-b t )^{-  \nu }+  2 b ^{1-\frac{2\sigma}{2-N\alpha}}    \int_{0}^{t} (1-bs)^{-1- \nu } \\ & \le  \frac {C} { b ^{ \frac{2\sigma}{2-N\alpha}} } (1- bt) ^{- \nu }.
\end{split} 
\end{equation*} 
Thus~\eqref{feLemInt1:1} is satisfied for all $0\le t< T^\ast $.

We now prove~\eqref{feLemInt1:2}.
For a given $b\ge 1,$ we consider
$ \widetilde{T} _b  \in (0, \frac {1} {b} ) $ defined by
\begin{equation} \label{feLemInt1:12}
(1-b \widetilde{T} _b)^{-\frac{2-N\alpha}{2}}= 1+ b ^{ 1 - \frac{2\sigma}{2-N\alpha}},
\end{equation}
and study two cases separately: the case $0\leq t \le  \widetilde{T} _{b}$, $t< T^\ast$; and, if
$ \widetilde{T} _{b}<T^{\ast}$ the case $ \widetilde{T}_ {b}\leq t < T^{\ast}$.

Suppose first that
$0\leq t \le  \widetilde{T} _{b}$, $t< T^\ast$. Since  $\min \{  M, (  bG (s) )  ^{1-\frac{2\sigma}{2-N\alpha}} \} \le M$, we obtain
\begin{equation*} 
\begin{split} 
I_2 (t) & \le M \int _0^t (1-bs)^{-\frac {4-N\alpha}{2} - \nu  } \le  M (1-bt)^{-\nu } \int _0^t (1-bs)^{-\frac {4-N\alpha}{2}} \\
& \le M (1-bt)^{-\nu }  \frac {2} {(2-N\alpha ) b} [ (1-bt )^{- \frac {2-N\alpha } {2} }-1 ] ,
\end{split} 
\end{equation*} 
so that
\begin{equation}  \label{feLemInt1:5} 
I_2 (t)  \le  \frac {2M} { (2-N\alpha )  b^{  \frac{2\sigma}{2-N\alpha}}  } (1-bt)^{-\nu }.
\end{equation} 
Suppose now that $ \widetilde{T} _{b}<T^{\ast}$ and $ \widetilde{T} _{b}\leq t < T^{\ast}$.  
Given $s\in (t, T^{\ast})$, we have by~\eqref{defG} and~\eqref{feLemInt1:12},
\begin{equation*} 
 G(s) \le  \frac{(1-bs)^{\frac{2-N\alpha}{2}}  }{1-(1-b \widetilde{T} _b)^{\frac{2-N\alpha}{2}}}  =   \frac{(1-bs)^{\frac{2-N\alpha}{2}}  } { 1   - \frac {1} {1+ b ^{ 1 - \frac{2\sigma}{2-N\alpha}} } }  \le 2 (1-bs)^{\frac{2-N\alpha}{2}}.
\end{equation*} 
It follows that
\begin{equation*} 
\int  _{  \widetilde{T}_b  } ^t bG(s) (1-bs)^{- \frac {4-N\alpha } {2} - \nu }   \le 
2 b \int  _{  \widetilde{T}_b  } ^t  (1-bs)^{-1 - \nu } \le \frac {2} {\nu }  (1-bt)^{- \nu }.
\end{equation*} 
Applying~\eqref{feLemInt1:5} with $t=  \widetilde{T}_b $, we conclude that
\begin{equation*} 
I_2 (t) \le \Bigl(  \frac {2M} { (2-N\alpha )  b^{  \frac{2\sigma}{2-N\alpha}}  }   + \frac {2} {\nu }  \Bigr) (1-bt)^{- \nu }.
\end{equation*} 
This, together with~\eqref{feLemInt1:5}, proves~\eqref{feLemInt1:2}. 
\end{proof} 

\begin{proof}  [Proof of Proposition~$\ref{BD-GWP}$]
It follows from~\eqref{feContPsi:2} and Lemma~\ref{eContPsi} that $\Psi _0\le  \CSTnw $ and that $\Psi _t \le 2 \CSTnw $ for $t>0$ and small. We set
\begin{equation} \label{fClaimu1} 
T^{\ast}=\sup\{0\leq T<\Tma ;\Psi_{T}\leq5 \CSTnw \},
\end{equation} 
so that
\begin{equation*} 
0< T^{\ast} \le \Tma .
\end{equation*} 
We claim that there exists $b_{1}\geq b_{0},$ such that if $b\geq b_{1}$ then
\begin{equation} \label{T*=max}
T^{\ast}=\Tma . 
\end{equation}
Assuming the claim~\eqref{T*=max}, we complete the proof. 
First, if $\Tma <\frac{1}{b}$, then by~\eqref{fClaimu1}, \eqref{T*=max}  and Lemma~\ref{eContPsi}, we see that $\Psi_{\Tma }\leq5 \CSTnw $.
It follows that for all $0\le t< \Tma$
\begin{equation*}
  \Vert v(t)\Vert_\Spa + \Bigl(  \inf _{x\in \R^N }\langle x\rangle^{n}|v(t,x)| \Bigr)  ^{-1}\leq5 \CSTnw   (1-bt)^{-\sigma_{J}}  \leq5 \CSTnw (1-b\Tma )^{-\sigma_{J}}<\infty, 
\end{equation*}
which contradicts the blowup alternative~\eqref{blowup}. Thus $ T^{\ast}= \Tma =\frac{1}{b}$, and then estimate~\eqref{UB} follows from~\eqref{fClaimu1}.  
Moreover, applying~\eqref{bd-f} with $K_1= 5  \widetilde{K} $, we deduce that~\eqref{fEstFu} holds, by choosing $b_1$ possibly larger.

We prove the claim~\eqref{T*=max} by contradiction, so we assume that
\begin{equation*} 
T^{\ast}<\Tma .
\end{equation*} 
By the definition of $T^{\ast}$ and continuity of $\Psi_{T}$ in $T$, we see that 
\begin{equation}  \label{bd-T*}
\Psi_{T^{\ast}}=5 \CSTnw .
\end{equation}
In the rest of the proof, we apply Proposition~\ref{lem-Decay-v} and Corollary~\ref{cor-Decay-v} with $T= T^{\ast}$ and $K_1= 5  \widetilde{K} $, and in 
particular we assume $b\ge b_0$.  
For further reference, we note that by~\eqref{Decay-v0},
\begin{equation} \label{Decay-v0:b1}
\alpha  \| v(t) \| _{ L^\infty  }^{\alpha}\leq  \Bigl(  \frac{1 + \alpha  |\Im \lambda|}{  |\Im \lambda|} \Bigr)  \min \{ 2 K^\alpha  ,  bG (t) \} ,
\end{equation}
for all $0\le t \le T^\ast$. 
Moreover, we denote by $C_1, C_2, C_3,C_4, C_5 >0$ various constants depending possibly on $\beta, \alpha, N, K, K_{1}, \lambda$, etc, but not on $b$, $v$, $ T^{\ast} $, whose exact values are irrelevant and can change from line to line. 

We proceed in several steps.

\Step1 Control of $\Phi_{1,T^\ast}$.\quad 
It follows from~\eqref{Decay-vmu} that
\begin{equation} \label{fEstZOd} 
 \| \langle \cdot \rangle ^n v(t) \| _{ L^\infty  } \le 2   \| \langle \cdot \rangle ^n v_0 \| _{ L^\infty  } \le 2 \CSTnw .
\end{equation} 
Next, given $1\le  | \beta |\le 2m$, we apply~\eqref{2m-2} or~\eqref{2m}  with $f (  s )  =\lambda (1-bs)^{-\frac
{4-N\alpha}{2}} |v|^{\alpha} v $ to obtain
\begin{equation} \label{est-Phi1-lowreg}
 \vert \langle x \rangle^{n}D^{\beta}v \vert \leq \CSTnw  +
  I  +  \int_{0}^{t}  (1-bs)^{-\frac
{4-N\alpha}{2}} \Im  \Bigl( \lambda  \frac{   \langle x \rangle^{2n}  D^{\beta } (  |v|^\alpha v)  D^{\beta}\overline{v} }{ 
\langle x \rangle^{n} \vert D^{\beta}v \vert } \Bigr)ds ,
\end{equation} 
where
\begin{equation*} 
I= 
\begin{cases} 
\displaystyle \int_{0}^{t} 
\sup_{  |\gamma | \leq  |\beta |+2 } \Vert\langle\cdot\rangle^{n}D^{\gamma}v(s)\Vert_{L^{\infty}}ds,  &  |\beta |\le 2m-2,  \\
\displaystyle A \int_{0}^{t}  \sup_{  |\beta |+ 2 \le  |\gamma | \leq   |\beta | +k+2} \Vert\langle
\cdot\rangle^{n}D^{\gamma }v(s)\Vert_{L^{2}}ds,  & 2m-1\le  |\beta |\le 2m .
\end{cases} 
\end{equation*} 
It follows from~\eqref{bd-T*}, \eqref{fDfnPhi1}, \eqref{fDfnPhi2} and~\eqref{de-sigma} that
\begin{equation*} 
I \le 5 \CSTnw (1 + A) \int _0^t (1-bs) ^{- \sigma (  |\beta | + k + 4)} .
\end{equation*} 
Since $\sigma (k+4) < 1$ by the fourth inequality in~\eqref{fSupplN2:b1} and $\sigma  |\beta | = \sigma  _{  |\beta | }$, we obtain
\begin{equation} \label{est-Phi1-lowreg-1}
I   \le 5 \CSTnw (1 + A) (1-bt)^{- \sigma  _{  |\beta | }}  \int _0^{\frac {1} {b}} (1-bs) ^{- \sigma (  k + 4)}  \le \frac {C_1} {b} (1-bt)^{- \sigma  _{  |\beta | }}  .
\end{equation} 
To estimate the last term in~\eqref{est-Phi1-lowreg}, we write 
\begin{equation} \label{fMMre7} 
D^\beta ( |v|^\alpha v) =  |v |^\alpha D^\beta v + \sum_{\substack{\gamma_{1}+\gamma_{2}=\beta \\ |\gamma_{1} | \ge 1} }c_{\gamma_{1},\gamma_{2}} D^{\gamma_{1}} ( |v|^{\alpha} ) D^{\gamma_{2}}v
\end{equation} 
where the coefficients $c _{ \gamma _1, \gamma _2 }$ are given by Leibniz's rule. 
Since
\begin{equation} \label{fMMre6} 
\Im ( \lambda    |v|^\alpha D^{\beta }   v  D^{\beta}\overline{v}  ) = (\Im \lambda )    |v|^\alpha  |D^{\beta }   v|^2 \le 0
\end{equation} 
we see that
\begin{equation*} 
\Im  \Bigl( \lambda  \frac{   \langle x \rangle^{2n}   |v|^\alpha D^{\beta }   v  D^{\beta}\overline{v} }{ 
\langle x \rangle^{n} \vert D^{\beta}v \vert } \Bigr) \le 0 .
\end{equation*} 
Moreover, 
\begin{equation*} 
 \Bigl| \lambda  \frac{   \langle x \rangle^{2n}  D^{\gamma_{1}} ( |v|^{\alpha} ) D^{\gamma_{2}} v  D^{\beta}\overline{v} }{ 
\langle x \rangle^{n} \vert D^{\beta}v \vert } \Bigr| \le  |\lambda |   \langle x \rangle^n  | D^{\gamma_{1}} ( |v|^{\alpha} ) | \,  | D^{\gamma_{2}} v  | ,
\end{equation*} 
so that
\begin{equation} \label{fINTM1} 
 \Im  \Bigl( \lambda  \frac{   \langle x \rangle^{2n}  D^{\beta } (  |v|^\alpha v)  D^{\beta}\overline{v} }{ 
\langle x \rangle^{n} \vert D^{\beta}v \vert } \Bigr) \le C_1 \sum_{\substack{\gamma_{1}+\gamma_{2}=\beta \\ |\gamma_{1} | \ge 1} }  \| \langle \cdot \rangle ^n D^{\gamma_{1}} ( |v|^{\alpha} ) D^{\gamma_{2}}v \| _{ L^\infty  } .
\end{equation} 
Applying~\eqref{bd-T*}, \eqref{Decay-v} (with $\beta $ replaced by $\gamma _1$, recall that $ |\gamma _1|\ge 1$), \eqref{fDfnPhi1} and~\eqref{de-sigma}, we obtain
\begin{equation*}
 \| \langle \cdot \rangle ^n  D^{\gamma_{1}} ( |v|^{\alpha} ) D^{\gamma_{2}}v \| _{ L^\infty  }   \leq C_1   \min \{  1, (  bG ( s ) ) ^{1-\frac{2\sigma}{2-N\alpha}} \} (1-b s )^{- ( |\beta | -1)   \sigma } .
\end{equation*} 
Applying~\eqref{feLemInt1:1} of Lemma~\ref{eLemInt1} with $\nu =  |\beta | \sigma $, we deduce that
\begin{equation}  \label{fMMre2} 
 \int_{0}^{t}  (1-bs)^{-\frac {4-N\alpha}{2}}  \| \langle \cdot \rangle ^n  D^{\gamma_{1}} ( |v|^{\alpha} ) D^{\gamma_{2}}v \| _{ L^\infty  } \le  \frac {C_1} { b ^{ \frac{2\sigma}{2-N\alpha}} } (1- bt) ^{-  |\beta | \sigma   } .
\end{equation} 
Estimates~\eqref{fINTM1} and~\eqref{fMMre2} prove that
\begin{equation} \label{fMMre3} 
\int_{0}^{t}  (1-bs)^{-\frac
{4-N\alpha}{2}} \Im  \Bigl( \lambda  \frac{   \langle x \rangle^{2n}  D^{\beta } (  |v|^\alpha v)  D^{\beta}\overline{v} }{ 
\langle x \rangle^{n} \vert D^{\beta}v \vert } \Bigr)ds \le  \frac {C_1} { b ^{ \frac{2\sigma}{2-N\alpha}} } (1- bt) ^{-  |\beta | \sigma } ,
\end{equation} 
for all $0\le t\le T^\ast$. 
Since  $ |\beta |\sigma =\sigma  _{  |\beta | }$, it now follows from~\eqref{fEstZOd}, \eqref{est-Phi1-lowreg}, \eqref{est-Phi1-lowreg-1}, and~\eqref{fMMre3} that if $b_1\ge b_0$ is sufficiently large so that
\begin{equation} \label{fMMre4} 
C_1  \Bigl( \frac1b_1 +\frac{1}{ b_1 ^\frac{2\sigma}{2-N\alpha}}  \Bigr)< \CSTnw ,
\end{equation} 
then 
\begin{equation} \label{fMMre5} 
\Phi_{1,T^*}\leq 2 \CSTnw 
\end{equation} 
provided $b\ge b_1$. 

\Step2 Control of $(1-bt) ^{\sigma  _{  |\beta | }} \Vert\langle
x\rangle^{ \MDb (  |\beta | ) }D^{\beta}v\Vert_{L^{2}} $ for $2m+1\le  |\beta |\le J-2$. \quad 
Applying~\eqref{2m+1} with $\mu=n- \MDb (|\beta|)\in [0, n]$ (which corresponds to $\nu =  |\beta |-2m -1 \le k+1$ if $\MDb ( |\beta |) =n$ and $\nu= k+1 $ if $\MDb ( |\beta |) <n$) and $f (  s )  =\lambda(1-bs)^{-\frac
{4-N\alpha}{2}}|v|^{\alpha}v$ we obtain
\begin{equation} \label{dis45-1}
\Vert\langle\cdot\rangle^{ \MDb (  |\beta | ) }  D^{\beta}v\Vert_{L^{2}}\leq \CSTnw + C_2 \int_{0}^{t} 
 \Vert\langle\cdot\rangle^{ \MDb (  |\beta | )-1}\nabla D^{\beta }v\Vert_{L^{2}}ds  +  \widetilde{I} , 
\end{equation} 
where
\begin{equation*} 
 \widetilde{I} = \int_{0}^{t}  (1-bs)^{-\frac{4-N\alpha}{2}}   \frac{ \Im \int_{ \R^N } \lambda   \langle x
\rangle^{2 \MDb (  |\beta | ) } D^{\beta} ( | v |^\alpha v)  D^{\beta}\overline
{v}  dx}{\Vert\langle\cdot\rangle^{ \MDb (  |\beta | ) }D^{\beta} v\Vert_{L^{2}}} ds
\end{equation*} 
Note that $ \MDb (  |\beta | ) -1 \le  \MDb (  |\beta | +1 ) $. Therefore, it follows from~\eqref{bd-T*}, \eqref{fDfnPhi2} and~\eqref{de-sigma} that
\begin{equation*} 
 \int_{0}^{t} 
 \Vert\langle\cdot\rangle^{ \MDb (  |\beta | )-1}\nabla D^{\beta }v\Vert_{L^{2}}ds \le 5 \CSTnw \int_{0}^{t} (  1-bs )  ^{- (  |\beta|+3 )  \sigma}ds .
\end{equation*} 
Since $2\sigma <1$ by the fourth inequality in~\eqref{fSupplN2:b1}, we deduce that
\begin{equation} \label{fTBC1} 
\begin{split} 
 \int_{0}^{t} (  1-bs )  ^{- (  |\beta|+3 )  \sigma}ds & \le (  1-bt )  ^{- (  |\beta|+1 ) \sigma }  \int_{0}^{t} (  1-bs )  ^{- 2  \sigma}ds
\\ &  \le \frac {C_2} {b } (  1-bt )  ^{- (  |\beta|+1 ) \sigma }  \le \frac {C_2} {b } (  1-bt )  ^{-  \sigma _{  |\beta | } } .
\end{split} 
\end{equation} 
Thus we conclude that
\begin{equation}  \label{dis54:b1}
C_2 \int_{0}^{t} 
 \Vert\langle\cdot\rangle^{ \MDb (  |\beta | )-1}\nabla D^{\beta }v\Vert_{L^{2}}ds  \le \frac {C_2} {b } (  1-bt )  ^{-  \sigma _{  |\beta | } } .
\end{equation} 
To estimate $ \widetilde{I} $ we use, as in Step~1, \eqref{fMMre7} and~\eqref{fMMre6} and we deduce that
\begin{equation*} 
\begin{split} 
\Im \int_{ \R^N } \lambda   \langle x
\rangle^{2 \MDb (  |\beta | ) } & D^{\beta} ( | v |^\alpha v)   D^{\beta}\overline
{v}  dx \\ & \le C_2 \sum_{\substack{\gamma_{1}+\gamma_{2}=\beta \\ |\gamma_{1} | \ge 1} }  \| \langle \cdot \rangle ^{2 \MDb (  |\beta | )} D^{\gamma_{1}} ( |v|^{\alpha} ) D^{\gamma_{2}}v D^\beta  \overline{v}  \| _{ L^1  }  \\ & \le C_2 
 \| \langle \cdot \rangle ^{ \MDb (  |\beta | )}  D^\beta  v  \| _{ L^2  } \sum_{\substack{\gamma_{1}+\gamma_{2}=\beta \\ |\gamma_{1} | \ge 1} }  \| \langle \cdot \rangle ^{ \MDb (  |\beta | )} D^{\gamma_{1}} ( |v|^{\alpha} ) D^{\gamma_{2}}v  \| _{ L^2  } ;
\end{split} 
\end{equation*} 
and so, 
\begin{equation} \label{dis5}
\begin{split} 
  \frac{ \Im \int_{ \R^N } \lambda   \langle x 
\rangle^{2 \MDb (  |\beta | ) }  D^{\beta} ( | v |^\alpha v)  D^{\beta}\overline
{v}  dx}{\Vert\langle\cdot\rangle^{ \MDb (  |\beta | ) }D^{\beta} v\Vert_{L^{2}}}  & \\  \le C_2 \sum_{\substack{\gamma_{1}+\gamma_{2}=\beta \\ |\gamma_{1} | \ge 1} } & \| \langle \cdot \rangle ^{ \MDb (  |\beta | )} D^{\gamma_{1}} ( |v|^{\alpha} ) D^{\gamma_{2}}v  \| _{ L^2  } .
\end{split} 
\end{equation} 
To estimate the right-hand side of~\eqref{dis5}, we distinguish three cases: $ |\gamma _2| \ge 2m+1$; $ |\gamma _2| \le 2m$ and $ |\gamma _1|\ge 2m+1$; $ |\gamma _2| \le 2m$ and $ |\gamma _1|\le 2m$. 
If $ |\gamma _2| \ge 2m+1$, then $ |\gamma _1|\le 2m$. (Indeed, $ |\gamma _1| +  |\gamma _2| =  |\beta |\le J\le 4m +1$, by definition of $J$ and the first inequality in~\eqref{fCondonm}.) 
Since $\MDb ( |\beta |) \le \MDb ( |\gamma _2 |)$, we have
\begin{equation*} 
\| \langle \cdot \rangle ^{ \MDb (  |\beta | )} D^{\gamma_{1}} ( |v|^{\alpha} ) D^{\gamma_{2}}v  \| _{ L^2  }  \le  \|  D^{\gamma_{1}} ( |v|^{\alpha} )  \| _{ L^\infty  }  \| \langle \cdot \rangle ^{ \MDb (  |\gamma _2 | )} D^{\gamma_{2}}v  \| _{ L^2  }.
\end{equation*} 
It follows from~\eqref{Decay-v} (with $\beta $ replaced by $\gamma _1$, recall that $ |\gamma _1|\ge 1$),  
and from~\eqref{bd-T*}, \eqref{fDfnPhi2}, \eqref{de-sigma}, that
\begin{equation} \label{fEstSt2:1} 
 \| \langle \cdot \rangle  ^{ \MDb (  |\beta | )} D^{\gamma_{1}} ( |v|^{\alpha} ) D^{\gamma_{2}}v  \| _{ L^2  }  
 \le C_2  \min \{  1, (  bG ( s ) ) ^{1-\frac{2\sigma}{2-N\alpha}} \} (1-bs)^{-  ( | \beta  | +1 ) \sigma} .
\end{equation} 
(We use the fact that since $ |\gamma _2|\le J-2$, $\sigma  _{  |\gamma _2| } \le ( |\gamma _2|+2) \sigma $.)
If $ |\gamma _2| \le 2m$ and $ |\gamma _1|\ge 2m+1$, we deduce from~\eqref{Decay-v1} (with $\beta $ replaced by $\gamma _1$ and $\gamma $ by $\gamma _2$) and the property $\MDb ( |\beta |) \le \MDb ( |\gamma _1 |)$ that estimate~\eqref{fEstSt2:1} also holds. 
If $ |\gamma _1|,  |\gamma _2| \le 2m$, we use the fact that $\langle x \rangle ^{-\frac {2n\alpha \sigma } {2-N\alpha }} \in L^2 (\R^N )  $ by~\eqref{fSupplN3:b3}, and we estimate
\begin{equation*} 
 \begin{split} 
 \| \langle \cdot \rangle ^{ \MDb (  |\beta | )} D^{\gamma_{1}} ( |v|^{\alpha} ) D^{\gamma_{2}}v  \| _{ L^2  } & \le   \|   D^{\gamma_{1}} ( |v|^{\alpha} )   \| _{ L^2  }  \| \langle \cdot \rangle ^n D^{\gamma_{2}}v  \| _{ L^\infty  }
\\ &  \le C_2   \| \langle \cdot \rangle ^{\frac {2n\alpha \sigma } {2-N\alpha }}  D^{\gamma_{1}} ( |v|^{\alpha} )   \| _{ L^\infty }  \| \langle \cdot \rangle ^n D^{\gamma_{2}}v  \| _{ L^\infty  } .
 \end{split} 
\end{equation*} 
Using~\eqref{Decay-v}, and~\eqref{bd-T*}, \eqref{fDfnPhi1} and~\eqref{de-sigma} we see that in this case also estimate~\eqref{fEstSt2:1} holds. 
Applying~\eqref{fEstSt2:1} and~\eqref{feLemInt1:1} of Lemma~\ref{eLemInt1} with $\nu =  (|\beta | +2)  \sigma $, we deduce that
\begin{equation} \label{est-nonlinearity-2m+1-J-2}
 \widetilde{I} \le \frac {C_2} { b ^{ \frac{2\sigma}{2-N\alpha}} } (1- bt) ^{- ( |\beta | +2) \sigma   } .
\end{equation} 
Since  $ ( |\beta  | +2) \sigma \ge  \sigma  _{  |\beta | }$, it now follows from \eqref{dis45-1}, \eqref{dis54:b1} and~\eqref{est-nonlinearity-2m+1-J-2} that
$b_1\ge b_0$ is sufficiently large so that
\begin{equation} \label{fMMre4:b2} 
C_2  \Bigl( \frac1b_1 +\frac{1}{ b_1 ^\frac{2\sigma}{2-N\alpha}}  \Bigr)< \CSTnw ,
\end{equation} 
then 
\begin{equation} \label{fMMre5:b2} 
\sup  _{ 2m+1 \le  |\beta | \le J-2} \sup  _{ 0\le t\le T^\ast  }(1-bt) ^{\sigma  _{  |\beta | }} \Vert\langle
\cdot \rangle^{ \MDb (  |\beta | ) }D^{\beta}v\Vert_{L^{2}} \leq 2 \CSTnw 
\end{equation} 
provided $b\ge b_1$. 

\Step3 Control of $(1-bt) ^{\sigma  _{  |\beta | }} \Vert\langle
x\rangle^{ \MDb (  |\beta | ) }D^{\beta}v\Vert_{L^{2}} $ for $J-1 \le  |\beta |\le J$. \quad 
Applying~\eqref{2m+1} with $\mu=n- \MDb (|\beta|) = n-(J -  |\beta |)\in \{ n-1, n \}$, $\nu= k+1 $, and $f (  s )  =\lambda(1-bs)^{-\frac
{4-N\alpha}{2}}|v|^{\alpha}v$ we obtain
\begin{equation} \label{dis45}
\Vert\langle\cdot\rangle^{ \MDb (  |\beta | ) }  D^{\beta}v\Vert_{L^{2}}\leq \CSTnw + C_2 (J -  |\beta | ) \int_{0}^{t} 
 \Vert\langle\cdot\rangle^{ J-1 -  |\beta |} \nabla D^{\beta }v\Vert_{L^{2}}ds  +  \widetilde{I} , 
\end{equation} 
where
\begin{equation*} 
 \widetilde{I} = \int_{0}^{t}  (1-bs)^{-\frac{4-N\alpha}{2}}   \frac{ \Im \int_{ \R^N } \lambda   \langle x
\rangle^{2 \MDb (  |\beta | ) } D^{\beta} ( | v |^\alpha v)  D^{\beta}\overline
{v}  dx}{\Vert\langle\cdot\rangle^{ \MDb (  |\beta | ) }D^{\beta} v\Vert_{L^{2}}} ds
\end{equation*} 
The second term in the right-hand side of~\eqref{dis45} vanishes if $ |\beta |=J$, so we need only estimate it if $ |\beta |=J-1$. In this case, it follows from~\eqref{bd-T*}, \eqref{fDfnPhi2} and~\eqref{de-sigma}
\begin{equation*} 
\begin{split} 
\int _0^t  \Vert\langle\cdot\rangle^{ J-1 -  |\beta |} \nabla D^{\beta }v\Vert_{L^{2}} ds & = \int _0^t  \Vert \nabla D^{\beta }v\Vert_{L^{2}} \\ & \le 5 \CSTnw  \int_{0}^{t} (  1-bs )  ^{- (  |\beta|+5 )  \sigma}ds .
\end{split} 
\end{equation*} 
Since $2\sigma <1$ by the fourth inequality in~\eqref{fSupplN2:b1}, we see that (compare~\eqref{fTBC1})
\begin{equation*} 
 \int_{0}^{t} (  1-bs )  ^{- (  |\beta|+5 )  \sigma}ds \le \frac {C_3} {b} (1-bt)^{- ( |\beta |+3)\sigma }= \frac {C_3} {b} (1-bt)^{- \sigma _{  |\beta | } }. 
\end{equation*} 
Thus we obtain
\begin{equation} \label{dis54}
C_2 (J -  |\beta | ) \int_{0}^{t} 
 \Vert\langle\cdot\rangle^{ J-1 -  |\beta |} \nabla D^{\beta }v\Vert_{L^{2}}ds \le \frac {C_3} {b} (1-bt)^{- \sigma _{  |\beta | } },
\end{equation} 
for both $ |\beta |=J-1$ and $ |\beta |=J$. 

To estimate $ \widetilde{I} $ we use a more precise version of~\eqref{fMMre7} where we isolate the term corresponding to  $ \gamma _1= \beta $, i.e.
\begin{equation} \label{fLbzPrec1} 
D^\beta ( |v|^\alpha v) = |v |^\alpha D^\beta v  + D^\beta ( |v|^\alpha ) v + c_{\gamma_{1},\gamma_{2}}  \sum _{\substack{\gamma_{1}+\gamma_{2}=\beta \\ |\gamma_{1} |,  |\gamma _2| \ge 1} }    D^{\gamma_{1}} ( |v|^{\alpha} ) D^{\gamma_{2}}v .
\end{equation} 
Using~\eqref{fMMre6},  we obtain (compare~\eqref{dis5})
\begin{equation} \label{dis53}
\begin{split} 
  \frac{ \Im \int_{ \R^N } \lambda   \langle x 
\rangle^{2 \MDb (  |\beta | ) }  D^{\beta} ( | v |^\alpha v)  D^{\beta}\overline
{v}  dx}{\Vert\langle\cdot\rangle^{ \MDb (  |\beta | ) }D^{\beta} v\Vert_{L^{2}}}  &
\le  |\lambda |\,  \| \langle \cdot \rangle ^{ \MDb (  |\beta | )}  D^\beta ( |v|^\alpha ) v \| _{ L^2 }
 \\  + C_3 \sum_{\substack{\gamma_{1}+\gamma_{2}=\beta \\ |\gamma_{1} |,  |\gamma _2| \ge 1} }  & \| \langle \cdot \rangle ^{ \MDb (  |\beta | )} D^{\gamma_{1}} ( |v|^{\alpha} ) D^{\gamma_{2}}v  \| _{ L^2  } .
\end{split} 
\end{equation} 
For the first term on the right-hand side of~\eqref{dis53}, we use formula~\eqref{fEstDerv1} with $\rho  $ replaced by $\alpha $,  and we obtain
\begin{equation} \label{dis53:b1}
\begin{split} 
 |\lambda |\,  \| \langle \cdot \rangle ^{ \MDb (  |\beta | )} D^\beta ( |v|^\alpha ) v \| _{ L^2 }  \le &  |\lambda | \alpha  \| \langle \cdot \rangle ^{ \MDb (  |\beta | )} |v|^\alpha D^\beta v \| _{ L^2 } \\ & + C_3  \sup  _{  D  }  \Bigl\| \langle \cdot \rangle ^{ \MDb (  |\beta | )}  |v|^{\alpha +1} \prod _{ \ell =1 }^{  |\beta  |}  \frac {| D^{\beta _\ell } v |} { |v|}    \Bigr\| _{ L^2 },
\end{split} 
\end{equation} 
where $ {D} $ is the set of $( \beta _\ell ) _{ 1\le  \ell \le  | \beta  | }$ with $ 0 \le |\beta _\ell | \le  |\beta | -1 $ and $\sum_{ \ell=1 }^{ |\beta  |} \beta _\ell =\beta $. 
It follows from~\eqref{Decay-v0:b1} and~\eqref{bd-T*} that 
\begin{equation*} 
\begin{split} 
 |\lambda | \alpha   \| \langle \cdot \rangle ^{ \MDb (  |\beta | )}  |v|^\alpha D^\beta v \| _{ L^2 } &  \le  |\lambda | \alpha   \|v \| _{ L^\infty  }^\alpha \| \langle \cdot \rangle ^{ \MDb (  |\beta | )} D^\beta v\| _{ L^2 } \\ &
 \le 5  |\lambda | \CSTnw  \Bigl(  \frac{1 + \alpha  |\Im \lambda|}{  |\Im \lambda|} \Bigr)  \min \{ 2 K^\alpha  ,  bG (t) \}  (1-bs)^{-\sigma_{|\beta |}}  .
\end{split} 
\end{equation*} 
Applying~\eqref{feLemInt1:1:bz1} of Lemma~\ref{eLemInt1} with $\nu = \sigma  _{  |\beta |  }$ and $M= 2K^\alpha $, and then~\eqref{de-sigma:b1}, we deduce that
\begin{equation} \label{dis53:b2}
\begin{split} 
 \int_{0}^{t}  (1-bs) & ^{-\frac{4-N\alpha}{2}}  |\lambda | \alpha   \| \langle \cdot \rangle ^{ \MDb (  |\beta | )}   |v|^\alpha D^\beta v \| _{ L^2 }     ds \\ & \le 5  |\lambda | \CSTnw  \Bigl(  \frac{1 + \alpha  |\Im \lambda|}{  |\Im \lambda|} \Bigr)
 \Bigl(  \frac {2 C K^\alpha  } { b^{  \frac{2\sigma}{2-N\alpha}}  }   + \frac {2} {\sigma  _{  |\beta |  } }  \Bigr) (1-bt)^{- \nu }
  \\ & \le  \Bigl[  5  |\lambda | \CSTnw  \Bigl(  \frac{1 + \alpha  |\Im \lambda|}{  |\Im \lambda|} \Bigr)
 \Bigl(  \frac {2 C K^\alpha  } { b^{  \frac{2\sigma}{2-N\alpha}}  }    \Bigr)   +  \widetilde{K}  \Bigr] (1-bt)^{- \nu }
 \\ & \le  \Bigl(   \frac {C_3 } { b^{  \frac{2\sigma}{2-N\alpha}}  }   
 +  \widetilde{K}  \Bigr) (1-bt)^{- \nu } . 
\end{split} 
\end{equation} 
We next consider the last terms in the right-hand side of~\eqref{dis53}, 
and we consider separately the cases $ |\gamma _1|\le J-2$ and $J-1 \le |\gamma _1|\le J$. 
(Indeed, if $ |\gamma _1| \ge J-1$, then we may not apply estimate~\eqref{Decay-v1} with $\beta $ replaced by $\gamma _1$ as we do in Step~2.) 
If $ |\gamma _1|\le J-2$, then we may proceed as in Step~2 and obtain, similarly to~\eqref{fEstSt2:1} and~\eqref{est-nonlinearity-2m+1-J-2}, then using $( |\beta |+2) \sigma \le  \sigma  _{  |\beta | } - \sigma $
\begin{equation} \label{fEstSt2:1:b1} 
\begin{split} 
 \int_{0}^{t}  (1-bs)^{-\frac{4-N\alpha}{2}}    \| \langle \cdot \rangle  ^{ \MDb (  |\beta | )} D^{\gamma_{1}} ( |v|^{\alpha} ) D^{\gamma_{2}}v  \| _{ L^2  }  ds & 
 \le  \frac {C_3} { b ^{ \frac{2\sigma}{2-N\alpha}} } (1- bt) ^{- ( |\beta | +2) \sigma   } \\ & 
 \le  \frac {C_3} { b ^{ \frac{2\sigma}{2-N\alpha}} } (1- bt) ^{-  \sigma   _{  |\beta | } } .
\end{split} 
\end{equation} 
We now consider the case $J-1 \le |\gamma _1|\le J$. 
Applying formula~\eqref{fEstDerv1} with $\beta $ replaced by $\gamma _1$ and $\rho $ by $\alpha $, we obtain
\begin{equation*} 
\frac { | D^{ \gamma _1 } (  |v|^\alpha   ) | } { |v|^\alpha  } \le   C_3 \sup  _{   \widetilde{D }  } \prod _{ \ell =1 }^{  |\gamma _1 |}    \frac {| D^{\beta _\ell } v |} { |v|}  
\end{equation*} 
where $ \widetilde{D} $ is the set of $( \beta _\ell ) _{ 1\le  \ell \le  |\gamma _1 | }$ with $ 0 \le |\beta _\ell | \le  |\gamma _1| $) and $\sum_{ \ell=1 }^{ |\gamma _1 |} \beta _\ell =\gamma _1 $. Thus we see that we have to estimate terms of the form
\begin{equation*} 
\CSTu  = \Bigl\| \langle \cdot \rangle ^{ \MDb (  |\beta | )}  |v|^{\alpha}  |D^{\gamma _2} v  | \prod _{ \ell =1 }^{  |\gamma _1 |}    \frac {| D^{\beta _\ell } v |} { |v|} \Bigr\| _{ L^2 } ,
\end{equation*} 
where $\gamma _1 + \gamma _2 = \beta $, $ 0 \le |\beta _\ell | \le   |\beta |-1 $) and $\sum_{ \ell=1 }^{ |\gamma _1 |} \beta _\ell =\gamma _1 $. 
 If all $| \beta _\ell |\leq 2m$, then writing
\begin{equation*} 
\CSTu  \le  \| \,  |v|^\alpha \| _{ L^2  }  \| \langle \cdot \rangle ^n D^{\gamma _2} v\| _{ L^\infty  } \prod _{ \ell =1 }^{  |\gamma _1 |}    \Bigl\| \frac {| D^{\beta _\ell } v |} { |v|} \Bigr\| _{ L^\infty } ,
\end{equation*} 
we deduce from~\eqref{est-v-alpha-L2}, \eqref{bd-T*},  \eqref{fDfnPhi1}, \eqref{fDfnPhi6} and~\eqref{de-sigma}  that
\begin{equation*} 
\CSTu \le C_3 \min\{1, (bG (s) )^{1-\frac{2\sigma}{2-N\alpha}} \}(1-bs)^{-|\beta|\sigma} .
\end{equation*} 
Since $ |\beta |\ge J-1$, we have $|\beta|\sigma \le \sigma  _{  |\beta | }- \sigma $, so that
\begin{equation} \label{fEstdau} 
\CSTu \le C_3 \min\{1, (bG (s) )^{1-\frac{2\sigma}{2-N\alpha}} \}(1-bs)^{-\sigma  _{  |\beta | }+ \sigma } .
\end{equation} 
If one the $ |\beta _\ell |\ge 2m+1$, say $ |\beta _1 |\ge 2m+1$, then $ |\beta _\ell |\le 2m$ for $\ell \ge 2$.
Writing
\begin{equation*} 
\CSTu   \le  \| \,  |v|^\alpha \| _{ L^\infty }  \Bigl\| \frac {| D^{ \gamma _2 } v |} { |v|} \Bigr\| _{ L^\infty } 
 \| \langle \cdot \rangle ^{ \MDb (  |\beta _1| )} D^{\beta _1} v\| _{ L^2} \prod _{ \ell =2 }^{  |\gamma _1 |}    \Bigl\| \frac {| D^{\beta _\ell } v |} { |v|} \Bigr\| _{ L^\infty } ,
\end{equation*} 
(where we use $\MDb (  |\beta | ) \le \MDb (  |\beta _1| )$), we deduce from~\eqref{Decay-v0}, \eqref{bd-T*},  \eqref{fDfnPhi2}, \eqref{fDfnPhi6} and~\eqref{de-sigma}  that
\begin{equation*} 
\CSTu \le C_3 \min\{1,(bG (s) ) \} (1-bs)^{-  |\gamma _2|\sigma - \sigma  _{  |\beta _1| } - (\sum_{ \ell=2 }^{ |\gamma _1|}  |\beta _\ell | ) \sigma   } .
\end{equation*} 
Since $ \min\{1,(bG (s) ) \} \le \min\{1,(bG (s) )^{1-\frac{2\sigma}{2-N\alpha}} \} $
and $\sum_{ \ell=2 }^{ |\gamma _1|}  |\beta _\ell | =  |\gamma _1|-  |\beta _1|$, we see that
\begin{equation*} 
\CSTu \le C_3 \min\{1,(bG (s) ) \} (1-bs)^{-  |\beta |\sigma - \sigma  _{  |\beta _1| } +  |\beta _1| \sigma   } .
\end{equation*} 
If $ |\beta |=J$, then $\sigma  _{  |\beta | }= ( |\beta | +4) \sigma $, and $ |\beta _1|\le  |\beta |-1$, so that $\sigma  _{  |\beta _1| } -  |\beta _1| \sigma \le 3\sigma $; and so $-  |\beta |\sigma - \sigma  _{  |\beta _1| } +  |\beta _1| \sigma \ge -\sigma  _{  |\beta | } + \sigma $. 
If $ |\beta |=J-1$, then $\sigma  _{  |\beta | }= ( |\beta | +3) \sigma $, and $ |\beta _1|\le  |\beta |-1$, so that $\sigma  _{  |\beta _1| } -  |\beta _1| \sigma \le 2\sigma $; and so $-  |\beta |\sigma - \sigma  _{  |\beta _1| } +  |\beta _1| \sigma \ge -\sigma  _{  |\beta | } + \sigma $. 
In both cases, we see that $ \CSTu  $ satisfies the estimate~\eqref{fEstdau}. 
Applying~\eqref{fEstdau} and~\eqref{feLemInt1:1} of Lemma~\ref{eLemInt1} with $\nu =   \sigma  _{  |\beta | } $, we deduce that
\begin{equation} \label{est-B3}
 \int_{0}^{t}  (1-bs)^{-\frac{4-N\alpha}{2}} \CSTu \,   ds \le  \frac {C_3} { b ^{ \frac{2\sigma}{2-N\alpha}} } (1-bt)^{- \sigma  _{  |\beta | }}. 
\end{equation} 
Note also that the last terms in~\eqref{dis53:b1} are of the form $  \CSTu $ with $  |\gamma _2|= 0$, so they also satisfy~\eqref{est-B3}.  
Applying~\eqref{dis53}, \eqref{dis53:b1}, \eqref{dis53:b2}, \eqref{fEstSt2:1:b1}  and~\eqref{est-B3}, we deduce that
\begin{equation} \label{est-B3:b1}
 \widetilde{I} \le  \Bigl(  \CSTnw + 
  \frac {C_3} { b ^{ \frac{2\sigma}{2-N\alpha}} } \Bigr) (1-bt)^{- \sigma  _{  |\beta | }} . 
\end{equation} 
Now putting together~\eqref{dis45}, \eqref{dis54} and~\eqref{est-B3:b1}, we see that
\begin{equation*} 
(1-bt)^{\sigma_{|\beta|}} \Vert\langle\cdot\rangle^{ \MDb (  |\beta | ) }  D^\beta v \Vert_{L^2}\leq 2 \CSTnw + C_3   \Bigl( \frac {1} {b} +   \frac {1} { b ^{ \frac{2\sigma}{2-N\alpha}} }   \Bigr),
\end{equation*} 
for $J-1\leq|\beta|\leq J$. Therefore, if $b_1\ge b_0$ is sufficiently large so that
\begin{equation} \label{fMMre4:b3} 
C_3  \Bigl( \frac1b_1 +\frac{1}{ b_1 ^\frac{2\sigma}{2-N\alpha}}  \Bigr) \le  \CSTnw ,
\end{equation} 
then 
\begin{equation} \label{fMMre5:b3} 
\sup  _{ J- 1 \le  |\beta | \le J} \sup  _{ 0\le t\le T^\ast  }(1-bt) ^{\sigma  _{  |\beta | }} \Vert\langle
\cdot \rangle^{ \MDb (  |\beta | ) }D^{\beta}v\Vert_{L^{2}} \leq 3 \CSTnw ,
\end{equation} 
provided $b\ge b_1$. 

\Step4 Control of $\Phi_{3,T^\ast}$.\quad 
 It follows from~\eqref{eq-v-3} that
\begin{equation*} 
\begin{split} 
\frac{(1-bt)^{\frac{2-N\alpha}{2}}}{\inf _{ x\in \R^N  } (
\langle x\rangle^{\alpha n}|v(t,x)|^{\alpha} )  }  \leq & \Bigl(
\inf _{x\in \R^N  } (  \langle x\rangle^{\alpha n} 
|v_{0}(x)|^{\alpha} )  \Bigr)  ^{-1}+\frac{2\alpha|\Im\lambda|}{b ( 2-N\alpha )  }\\ &
+\alpha  (1-bt)^{\frac{2-N\alpha}{2}}  \int_{0}^{t} (\langle x\rangle^n |v |)^{-\alpha
} \frac{ \vert L(s,x)\vert}{|v|} ds.
\end{split} 
\end{equation*}  
On the other hand, it follows from~\eqref{bd-T*}, \eqref{fDfnPhi3}, \eqref{eq-v-1:b1} and~\eqref{fDfnPhi6} that
\begin{equation*} 
(\langle x\rangle^n |v |)^{-\alpha } \frac{ \vert L(s,x) \vert}{|v|}  \le \CSTnw ^{ \alpha +1 }  (1-bs)^{ - \frac{2-N\alpha}{2} - 2 \sigma } . 
\end{equation*}  
Since $ \frac{2-N\alpha}{2} + 2 \sigma < 1$ by~\eqref{fSupplN3:b1}, we deduce that
\begin{equation*} 
 \int_{0}^{t} (\langle x\rangle^n |v |)^{-\alpha
} \frac{ \vert L(s,x) \vert}{|v|} ds \le \frac {C_4} {b}. 
\end{equation*} 
Thus we see that
\begin{equation*} 
\frac{(1-bt)^{\frac{2-N\alpha}{2}}}{\inf _{ x\in \R^N  } (
\langle x\rangle^{\alpha n}|v(t,x)|^{\alpha} )  } \le \CSTnw^\alpha + \frac {C_4} {b} . 
\end{equation*} 
Therefore, if $b_1\ge b_0$ is sufficiently large so that
\begin{equation} \label{fCondC4} 
\frac {C_4} {b_1} \le (2^\alpha -1) \CSTnw^\alpha ,
\end{equation} 
then
\begin{equation} \label{4}
\Phi_{3,T^{\ast}}\leq2 \CSTnw ,
\end{equation}
provided $b\ge b_1$. 

\Step5 Control of $\Phi_{4,T^\ast}$.\quad 
The case $\beta =0$ is trivial, so we assume $ |\beta |\ge 1$.
Applying~\eqref{dis64} with
$f ( t )  =\lambda(1-bt)^{-\frac{4-N\alpha}{2}}|v|^{\alpha}v$ and
$n=0$,  and~\eqref{eq-v-1}, we obtain
\begin{equation} \label{dis76}
\begin{split} 
\frac {\partial } {\partial t}  \Bigl( \frac { | D^\beta v |} { |v|} \Bigr) = & |v|^{-1} \frac {\partial } {\partial t}  | D^\beta v | -  |v|^{-2} | D^\beta v | \frac {\partial } {\partial t}   |v| \\  = & - ( |v|\,  |D^\beta v |)^{-1} \Im ( D^\beta \Delta v D^\beta  \overline{v}  )  -  |v|^{-2} | D^\beta v | L  \\ & + (1-bt)^{-\frac{4-N\alpha}{2}}  ( |v|\,  |D^\beta v |)^{-1} \Im ( \lambda D^\beta ( |v|^\alpha v) D^\beta  \overline{v} ) \\ & -   (1-bt)^{-\frac{4-N\alpha}{2}}  |v|^{\alpha -1} | D^\beta v | \Im \lambda  .
\end{split} 
\end{equation} 
We first use a cancellation. Applying formula~\eqref{fLbzPrec1}, we see that  
\begin{equation*} 
\begin{split} 
( |v|\, & |D^\beta v |)^{-1}  \Im ( \lambda D^\beta ( |v|^\alpha v) D^\beta  \overline{v} ) -  |v|^{\alpha -1}  |D^\beta v| \Im \lambda  \\ = &  
( |v|\,  |D^\beta v |)^{-1} \Im  \Bigl( \lambda \Bigl[  D^\beta ( |v|^\alpha ) v + \lambda  c_{\gamma_{1},\gamma_{2}}  \sum _{\substack{\gamma_{1}+\gamma_{2}=\beta \\ |\gamma_{1} |,  |\gamma _2|  \ge 1} }   D^{\gamma_{1}} ( |v|^{\alpha} ) D^{\gamma_{2}}v \Bigr] D^\beta  \overline{v} \Bigr) ,
\end{split}  
\end{equation*} 
so that 
\begin{equation*} 
( |v|\,  |D^\beta v |)^{-1}  \Im ( \lambda D^\beta ( |v|^\alpha v) D^\beta  \overline{v} ) -  |v|^{\alpha -1} | D^\beta v | \Im \lambda   \le \CSTs ,
\end{equation*} 
where
\begin{equation} \label{fDfnG} 
\CSTs =   |\lambda | \,  | D^\beta ( |v|^\alpha ) |  + C_5  \sum _{\substack{\gamma_{1}+\gamma_{2}=\beta \\ |\gamma_{1} |,  |\gamma _2| \ge 1} }  \frac {  | D^{\gamma_{1}} ( |v|^{\alpha} )| \,  | D^{\gamma_{2}}v |} {  |v| } .
\end{equation} 
Therefore, \eqref{dis76} yields
\begin{equation*} 
\frac {\partial } {\partial t}  \Bigl( \frac { | D^\beta v |} { |v|} \Bigr) \le  \frac { | D^\beta \Delta v |} {  |v| } +  \frac { | D^\beta v | \,  |\Delta v |} {  |v|^2 }  +  (1-bt)  ^{-\frac{4-N\alpha}{2}}  \CSTs .
\end{equation*} 
Integrating in $t$ and taking the sup in $x$, we deduce that
\begin{equation} \label{fStpF1:b11} 
\begin{split} 
\Bigl\| \frac { D^\beta v } { |v|} \Bigr\|  _{ L^\infty  } \le  &\CSTnw + \int _0^t  \Bigl(   \Bigl\|  \frac { D^\beta \Delta v } {  |v| }  \Bigr\| _{ L^\infty  } +  \Bigl\|  \frac { D^\beta  v } {  |v| }  \Bigr\| _{ L^\infty  }  \Bigl\|  \frac {   \Delta v } {  |v| }  \Bigr\| _{ L^\infty  }  \Bigr) \\ &  + \int _0^t (1-bs)  ^{-\frac{4-N\alpha}{2}}  \| \CSTs \| _{ L^\infty  } .
\end{split} 
\end{equation} 
We first apply~\eqref{bd-T*}, \eqref{fDfnPhi6} and the property $2\sigma <1$ (by the fourth inequality in~\eqref{fSupplN2:b1}),  and we obtain
\begin{equation} \label{fEstLin1} 
\begin{split} 
\int _0^t  \Bigl\|  \frac { D^\beta  v } {  |v| }  \Bigr\| _{ L^\infty  }  \Bigl\|  \frac {   \Delta v } {  |v| }  \Bigr\| _{ L^\infty  } & \le C_5 \int _0^t (1-bs)^{- \sigma  _{  |\beta | } -2\sigma } \\ &
\le  C_5 (1-bt)^{- \sigma  _{  |\beta | } } \int _0^t (1-bs)^{- 2\sigma } \\ & \le \frac {C_5} {b}  (1-bt)^{- \sigma  _{  |\beta | } } .
\end{split} 
\end{equation} 
Similarly, if $ |\beta |\le 2m$, then (using also~\eqref{de-sigma})
\begin{equation*} 
\int _0^t   \Bigl\|  \frac { D^\beta \Delta v } {  |v| }  \Bigr\| _{ L^\infty  } \le C_5 \int _0^t (1-bs)^{- \sigma  _{  |\beta | } -2\sigma }
\le \frac {C_5} {b}  (1-bt)^{- \sigma  _{  |\beta | } } .
\end{equation*} 
For the case of $2m+1\leq|\beta|\leq2m+2$, we use Sobolev's embedding~\eqref{fEstSob1}: 
\begin{equation*} 
 \Bigl\|  \frac { D^\beta \Delta v } {  |v| }  \Bigr\| _{ L^\infty  }
 =  \Bigl\|  \frac {\langle \cdot \rangle ^n D^\beta \Delta v } { \langle \cdot \rangle ^n |v| }  \Bigr\| _{ L^\infty  }\le 
 C_5 (\inf \langle x \rangle ^n  |v|)^{-1} \sum_{  |\beta |+ 2 \le  |\gamma |\le  |\beta | +  k }  \| \langle \cdot \rangle^{n} D^{ \gamma} v \| _{ L^2  }.
\end{equation*} 
It follows from~\eqref{bd-T*}, \eqref{fDfnPhi2}, \eqref{fDfnPhi3} and~\eqref{de-sigma} that 
\begin{equation*} 
 \Bigl\|  \frac { D^\beta \Delta v } {  |v| }  \Bigr\| _{ L^\infty  } \le C_5 (1-bs) ^{- \frac {2-N\alpha } {2\alpha } - ( |\beta | + k +3  ) \sigma } .
\end{equation*} 
Since 
\begin{equation} \label{dis86} 
\frac{2-N\alpha}{2\alpha} + (k+2)\sigma < 1 
\end{equation} 
by the last inequality in~\eqref{fSupplN2:b1}, we deduce that 
\begin{equation*} 
\int _0^t    \Bigl\|  \frac { D^\beta \Delta v } {  |v| }  \Bigr\| _{ L^\infty  } \le  \frac {C_5} {b}  (1-bt)^{- (  |\beta | +1 ) \sigma  }  \le  \frac {C_5} {b}  (1-bt)^{- \sigma  _{  |\beta | } } .
\end{equation*} 
Thus we see that for every $ |\beta |\le 2m+2$,
\begin{equation}  \label{fEstLin2} 
\int _0^t    \Bigl\|  \frac { D^\beta \Delta v } {  |v| }  \Bigr\| _{ L^\infty  } \le  \frac {C_5} {b}  (1-bt)^{- \sigma  _{  |\beta | } } .
\end{equation} 
We now estimate the last term in~\eqref{fStpF1:b11}.  
We first assume $ |\beta |\le 2m$, and we note that
\begin{equation*} 
 |\CSTs |\le  C_5  \sum _{\substack{\gamma_{1}+\gamma_{2}=\beta \\ |\gamma_{1} | \ge 1} }  \frac {  | D^{\gamma_{1}} ( |v|^{\alpha} )| \,  | D^{\gamma_{2}}v |} {  |v| } .
\end{equation*} 
Since $ 1\le |\gamma _1|\le 2m$, we may apply~\eqref{Decay-v} with $\beta $ replaced by $\gamma _1$. 
Using also~\eqref{bd-T*} and~\eqref{fDfnPhi6}, we see that 
\begin{equation*} 
\begin{split} 
 \Bigl\| \frac {   D^{\gamma_{1}} ( |v|^{\alpha} )   D^{\gamma_{2}}v } {  |v| }  \Bigr\| _{ L^\infty  } & \le  \| D^{\gamma_{1}} ( |v|^{\alpha} ) \| _{ L^\infty  }  \Bigl\| \frac {   D^{\gamma_{2}}v } {  |v| }  \Bigr\| _{ L^\infty  } \\ &
\le C_5    \min \{  1, (  bG (t) ) ^{1-\frac{2\sigma}{2-N\alpha}} \} (1-bt)^{- (   \vert \gamma _1 \vert -1 )  \sigma -  |\gamma _2| \sigma  }  \\ &
\le C_5    \min \{  1, (  bG (t) ) ^{1-\frac{2\sigma}{2-N\alpha}} \} (1-bt)^{-  \sigma _{  |\beta | } + \sigma  } .
\end{split} 
\end{equation*} 
Applying~\eqref{feLemInt1:1} of Lemma~\ref{eLemInt1} with $\nu = \sigma  _{  |\beta |  }$, we deduce that
\begin{equation} \label{fStpF1:b12} 
\int _0^t (1-bs)  ^{-\frac{4-N\alpha}{2}}  \| \CSTs \| _{ L^\infty  }  \le  \frac {C_5} {b ^{\frac {2\sigma } {2- N\alpha }}} (1-bt)^{ -\sigma  _{  |\beta | }  }.
\end{equation} 
We now consider the case $2m+1\le  |\beta |\le 2m+2$.
Using~\eqref{fEstDerv1} with $\rho = \alpha $, 
\begin{equation*} 
 | D^\beta (  |v|^\alpha  ) |   \le  \alpha |v|^\alpha \frac {| D^{\beta  } v |} { |v|} +   C |v|^\alpha  \sup  _{ \substack {0\le  |\beta _\ell |\le  |\beta |-1 \\ \sum \beta _\ell = \beta  } } \prod _{ \ell =1 }^{  |\beta |}    \frac {| D^{\beta _\ell } v |} { |v|}  .
\end{equation*} 
Moreover, using~\eqref{fEstDerv1} with $\rho = \alpha $ and $\beta $ replaced by $\gamma _1$,
\begin{equation*} 
 | D^{\gamma _1} (  |v|^\alpha  ) |   \le   C |v|^\alpha  \sup  _{ \substack {0\le  |\beta _\ell |\le  |\gamma _1 | \\ \sum \beta _\ell = \gamma _1  } } \prod _{ \ell =1 }^{  |\beta |}    \frac {| D^{\beta _\ell } v |} { |v|}  .
\end{equation*} 
The two above inequalities imply
\begin{equation*} 
 | \CSTs |    \le   |\lambda |  \alpha |v|^\alpha \frac {| D^{\beta  } v |} { |v|}  + C_5  |v|^\alpha  \sum  _{ \substack {0\le  |\beta _\ell |\le  |\beta |-1 \\ \sum \beta _\ell = \beta  } } \prod _{ \ell =1 }^{  |\beta |}    \frac {| D^{\beta _\ell } v |} { |v|} ,
\end{equation*} 
so that
\begin{equation} \label{fStpF1} 
\int _0^t (1-bs)  ^{-\frac{4-N\alpha}{2}}  \| \CSTs \| _{ L^\infty  } \le   \widetilde{I}_1 +  \widetilde{I}_2 ,
\end{equation} 
where
\begin{equation} \label{fStpF1:b1} 
 \widetilde{I}_1 =    |\lambda |  \alpha  \int _0^t   (1-bs)  ^{-\frac{4-N\alpha}{2}}   |v|^\alpha   \frac {| D^{\beta  } v |} { |v|}  
\end{equation} 
and
\begin{equation} \label{fStpF1:b2} 
 \widetilde{I}_2 = C_5  \int _0^t 
  (1-bs)  ^{-\frac{4-N\alpha}{2}}   |v|^\alpha    \sum  _{ \substack {0\le  |\beta _\ell |\le  |\beta |-1 \\ \sum \beta _\ell = \beta  } } \prod _{ \ell =1 }^{  |\beta |}    \frac {| D^{\beta _\ell } v |} { |v|} .
\end{equation} 
We estimate $ \widetilde{I}_1 $ given by~\eqref{fStpF1:b1}. Using~\eqref{Decay-v0:b1}, \eqref{bd-T*}  and~\eqref{fDfnPhi6},
\begin{equation*} 
   |\lambda |  \alpha     |v|^\alpha   \frac {| D^{\beta  } v |} { |v|} \le 5  \widetilde{K}   |\lambda |  \Bigl(  \frac{1 + \alpha  |\Im \lambda|}{  |\Im \lambda|} \Bigr)  \min \{ 2 K^\alpha  ,  bG (t) \} (1-bs)^{- \sigma  _{  |\beta | }} .
\end{equation*} 
Applying~\eqref{feLemInt1:1:bz1} of Lemma~\ref{eLemInt1} with $\nu = \sigma  _{  |\beta |  }$ and $M= 2K^\alpha $, and then~\eqref{de-sigma:b1}, we deduce as in~\eqref{dis53:b2}  that
\begin{equation} \label{dis53:b12}
 \widetilde{I}_1    \le  \Bigl(   \frac {C_5 } { b^{  \frac{2\sigma}{2-N\alpha}}  }   
 +  \widetilde{K}  \Bigr) (1-bt)^{-  \sigma  _{  |\beta | } } . 
\end{equation}
Finally, we consider $ \widetilde{I}_2 $ given by~\eqref{fStpF1:b2}. 
We estimate $ |v|^\alpha $ by~\eqref{Decay-v0} and the terms $  \frac {| D^{\beta _\ell } v |} { |v|} $ by~\eqref{bd-T*} and~\eqref{fDfnPhi6}, and we obtain
\begin{equation*} 
 \Bigl\|   |v|^\alpha \prod _{ \ell =1 }^{  |\beta |}    \frac {| D^{\beta _\ell } v |} { |v|}   \Bigr\| _{ L^\infty  } \le 
 C_5  \min \{ 1 ,  bG (t) \}  (1-bt) ^{- \sum  _{ \ell =1}^{  |\beta | }\sigma  _{  |\beta _\ell| } }. 
\end{equation*} 
Since $\sum  _{ \ell =1}^{  |\beta | }  |\beta _\ell | =  |\beta |$, $ |\beta _\ell |\le  |\beta |-1$ and $2m+1\le  |\beta |\le 2m+2$, it follows from~\eqref{de-sigma} that 
\begin{equation*} 
\sum  _{ \ell =1}^{  |\beta | }\sigma  _{  |\beta _\ell| } \le \sigma  _{  |\beta | } - \sigma .
\end{equation*} 
Moreover, $   \min \{ 1 ,  bG (t) \} \le  \min \{ 1 , (bG(t))^{1 - \frac {2\sigma } {2-N\alpha }} \}$, so that
\begin{equation*} 
 \Bigl\|   |v|^\alpha \prod _{ \ell =1 }^{  |\beta |}    \frac {| D^{\beta _\ell } v |} { |v|}   \Bigr\| _{ L^\infty  } \le 
 C_5 \min \{ 1 , (bG(t))^{1 - \frac {2\sigma } {2-N\alpha }} \} (1-bt) ^{-  \sigma  _{  |\beta | }+ \sigma }. 
\end{equation*} 
Therefore, we deduce from estimate~\eqref{feLemInt1:1} of Lemma~\ref{eLemInt1} with $\nu = \sigma  _{  |\beta |  }$ that
\begin{equation} \label{dis53:b14}
 \widetilde{I}_2    \le   \frac {C_5 } { b^{  \frac{2\sigma}{2-N\alpha}}  }    (1-bt)^{-  \sigma  _{  |\beta | } } . 
\end{equation}
Estimates~\eqref{fStpF1:b12}, \eqref{fStpF1}, \eqref{dis53:b12} and~\eqref{dis53:b14} show that 
\begin{equation} \label{dis53:b15}
 \int _0^t (1-bs)  ^{-\frac{4-N\alpha}{2}}  \| \CSTs \| _{ L^\infty  } \le  \Bigl(   \frac {C_5 } { b^{  \frac{2\sigma}{2-N\alpha}}  }   
 +  \widetilde{K}  \Bigr) (1-bt)^{-  \sigma  _{  |\beta | } } ,
\end{equation}
for all $ |\beta |\le 2m+2$. 
Now we deduce from~\eqref{fStpF1:b11}, \eqref{fEstLin1}, \eqref{fEstLin2} and~\eqref{dis53:b15} that
 if $b_1\ge b_0$ is sufficiently large so that
\begin{equation} \label{dis53:b16}
C_5  \Bigl( \frac1b_1 +\frac{1}{ b_1 ^\frac{2\sigma}{2-N\alpha}}  \Bigr) \le  \CSTnw ,
\end{equation} 
then
\begin{equation} \label{6}
\Phi_{4,T^{\ast}}\leq2 \CSTnw ,
\end{equation}
provided $b\ge b_1$. 

\Step6 Conclusion.\quad 
We choose $b_1\ge b_0$ sufficiently large so that~\eqref{fMMre4}, \eqref{fMMre4:b2}, \eqref{fMMre4:b3}, \eqref{fCondC4} and~\eqref{dis53:b16} are satisfied.
It follows from~\eqref{fMMre5}, \eqref{fMMre5:b2}, \eqref{fMMre5:b3}, \eqref{4} and~\eqref{6} that if $b\ge b_1$, then 
\begin{equation*}
\Psi_{T^{\ast}}\le 3 \CSTnw .
\end{equation*}
This contradicts~\eqref{bd-T*} and completes the proof. 
\end{proof}

\section{Asymptotics for~\eqref{NLS-1} and proof of Theorem~$\ref{T1}$} \label{sAsymp} 

In this section, we establish time-asymptotic estimates for the solutions of~\eqref{NLS-1} given by Proposition~\ref{BD-GWP}, which we use to prove Theorem~\ref{T1}. The asymptotic estimates are given by the following result.

\begin{prop} \label{PropAsym}
Let $\lambda \in \C$ satisfy $\Im\lambda<0$. 
Assume~\eqref{Conditionalpha}, \eqref{fCondonk}--\eqref{fCondonm}, \eqref{fSupplN2:b1}  and let $\Spa $ be defined by~\eqref{n27}.
Let $v_{0}\in \Spa $ satisfy~\eqref{IV:b1} for some $K\ge 1$, let $b> b_1$ where $b_1$ is given by Proposition~$\ref{BD-GWP}$, and let $v\in
C([0,\frac{1}{b}),\Spa )$ be the solution of~\eqref{NLS-1}
given by Proposition~$\ref{BD-GWP}$. It follows that there exist $C, \delta >0$, $\eta > \frac {N} {2}$ and $f_{0},\omega_{0}\in
L^{\infty} (\R^N ) \cap C(\R^N ) $ with $\Vert f_{0}\Vert_{L^{\infty}}\leq\frac{1}{2}$ and $\langle
\cdot \rangle^{ n }\omega_{0}\in L^{\infty}( \R^N  )$ such that
\begin{equation} \label{asymptotic}
 \Vert \langle \cdot \rangle^{\eta} [  v(t,\cdot )-\omega_{0}(\cdot )\psi
(t,\cdot )e^{-i\theta(t,\cdot )} ]   \Vert _{L^{\infty}}\leq C(1-bt) ^{ - \frac{2-N\alpha}{2\alpha } -\delta } ,
\end{equation}
for all $0\leq t<\frac{1}{b}$, where
\begin{equation} \label{def-psi}
\psi(t,x)=\Bigl(  \frac{1+f_{0}(x)}{1+f_{0}(x)+\frac{2\alpha|\Im\lambda
|}{b  (  2-N\alpha  )  }|v_{0}(x)|^{\alpha}[(1-bt)^{-\frac{2-N\alpha }{2}}-1]} \Bigr)  ^{\frac{1}{\alpha}} 
\end{equation}
and
\begin{equation}
\theta(t,x)=\frac{\Re\lambda}{\Im\lambda}\log\psi(t,x). \label{def-theta}
\end{equation}
Moreover,
\begin{equation}
 \vert \omega_{0} \vert ^{\alpha}=\frac{ \vert v_{0} \vert ^{\alpha}}{1+f_{0}}. \label{omega0}
\end{equation}
In addition,
\begin{equation}  \label{converge}
(1-bt)^{-\frac{2-N\alpha}{2}}\Vert v\Vert_{L^{\infty}}^{\alpha} 
\goto_{t\uparrow \frac{1}{b}} \frac{b (2-N\alpha) } { 2 \alpha |\Im\lambda|}
\end{equation}
and there exist two constants $0<a\leq A<\infty$ such that
\begin{equation} \label{estimateL2}
a(1-bt)^{ (  \frac{1}{\alpha}-\frac{N}{2} )   (  1-\frac{N} 
{2n} )  }\leq \Vert v (  t )   \Vert _{L^{2}}\leq
A(1-bt)^{ (  \frac{1}{\alpha}-\frac{N}{2} )   (  1-\frac{N}
{2n} )  },
\end{equation}
for all $0 \leq t<\frac{1}{b}$. 
\end{prop}

\begin{proof}
Using the a priori estimates of Proposition~\ref{BD-GWP}, we follow the proof of~\cite[Proposition 4.1]{CH}. 

We recall that~\eqref{UB} and~\eqref{fEstFu} hold. 
Also, since $b_1\ge b_0$ of Proposition~\ref{lem-Decay-v}, $v$ satisfies~\eqref{eq-v-4} and~\eqref{Decay-v0}. In particular
\begin{equation} \label{Decay-v0:b11} 
 |v (t, x)| \le C \min \{ \langle x\rangle ^{-n},  (1-bt)^{\frac{2-N\alpha}{2\alpha }} \}, 
\end{equation} 
for all $x\in \R^N $ and $\frac {1} {2b}\le t< \frac {1} {b}$. 
We let $f \in C((0, \frac {1} {b}) \times \R^N )$ be defined by~\eqref{eq-v-4:b2}. It follows from~\eqref{UB}, \eqref{fEstFu}, \eqref{fEstintf} with $K_1$ replaced by $ 5\widetilde{K} $,  and~\eqref{fdonaod}  that  
\begin{equation*} 
 \| f(t, \cdot )- f(s, \cdot ) \| _{ L^\infty (\R^N ) } \le C \int _t^s (1-b\tau )^{- \frac {2-N\alpha } {2} - 2 \sigma }d \tau  \le C (1-bt)^{1 - \frac {2-N\alpha } {2} - 2 \sigma },
\end{equation*} 
for all $0\le t<s< \frac {1} {b}$. Thus we see that $f(t, \cdot )$ is convergent in $L^\infty(\R^N)$ as $t\uparrow\frac1b$. Then $f$ can be extended to a continuous function $[0,\frac1b] \to  L^\infty(\R^N)$. We set
\begin{equation}  \label{def-f0}
f_0(x)=f  \Bigl( \frac1b,x  \Bigr) =-\alpha\int^\frac1b_0|v_0(x)|^\alpha|v(s,x)|^{-\alpha-1}L(s,x)\,ds,
\end{equation}
so that $f_{0} \in L^{\infty} (\R^N ) \cap C(\R^N ) $ and
\begin{equation} \label{diff-f-f0}
\Vert f(t)-f_{0}\Vert_{L^{\infty}} \le C (1-bt)^{1-\frac{2-N\alpha} {2}-2\sigma} .
\end{equation}
Moreover, using~\eqref{fEstFu},
\begin{equation} \label{bd-f-2} 
\Vert f(t)\Vert_{L^{\infty}}\leq\frac{1}{4},
\end{equation} 
for all $0\leq t\leq\frac{1}{b}$, and so
\begin{equation} \label{bd-f0}
\Vert f_{0}\Vert_{L^{\infty}}\leq\frac{1}{4} .
\end{equation} 
In particular, $1+ f_0 >0$, so that by~\eqref{def-psi},
\begin{equation} \label{fEstPsiinfty} 
0 < \psi (t,x) \le 1
\end{equation}  
for all $0\le t< \frac {1} {b}$ and $x\in \R^N $.
Moreover, it follows from~\eqref{bd-f-2} and~\eqref{bd-f0} that
\begin{equation}\label{fUEHtl2} 
\max   \Bigl\{ \frac {1} {\Jdtx } , \frac {1} { \widetilde{\Jdtx }  } \Bigr\} \le 2 ,
\end{equation}
for all $0\leq t<\frac1b$, where $\Jdtx $ is defined by~\eqref{eq-v-4:b3} and
\begin{equation*} 
 \widetilde{\Jdtx } (t, x)  =  1+f_0(x)+\frac{ 2\alpha| \Im \lambda|}{b ( 2-N\alpha ) }|v_0(x)|^\alpha[(1-bt)^{-\frac{2-N\alpha}{2}}-1]  .
\end{equation*} 
Note also that by~\eqref{bd-f-2}, \eqref{bd-f0} and~\eqref{IV:b1},
\begin{equation*} 
\begin{split} 
\langle x\rangle ^{n\alpha } \min \{ \widetilde{\Jdtx },   \Jdtx  \}& \ge \frac{ 2\alpha| \Im \lambda|}{b ( 2-N\alpha ) } ( \langle x\rangle ^{n } |v_0(x)|)^\alpha[(1-bt)^{-\frac{2-N\alpha}{2}}-1] \\ &
\ge \frac{ 2\alpha| \Im \lambda|}{b ( 2-N\alpha ) K^\alpha  } [(1-bt)^{-\frac{2-N\alpha}{2}}-1] ,
\end{split} 
\end{equation*} 
so that 
\begin{equation} \label{fUEHtl1} 
\max   \Bigl\{ \frac {1} {\Jdtx } , \frac {1} { \widetilde{\Jdtx }  } \Bigr\} \le C (1-bt)^{ \frac {2-N\alpha } {2} } \langle x\rangle ^{ n\alpha },
\end{equation} 
for $\frac {1} {2b} \le t < \frac {1} {b}$. 
Moreover, it follows from~\eqref{diff-f-f0} that
\begin{equation} \label{fEstHtoH} 
 \Bigl\| \frac { \widetilde{\Jdtx } } {\Jdtx } -1 \Bigr\| _{ L^\infty  } =  \Bigl\| \frac { f(t) - f_0 } {\Jdtx }  \Bigr\| _{ L^\infty  } \le C .
\end{equation} 
We now set
\begin{equation}\label{def-tilde-v}
\widetilde{v}(t,x)=\Big( \frac{|v_0(x)|^\alpha} {  \widetilde{\Jdtx } (t,x)  }\Big)^\frac{1}{\alpha}.
\end{equation}
It follows from~\eqref{eq-v-4}, \eqref{def-tilde-v}, \eqref{diff-f-f0}, \eqref{fUEHtl2} and~\eqref{fUEHtl1}  that
\begin{equation*}
\| \, |v(t, \cdot )|^\alpha-\widetilde v(t, \cdot )^\alpha \|_{L^\infty} = 
 \Bigl\| ( \langle \cdot \rangle ^n  | v_0 |)^\alpha \frac { f(t)- f_0} { \langle \cdot \rangle ^{n\alpha } \Jdtx  \widetilde{\Jdtx } } \Bigr\| _{ L^\infty  }  \le  C (1-bt)^{1 -2\sigma } ,
\end{equation*} 
and
\begin{equation*}
\begin{split} 
\Bigl\| \frac { |v(t, \cdot )|^\alpha-\widetilde v(t, \cdot )^\alpha} {\langle \cdot \rangle ^{n\alpha }} \Bigr\|_{L^\infty} & = 
 \Bigl\| ( \langle \cdot \rangle ^n  | v_0 |)^\alpha \frac { f(t)- f_0} { \langle \cdot \rangle ^{2n\alpha } \Jdtx  \widetilde{\Jdtx } } \Bigr\| _{ L^\infty  } \\ & \le  C (1-bt)^{1 -2\sigma + \frac {2- N\alpha } {2}} ,
\end{split} 
\end{equation*} 
for all $\frac {2} {b}\leq t<\frac{1}{b}$. 
Therefore, given any $0\le \rho \le 1$, we have
\begin{equation} \label{diff-v-tilde-v}
\Bigl\| \frac { |v(t, \cdot )|^\alpha-\widetilde v(t, \cdot )^\alpha} {\langle \cdot \rangle ^{ \rho n\alpha }} \Bigr\|_{L^\infty}  \le  C (1-bt)^{1 -2\sigma  + \rho \frac {2- N\alpha } {2} } ,
\end{equation} 
for all $\frac {2} {b}\leq t<\frac{1}{b}$. 
Next, we introduce the decomposition
\begin{equation}\label{decom-v}
v(t,x)=\omega(t,x)\psi(t,x)e^{ - i\theta(t,x)},
\end{equation}
where $\psi(t,x)$ and $\theta(t,x)$ are defined by~\eqref{def-psi} and~\eqref{def-theta} respectively. Differentiating~\eqref{decom-v} with respect to $t$, we obtain
\begin{equation} \label{fEqomega} 
i \partial _t \omega = i \frac{e^{i\theta}} {\psi} \partial _t v - i \omega \frac { \partial _t \psi } {\psi} - \omega \partial _t \theta .
\end{equation}
On the other hand, it follows easily from~\eqref{def-psi},~\eqref{def-theta} and~\eqref{def-tilde-v} that
\begin{align*}
\frac{ \partial _t \psi }{\psi}&=\Im \lambda(1-bt)^{-\frac{4-N\alpha}{2}}\widetilde v^\alpha,\\
 \partial _t \theta &= \Re \lambda(1-bt)^{-\frac{4-N\alpha}{2}}\widetilde v^\alpha.
\end{align*}
Therefore, we deduce from~\eqref{fEqomega}, \eqref{decom-v} and~\eqref{NLS-1} that
\begin{equation} \label{eq-omega}
\begin{split}
i \partial _t \omega &= i \frac{e^{  i\theta}}{\psi} \partial _t v -\omega(1-bt)^{-\frac{4-N\alpha}{2}}\lambda\widetilde v^\alpha\\
         &=\frac{e^{  i\theta}}{\psi}( i \partial _t v-\lambda(1-bt)^{-\frac{4-N\alpha}{2}}\widetilde v^\alpha v)\\
         &=\frac{e^{  i \theta}}{\psi}( - \Delta v+\lambda(1-bt)^{-\frac{4-N\alpha}{2}}(|v|^\alpha-\widetilde v^\alpha)v) . 
\end{split}
\end{equation}
We claim that there exists $\eta>0$ such that 
\begin{gather} 
0\le \eta \le n, \label{fCndeta1} \\
\eta \ge n(1-\alpha ) , \label{fCndeta2} \\
\eta > \frac {N} {2}, \label{fCndeta3} \\
2\sigma + \frac{2-N\alpha}{2} \Bigl[ \frac {\eta + n\alpha } {n\alpha } \Bigr] < 1 . \label{fCndeta4} 
\end{gather} 
Indeed, if $\alpha \ge \frac {4} {5}$, we let $\eta = \frac {n} {5}$. Conditions~\eqref{fCndeta1} and~\eqref{fCndeta2} are obviously satisfied, and~\eqref{fCndeta3} follows from the first inequality in~\eqref{fCondonn:2}. 
Moreover, \eqref{fCndeta4} is equivalent to 
\begin{equation*} 
2\sigma < \frac {N} {2}  \Bigl( \alpha + \frac {1} {5} \Bigr) - \frac {1} {5\alpha }.
\end{equation*} 
Since $\alpha \ge \frac {4} {5}$ and $N\ge 1$, the right-hand side of the above inequality is $\ge \frac {1} {4}$. Since $\sigma < \frac {1} {8}$ by the fourth inequality in~\eqref{fSupplN2:b1} and~\eqref{fCondonk} we see that~\eqref{fCndeta4} is satisfied. 
If $\alpha < \frac {4} {5}$, we let $\eta = n (1-\alpha )$, so that~\eqref{fCndeta1} and~\eqref{fCndeta2} hold. 
Moreover, $\eta \ge \frac {n} {5} $ so that~\eqref{fCndeta3} follows from the first inequality in~\eqref{fCondonn:2}.  
Furthermore, \eqref{fCndeta4} is equivalent to $\sigma < \frac {(N+2) \alpha -2} {4\alpha }$, which is a consequence of the last inequality in~\eqref{fSupplN2:b1}. 

We fix $\eta >0$ satisfying~\eqref{fCndeta1}--\eqref{fCndeta4} and we let
\begin{equation}  \label{fCndeta5} 
\delta = 1 - \frac {2- N\alpha } {2\alpha } - 2\sigma - \frac{2-N\alpha}{2} \Bigl[ \frac {\eta - n(1-\alpha )} {n\alpha } \Bigr] >0 .
\end{equation} 
It follows from~\eqref{eq-omega}  that
\begin{equation}  \label{fCndeta6} 
\begin{split} 
\Vert \langle \cdot \rangle^{ \eta } & \partial _t \omega \Vert_{L^{\infty}} \\ & \leq \Bigl\Vert
\langle \cdot \rangle^{n} \frac{v} {\psi} \Bigr\Vert _{L^{\infty}} \Bigl(  \Bigl\Vert
\frac {\Delta v} { \langle \cdot \rangle ^{ n - \eta } \vert v \vert } \Bigr \Vert _{L^{\infty}} 
+|\lambda|(1-bt)^{-\frac{4-N\alpha}{2}}  \Bigl\Vert \frac { |v|^{\alpha}-\widetilde{v} 
^{\alpha}} { \langle \cdot \rangle ^{ n - \eta }} \Bigr\Vert_{L^{\infty}} \Bigr) . 
\end{split} 
\end{equation}
Applying~\eqref{IV:b1} \eqref{eq-v-4} \eqref{def-psi}, \eqref{bd-f0}  and~\eqref{fEstHtoH}, we obtain
\begin{equation}  \label{fCndeta7}
 \Bigl\| \langle \cdot \rangle ^n \frac {v} {\psi} \Bigr\| _{ L^\infty  }=  \Bigl\| \langle \cdot \rangle ^n v_0  \Bigl( \frac {1} {1+ f_0} \frac { \widetilde{\Jdtx } } {\Jdtx } \Bigr)^{\frac {1} {\alpha }} \Bigr\| _{ L^\infty  } \le C.
\end{equation} 
Moreover, by \eqref{UB} and~\eqref{fDfnPhi6}, 
\begin{equation}  \label{fCndeta8}
 \Bigl\Vert
\frac {\Delta v} { \langle \cdot \rangle ^{ n - \eta } \vert v \vert } \Bigr \Vert _{L^{\infty}}  \le  \Bigl\Vert \frac{\Delta v}{ \vert v \vert }\Bigr\Vert _{L^{\infty} }\leq 5  \widetilde{K}  (1-bt)^{-2\sigma}. 
\end{equation}
Furthermore, since $0\le n - \eta \le n \alpha $, it follows from~\eqref{diff-v-tilde-v}  with $\rho = \frac {n- \eta } {n\alpha }$ that
\begin{equation}  \label{fCndeta9}
 \Bigl\Vert \frac { |v|^{\alpha}-\widetilde{v} 
^{\alpha}} { \langle \cdot \rangle ^{ n - \eta }} \Bigr\Vert_{L^{\infty}}\le C (1-bt)^{1 -2\sigma  + \frac { ( n- \eta) ( 2- N\alpha )} {2 n\alpha} } .
\end{equation} 
We deduce from~\eqref{fCndeta6}--\eqref{fCndeta9} that
\begin{equation} \label{est-omega-t}
\begin{split} 
\Vert\langle \cdot \rangle^{\eta } \partial _t \omega\Vert_{L^{\infty}} & \leq C(1-bt)^{  -2\sigma  }
+ C(1-bt)^{  -2\sigma  -\frac{2-N\alpha}{2} [\frac {\eta - n(1-\alpha )} {n\alpha } ] } \\
& \leq C(1-bt)^{  -2\sigma  -\frac{2-N\alpha}{2} [ \frac {\eta - n(1-\alpha )} {n\alpha } ] }. 
\end{split} 
\end{equation}
Applying~\eqref{fCndeta4} and~\eqref{fCndeta5}, we obtain
\begin{equation*} 
\Vert\langle \cdot \rangle^{\eta }(\omega(t)-\omega(s))\Vert_{L^{\infty}}  \leq C(1-bt)^{ 1 -2\sigma  -\frac{2-N\alpha}{2} [ \frac {\eta - n(1-\alpha )} {n\alpha } ] }  = C(1-bt)^{ \frac {2-N\alpha } {2\alpha } + \delta } , 
\end{equation*} 
for all $\frac {1} {2b}\leq t<s<\frac{1}{b}$. We conclude that there exists $\omega_{0} \in  L^{\infty} (\R^N ) \cap C(\R^N ) $ such that $\langle
x\rangle^{\eta }\omega_{0}\in L^{\infty}( \R^N ) $ and
\begin{equation} \label{diff-omega}
\Vert\langle \cdot \rangle^{ \eta }(\omega(t)-\omega_{0})\Vert_{L^{\infty}}\leq C(1-bt)^{ \frac {2-N\alpha } {2\alpha } + \delta } , 
\end{equation}
for all $\frac {1} {2b}\leq t<\frac{1}{b}$. 
Moreover, \eqref{decom-v} and~\eqref{fCndeta7} imply that $ \| \langle \cdot \rangle ^n \omega (t) \| _{ L^\infty  } \le C$, so that $\langle \cdot \rangle ^n \omega _0\in L^\infty  (\R^N ) $.
Applying~\eqref{decom-v}, \eqref{fEstPsiinfty}  and~\eqref{diff-omega}, we obtain
\begin{equation*} 
 \| \langle \cdot \rangle ^\eta ( v(t)-\omega_0 \psi e^{-i\theta } ) \| _{ L^\infty  } =  \| \langle \cdot \rangle ^\eta \psi  ( \omega (t)-\omega_0 ) \| _{ L^\infty  } \le C(1-bt)^{ \frac {2-N\alpha } {2\alpha } + \delta },
\end{equation*} 
for all $\frac {1} {2b}\leq t<\frac{1}{b}$. 
This proves the asymptotic estimate~\eqref{asymptotic} for $\frac {1} {2b}\leq t<\frac{1}{b}$.
For $0\le t\le \frac {1} {2b}$ it is clearly satisfied by possibly choosing $C$ larger.

Now, we prove~\eqref{omega0}.
It follows from \eqref{diff-v-tilde-v} and \eqref{decom-v} that 
\begin{equation}
\Vert|\omega(t)|^{\alpha} \psi(t) ^{\alpha}-\widetilde{v}
(t)^{\alpha}\Vert_{L^{\infty}}\leq C(1-bt)^{1-2\sigma} .
\end{equation}
Using~\eqref{def-psi} and~\eqref{def-tilde-v}, we deduce that
\begin{equation*} 
\Bigl\Vert  \frac { |\omega (t)|^\alpha (1+f_0) -  |v_0|^\alpha } {  \widetilde{\Jdtx } }  \Bigr\Vert_{L^{\infty}}\leq C(1-bt)^{1-2\sigma} .
\end{equation*} 
Since by \eqref{IV:b1} and \eqref{bd-f0}
\begin{equation*} 
 \widetilde{\Jdtx }  \leq C(1-bt)^{-\frac{2-N\alpha}{2}},
\end{equation*} 
we conclude that
\begin{equation}
\Vert  |\omega (t)|^\alpha (1+f_0) -  |v_0|^\alpha   \Vert_{L^{\infty}}\leq C(1-bt)^{1 -\frac{2-N\alpha}{2} -2\sigma}
\goto _{ t \to \frac {1} {b} } 0.
\end{equation}
Applying~\eqref{diff-omega}, we obtain property~\eqref{omega0}.

Next we prove~\eqref{converge}. 
Set
\begin{equation} \label{fDfnZ1} 
Z (t, x)=  (1-bt)^{-\frac{2-N\alpha}{2}}   \widetilde{ v} (t,x) ^\alpha  .
\end{equation} 
It follows from~\eqref{def-tilde-v}   that
\begin{equation*}
Z (t, x) =\frac{|v_0 (x) |^\alpha(1-bt)^{-\frac{2-N\alpha}{2}}} {1+f_0 (x) +\frac{ 2 \alpha| \Im \lambda|}{b ( 2-N\alpha ) }|v_0 (x) |^\alpha[(1-bt)^{-\frac{2-N\alpha}{2}}-1]}.
\end{equation*}
Since $1+ f_0\ge 0$ by~\eqref{bd-f0}, we obtain
\begin{equation*}
Z(t,x)  \leq \frac{b (2-N\alpha ) }{ 2 \alpha| \Im \lambda|}\frac{1}{1-(1-bt)^\frac{2-N\alpha}{2}} 
\end{equation*}
so that
\begin{equation*} 
\limsup  _{ t\uparrow \frac {1} {b} }  \| Z(t) \| _{ L^\infty  }  \le \frac{b (2-N\alpha ) }{ 2 \alpha| \Im \lambda|} . 
\end{equation*} 
Moreover, $1+ f_0 \le 2$, so that
\begin{equation*}
Z (t,0)\geq\frac{|v_0(0)|^\alpha(1-bt)^{-\frac{2-N\alpha}{2}}}{2+\frac{ 2 \alpha| \Im \lambda|}{b (2-N\alpha ) }|v_0(0)|^\alpha[(1-bt)^{-\frac{2-N\alpha}{2}}-1]}
\end{equation*}
Since $|v_0(0)|>0$ by~\eqref{IV:b1}, we deduce that
\begin{equation*}
\liminf_{t\uparrow\frac1b}   \| Z(t) \| _{ L^\infty  }   \ge \frac{b (2-N\alpha ) }{ 2 \alpha| \Re \lambda|} .
\end{equation*}
Thus we see that 
\begin{equation} \label{fDfnZ2} 
 \| Z(t) \| _{ L^\infty  }   \goto  _{ t\to \frac {1} {b} }\frac{b (2-N\alpha ) }{ 2 \alpha| \Re \lambda|} .
\end{equation} 
Applying~\eqref{fDfnZ1} and~\eqref{diff-v-tilde-v}, we deduce that 
\begin{equation*} 
\begin{split} 
 \| (1-bt) ^{- \frac {2-N\alpha } {2}}  |v(t ) |^\alpha -Z(t )  \| _{ L^\infty  }& = (1-bt) ^{- \frac {2-N\alpha } {2}}  \| \,  |v(t ) |^\alpha -  \widetilde{v}(t)^\alpha   \| _{ L^\infty  }
\\ & \le C(1-bt) ^{1 - (2-N\alpha )  -2\sigma} \goto _{ t\to \frac {1} {b} }0,
\end{split} 
\end{equation*} 
so that~\eqref{converge} follows from~\eqref{fDfnZ2}.

Finally we prove~\eqref{estimateL2}, and we let $\frac {1} {2b}\le t< \frac {1} {b}$.
It follows from~\eqref{Decay-v0:b11} that
\begin{equation*} 
\begin{split} 
\int  _{ \R^N  } |v(t)|^2 & = \int  _{ \langle x\rangle  > (1-bt)^{- \frac {2-N\alpha } {2n\alpha }} } |v(t)|^2 +  \int  _{\langle x\rangle  < (1-bt)^{- \frac {2-N\alpha } {2n\alpha }} } |v(t)|^2 \\ 
& \le C \int  _{ \langle x\rangle > (1-bt)^{- \frac {2-N\alpha } {2n\alpha }} }  \langle x\rangle ^{- 2n}  + C  \int  _{ \langle x\rangle  < (1-bt)^{- \frac {2-N\alpha } {2n\alpha }} }(1-bt)^{\frac {2-N\alpha } {\alpha }}
\\ & \le C (1-bt)^{ \frac {2-N\alpha } {\alpha } (1 - \frac {N} {2n})},
\end{split} 
\end{equation*} 
which proves the upper estimate in~\eqref{estimateL2}.  Next, using~\eqref{bd-f0}, we see that
\begin{equation*} 
 \widetilde{\Jdtx }^{\frac {1} {\alpha }} \le
\begin{cases} 
C, & \langle x\rangle  > (1-bt)^{- \frac {2-N\alpha } {2n\alpha }}, \\
C \langle x\rangle ^{-n },  &  \langle x\rangle  < (1-bt)^{- \frac {2-N\alpha } {2n\alpha }} .
\end{cases} 
\end{equation*} 
Applying~\eqref{eq-v-4} and $ |v_0 (x)|\ge K^{-1}  \langle x\rangle ^{-n}$ (by~\eqref{IV:b1}), we deduce that
\begin{equation*} 
 |v(t, x)| \ge
\begin{cases} 
c  \langle x\rangle ^{-n}, & \langle x\rangle  > (1-bt)^{- \frac {2-N\alpha } {2n\alpha }}, \\
c,  &  \langle x\rangle  < (1-bt)^{- \frac {2-N\alpha } {2n\alpha }} ,
\end{cases} 
\end{equation*} 
for some $c>0$. 
Therefore,
\begin{equation*} 
\int  _{ \R^N  } |v(t)|^2 \ge c \int  _{ \langle x\rangle > (1-bt)^{- \frac {2-N\alpha } {2n\alpha }} }  \langle x\rangle ^{- 2n}  + c  \int  _{ \langle x\rangle  < (1-bt)^{- \frac {2-N\alpha } {2n\alpha }} }(1-bt)^{\frac {2-N\alpha } {\alpha }},
\end{equation*} 
from which the lower estimate in~\eqref{estimateL2} follows. 
This completes the proof.
\end{proof}

We are now in a position to prove Theorem~\ref{T1}.

\begin{proof}[Proof of Theorem~$\ref{T1}$]
Let $v_{0}  \in \Spa $ satisfy~\eqref{boundbelow}, and set
\begin{equation*}
K= \Vert v_0 \Vert_ \Spa + \Bigl(  \inf _{x\in \R^N  }\langle x\rangle^{n}|v_0 (x)| \Bigr)  ^{-1} .
\end{equation*}
Let $b\ge b_1$, where $b_1$ is given by Proposition \ref{BD-GWP} for this value of $K$, and let  $v\in
C([0,\frac{1}{b}),\Spa )$ be the solution of~\eqref{NLS-1} given by Proposition~\ref{BD-GWP}.
Let $f_0, \omega _0, \theta ,\psi $ be given by Proposition~\ref{PropAsym}. Set 
\begin{equation} \label{dis84} 
u(t,x)=(1+bt)^{-\frac N2}e^{i\frac{b|x|^2}{4(1+bt)}}v\Big( \frac t{1+bt},\frac x{1+bt}\Big),\  t\geq0,\ x\in\R^N .
\end{equation}
It follows that  $u\in C([0,\infty), H^1(\R^N))$ is the solution of~\eqref{NLS-0}  with the initial value $u_0(x)=e^{i\frac{b|x|^2}{4}}v_0(x)$. 
Since $\eta >\frac{N}{2}$, the estimate~\eqref{asymptotic} implies
\begin{equation*} 
 \Vert v(t,x)-\omega_{0}(x)\psi(t,x)e^{-i\theta(t,x)} \Vert
_{L^{\infty}\cap L^{2}}\leq C(1-bt)^{\frac{2-N\alpha}{2\alpha}+\delta} .
\end{equation*} 
Then, \eqref{dis81} follows from~\eqref{asymptotic}. By using~\eqref{dis84},
we see that~\eqref{dis82} and~\eqref{dis83} are consequences of~\eqref{converge} and~\eqref{estimateL2}, respectively. This completes the proof.
\end{proof}

\end{document}